\documentclass[12pt]{article}
\usepackage[margin=1in]{geometry}
\usepackage[utf8]{inputenc}

\usepackage{natbib}
\setcitestyle{authoryear,open={(},close={)}}

\usepackage[dvipsnames]{xcolor}

\usepackage{amsmath, amssymb, amsthm,mathtools,dsfont}
\usepackage{bbm}
\usepackage{natbib}
\usepackage{blkarray}
\usepackage{cancel}
\usepackage{enumitem}
\usepackage{euscript}
\usepackage{float}
\usepackage{bm}
\usepackage{graphicx}
\usepackage{mathrsfs}
\usepackage{microtype}
\usepackage{ragged2e}
\usepackage{sectsty}
\usepackage{thmtools}
\usepackage{tikz}
\usepackage[Symbolsmallscale]{upgreek}

\usepackage{microtype}
\usepackage{graphicx}
\usepackage{subfigure}
\usepackage{hyperref}

\hypersetup{citecolor=red, colorlinks=true, linkcolor=blue}

\usepackage{bm}
\newcommand{\dotp}[1]{\langle #1 \rangle}

\newcommand{\cA}{\mathcal{A}}

\newcommand{\cC}{\mathcal{C}}

\newcommand{\cF}{\mathcal{F}}
\newcommand{\cG}{\mathcal{G}}
\newcommand{\cH}{\mathcal{H}}

\newcommand{\cK}{\mathcal{K}}

\newcommand{\cM}{\mathcal{M}}
\newcommand{\cN}{\mathcal{N}}

\newcommand{\cP}{\mathcal{P}}
\newcommand{\cQ}{\mathcal{Q}}

\newcommand{\cT}{\mathcal{T}}
\newcommand{\cU}{\mathcal{U}}
\newcommand{\cV}{\mathcal{V}}

\newcommand{\cX}{\mathcal{X}}

\newcommand{\NN}{\mathbb{N}}

\newcommand{\PP}{\mathbb{P}}

\newcommand{\RR}{\mathbb{R}}
\renewcommand{\S}{\mathbb{S}}

\newcommand{\ZZ}{\mathbb{Z}}

\newcommand*{\op}[1]{\|#1\|_{\mathrm{op}}}

\newcommand*{\triplenorm}[1]{{\left\vert\kern-0.25ex\left\vert\kern-0.25ex\left\vert #1
    \right\vert\kern-0.25ex\right\vert\kern-0.25ex\right\vert}}

\DeclareMathOperator{\id}{id}
\DeclareMathOperator{\supp}{supp}

\newcommand{\R}{\mathbb{R}}
\newcommand{\Rd}{\mathbb{R}^d}
\renewcommand{\phi}{\varphi}
\newcommand{\eps}{\varepsilon}
\newcommand{\sse}{\subseteq}

\newcommand*{\E}{\mathbb E}

\newcommand*{\pran}[1]{\left(#1\right)}

\newcommand*{\ep}{\varepsilon}
\newcommand*{\defeq}{\coloneqq}
\newcommand*{\eqdef}{\eqqcolon}

\newcommand*{\rd}{\mathrm{d}}
\newcommand*{\dd}{\, \rd}
\DeclareMathOperator*{\argmin}{argmin}
\DeclareMathOperator*{\argmax}{argmax}

\newcommand{\Faff}{\cF_{\mathrm{aff}}}
\newcommand{\ones}{\mathbf{1}}
\newcommand{\dom}{\mathrm{dom}}

\newcommand{\Var}{\mathrm{Var}}

\usepackage{fancyhdr}

\usepackage[capitalize,noabbrev]{cleveref}

\declaretheorem[name=Corollary]{cor}

\declaretheorem[name=Lemma]{lem}
\declaretheorem[name=Proposition]{prop}
\declaretheorem[name=Remark, style=remark]{rmk}
\declaretheorem[name=Theorem]{thm}

\renewcommand{\phi}{\varphi}

\usepackage{ulem}
\newcommand{\vertiii}[1]{{\left\vert\kern-0.25ex\left\vert\kern-0.25ex\left\vert #1 
     \right\vert\kern-0.25ex\right\vert\kern-0.25ex\right\vert}}
\renewcommand{\AA}{\mathbb{A}}

\title{Optimal transport map estimation in general function spaces}
\author{Vincent Divol\thanks{CEREMADE, Université Paris-Dauphine - PSL. \tt vincent.divol@psl.eu} \ \ Jonathan Niles-Weed\thanks{Courant Institute of Mathematical Sciences and Center for Data Science, New York University. \tt jnw@cims.nyu.edu} \ \  Aram-Alexandre Pooladian\thanks{Center for Data Science, New York University. \tt aram-alexandre.pooladian@nyu.edu}}

\begin{document}
\maketitle

\begin{abstract}
	We study the problem of estimating a function $T$ given independent samples from a distribution $P$ and from the pushforward distribution $T_\sharp P$.
	This setting is motivated by applications in the sciences, where $T$ represents the evolution of a physical system over time, and in machine learning, where, for example, $T$ may represent a transformation learned by a deep neural network trained for a generative modeling task. 
To ensure identifiability, we assume that $T = \nabla \phi_0$ is the gradient of a convex function, in which case $T$ is known as an \emph{optimal transport map}.
Prior work has studied the estimation of $T$ under the assumption that it lies in a H\"older class, but general theory is lacking.
We present a unified methodology for obtaining rates of estimation of optimal transport maps in general function spaces.
Our assumptions are significantly weaker than those appearing in the literature: we require only that the source measure $P$ satisfy a Poincar\'e inequality and that the optimal map be the gradient of a smooth convex function that lies in a space whose metric entropy can be controlled.
As a special case, we recover known estimation rates for H{\"o}lder transport maps, but also obtain nearly sharp results in many settings not covered by prior work.
For example, we provide the first statistical rates of estimation when $P$ is the normal distribution and the transport map is given by an infinite-width shallow neural network.
\end{abstract}

\section{Introduction}
A central question in statistics is to estimate the relationship between two data sets.
For instance, in the regression problem, the statistician observes data $(X_1, Y_1), \dots, (X_n, Y_n)$ satisfying
\begin{equation*}
	Y_i = f^\star(X_i) + \ep_i\,,
\end{equation*}
where $f^\star$ denotes an unknown regression function and where $\ep_i$ denotes mean-zero noise.
There are a number of ad hoc strategies for estimating $f^\star$ under different structural assumptions, see, e.g., \citet{Tsy09}; however, non-parametric least squares gives a general recipe:
The statistician faced with a regression problem can choose a class $\cF$ of candidate functions and construct a corresponding estimator $\hat f$ which, under weak conditions, satisfies an oracle inequality of the form
\begin{equation}\label{eq:regression}
	\E \|\hat f - f^\star\|_{L^2} \lesssim \inf_{f \in \cF} \|f - f^\star\|_{L^2} + \delta_{n, \cF}\,.
\end{equation}
The first term in~\eqref{eq:regression} represents the error made in approximating $f^\star$ by an element of $\cF$.
The second term is the estimation error, and scales with the complexity of the function class $\cF$, measured, for example, by its Gaussian width or metric entropy~\cite[Chapter 13]{wainwright2019high}.
The statistician has the freedom to choose the class $\cF$ to balance these sources of error.

We may view the regression problem as an instance of the following general setting: data $X_1, \dots, X_n \sim P$ undergoes an unknown transformation by some map\footnote{More generally, as in the case of regression, $T$ may be a ``randomized map,'' i.e., a Markov kernel.} $T$ to yield new data $T(X_1), \dots, T(X_n)$; the statistician's job is to estimate the transformation $T$.
When the statistician observes \emph{paired} data $\{(X_i, T(X_i))\}_{i=1}^n$, we recover the non-parametric regression problem.
However, more limited observational models have recently attracted interest as well.
For example, in a model that has been called ``permuted,'' ``unlinked,'' or ``uncoupled'' regression~\citep{degroot1980estimation,pananjady2016linear,rigollet2019uncoupled,balabdaoui2021unlinked,slawski2022permuted}, the statistician observes unpaired data $\{X_i\}_{i=1}^n, \{T(X_i)\}_{i=1}^n$, without knowing the correspondence between the datasets.
This setting is obviously more difficult than the non-parametric regression problem, and existing statistical results are more limited.
This line of work investigates conditions under which $T$ is identifiable and proposes rate-optimal estimators in several settings, but falls short of offering a general estimation methodology with guarantees comparable to~\eqref{eq:regression}.

In this paper, we study a challenging variant of this problem, in which the statistician has access to i.i.d.\ data $X_1, \dots X_n$ from a source distribution $P$, and independently, i.i.d.\ samples $Y_1, \dots, Y_n$ from the distribution $T_\sharp P$,\footnote{Recall that $T_\sharp P$ is the law of $T(X)$ for $X \sim P$.} for an unknown transformation $T$. 
As we detail below, this setting is motivated by applications in computational biology in which it is impossible to measure the same set of units before and after the application of the transformation $T$.
Our goal is to develop estimators for $T$, and study how optimal rates of estimation depend on the structural assumptions placed on $T$.
In short, we seek an analogue for~\eqref{eq:regression}, along with a general methodology for bounding the estimation error in terms of the complexity of our candidate class of transformations. 

To ensure identifiability and following existing work on the permuted regression model, we impose the general assumption that $T = \nabla \phi_0$, where $\phi_0$ is some convex function, which we call a \emph{Brenier potential}. Assuming $T$ is the gradient of a convex function makes a connection with the field of optimal transport.
Indeed, gradients of convex functions automatically possess a rigid optimality property: if $Q\defeq (\nabla \phi_0)_\sharp P$, then $\nabla \phi_0$ solves the \emph{Monge problem}
\begin{align*}
	\nabla \phi_0 \defeq \argmin_{T:T_\sharp P = Q} \int \|x - T(x)\|^2 \dd P(x)\,.
\end{align*}
In other words, $\nabla \phi_0$ automatically minimizes the average displacement cost among any map that transforms $P$ into $Q$.
For this reason, we call maps of the form $T = \nabla \phi_0$ optimal transport maps, or Monge maps.

From the perspective of modeling physical systems, the notion of considering maps that minimize movement is a natural one. As a motivating setup, we turn to applications in computational biology, which have received much interest over the last few years \citep{schiebinger2019optimal,dai2018autoencoder,moriel2021novosparc,Demetci2021.SCOTv2,bunne2022proximal}. In this framework, practitioners measure genomic profiles of large groups of stem cells evolving over time. The measurement pipeline requires lysing (i.e., destroying) the cells, so that no individual cell can be tracked as it evolves.
Nevertheless, between consecutive time steps, a viable assumption is that the genetic profiles evolve nearly optimally in the sense of Monge.
The data at times $t$ and $t+1$ therefore can be modeled as consisting of samples $X_1, \dots, X_n \sim P_t$ and independent samples $T(X_1'), \dots T(X_n') \sim T_\sharp P_t =: P_{t+1}$, for some Monge map $T$.
 In these sorts of applications, understanding the statistical efficiency of proposed estimators is of the utmost importance, as data is often scarce.

Beyond computational biology, the estimation of optimal transport maps is an increasingly relevant task in machine learning~\citep{WassersteinGAN,2017-Genevay-AutoDiff,grathwohl2018ffjord,salimans2018improving,finlay2020learning,huang2021convex}, computer graphics \citep{SolGoePey15,SolPeyKim16,feydy2017optimal}, and economics \citep{carlier2016vector,chernozhukov2017monge,gunsilius2021matching,torous2021optimal}. These developments in the applied sciences have been accompanied by several recent works in the area of \textit{statistical} estimation of optimal transport maps, introduced in \cite{hutter2021minimax}, and followed by \cite{deb2021rates,manole2021plugin,muzellec2021near,pooladian2021entropic}. 
These works all consider the case where $\phi_0$ lies in the class of $s$-times differentiable, $\beta$-smooth, $\alpha$-strongly convex functions, with $P$ having a density bounded away from $0$ and $\infty$ on a bounded convex domain $\Omega$. \citet{hutter2021minimax} develop an estimator $\hat{\phi}_{\cF_W}$ which achieves the following estimation rate:
\begin{align*}
    \E\|\nabla\hat{\phi}_{\cF_W} - \nabla \phi_0\|^2_{L^2(P)} \lesssim  n^{-\frac{2(s-1)}{2s  + d-4}} \log^2(n) \,.
\end{align*}
They also show this rate is minimax optimal up to logarithmic factors.
Writing the empirical source and target measures as $P_n \defeq n^{-1}\sum_{i=1}^n\delta_{X_i}$ and $Q_n \defeq n^{-1}\sum_{i=1}^n\delta_{Y_i}$, the estimator analyzed by \cite{hutter2021minimax} is 
\begin{align}\label{eq: semidual_n}
    \hat{\phi}_{\cF_W} \in \argmin_{\phi \in {\cF_W}} S_n(\phi) \defeq \int \phi \dd P_n + \int \phi^* \dd Q_n\,,
\end{align}
where $\cF_W$ is a particular family of convex functions with bounded wavelet expansions.
The form of this estimator suggests the possibility of a more general estimation procedure, in which the class $\cF_W$ is replaced by a different class of functions $\cF$.
For example, a number of works have proposed optimizing \cref{eq: semidual_n} over a class of input convex neural networks \citep{makkuva2020optimal,bunne2022supervised}; however, this approach currently comes with no statistical guarantees.
Indeed, the analysis of \cite{hutter2021minimax} relies on the explicit form of the class $\cF_W$ (including the multiresolution structure of the wavelet basis), and does not provide a road map for developing estimators over more general function classes.

In this work, we bridge this gap in the literature by presenting a unified perspective on estimating optimal transport maps. 
We give a general decomposition of the risk of $\hat\phi_\cF$ (\Cref{thm:bounded} and \Cref{thm:strongly_convex}), analogous to~\eqref{eq:regression}, which provides guarantees on the estimation error in terms of the metric entropy of $\cF$.
As in the regression context, the class $\cF$ can be chosen at the practitioner's discretion based on prior knowledge to balance the estimation and approximation terms. 

Structural assumptions on the optimal transport maps allow us to obtain near-optimal rates for a large number of cases with a single result, with the minimax result of \cite{hutter2021minimax} as a special case. For example, we provide convergence rates when $\cF$ is:
\begin{enumerate}
    \item the set of quadratics $x \mapsto \tfrac12 x^\top A x + b^\top x$,
    \item a finite set,
    \item a parametric family,
    \item a Reproducing Kernel Hilbert Space (RKHS),
    \item a class of ``spiked" potential functions (borrowing terminology from \cite{niles2022estimation}),
    \item a class of shallow neural networks (i.e., a Barron space \citep{barron,bach2017breaking,e_barron}),
    \item or the space of input convex neural networks.
\end{enumerate}

Additionally, unlike previous results in the literature, we are able to give general decomposition results for the risk of $\nabla\hat\phi_\cF$ under minimal regularity assumptions on the source measure (\Cref{thm:bounded} and \Cref{thm:strongly_convex}.\ref{it:only_poincare}): we only require that $P$ satisfies a Poincaré inequality. For example, $P$ can have a density (bounded away from $0$ and $\infty$) on either a domain with Lipschitz boundary or a manifold, or a density proportional to $e^{-\|x\|^s}$ for $s\geq 1$ on $\RR^d$. In particular, we cover the important case where $P$ is a Gaussian.
Prior work, even for the H\"older case, required stringent restrictions on $P$ which excluded this important special case.

As an appetizer to our main theorems (stated in full generality in \Cref{sec: main_results}), let us consider two very different situations. In the first one (\Cref{sec:parametric}), we assume that $\phi_0$ belongs to a parametric class $\cF$, indexed by a finite dimensional set $\Theta$. Then, under mild conditions on the parametrization, we show that
\begin{equation}\label{eq:thm_simple_param}
    \E\|\nabla\hat\phi_\cF-\nabla\phi_0\|^2_{L^2(P)} \lesssim_{\log n} n^{-1},
\end{equation}
that is, a parametric rate of convergence holds. Although simple and expected, such a result is not present in the literature so far. 

A second, non-parametric example covered by our results is when $\nabla\phi_0$ is a shallow neural network with ReLu activation function (\Cref{sec:barron}). Under such an assumption, for an appropriate class of candidate neural networks $\cF$, we are able to show that
\begin{equation}\label{eq:thm_simple_relu}
    \E\|\nabla\hat\phi_\cF-\nabla\phi_0\|^2_{L^2(P)} \lesssim_{\log n} n^{-\frac{1}{2}-\frac{1}{d-1}},
\end{equation}
a rate which we show is close to being minimax optimal.

\subsection*{Organization}
We begin by providing some background on optimal transport maps in \Cref{sec: background}. Our main results are presented in \Cref{sec: main_results}, where we state rates of estimation in two settings: (i) when $P$ is compactly supported and (ii) when $P$ is unbounded, but $\phi_0$ can be efficiently approximated by strongly convex potentials in $\cF$.
We present our proofs in \Cref{sec: proofs}, deferring technical lemmas to the appendix. We provide several examples that verify the metric entropy condition on $\cF$ required to obtain our proposed convergence rates in \Cref{sec:examples}. Finally, we give positive and negative results on algorithmic estimation of optimal transport maps in \cref{sec: computation}.

\subsection*{Notation}
Let $\Rd$ be the Euclidean space for $d\geq 2$. The open ball of radius $R>0$ centered at $x\in \R^d$ is denoted by $B(x;R)$. 
We assume that we have access to a collection of i.i.d.~random variables $(X_i,Y_i)_{i\geq 1}$ from law $P\times Q$, that are all defined on the same probability space. Depending on the context, we write either $P(f)$, $\int f\dd P$ or $\E_P[f(X)]$ for the integral of a function $f$ against a (probability) measure $P$. 
The pushforward of $P$ by a measurable function $T:\R^d\to \R^m$ is written as $T_\sharp P$. The $L^2(\rho)$-norm of a function $f:\R^d\to \R^m$ with respect to a $\sigma$-finite measure $\rho$ is written as $\|f\|_{L^2(\rho)} = (\int \|f\|^2\dd \rho)^{1/2}$. 
Let $p\in (0,1]$, and let $\psi_p:[0,+\infty)\to [0,+\infty)$ be a given convex increasing function with $\psi_p(x)=0$ and $\psi_p(x)=e^{x^p}-1$ for $x$ large enough. If $X$ is a random variable, we say that $X$ belongs to the Orlicz space of exponent $p$ if there exists $c>0$ with $\E[\psi_p(\|X\|/c)]\leq 1$. The smallest number $c$ satisfying this condition is called the $p$-Orlicz norm of $X$, denoted by $\|X\|_{\psi_p}$. Properties of random variables having finite Orlicz norms are given in \Cref{app: orlicz_app}.

For $\alpha > 1$, we write $\cC^{\alpha}(\Omega)$ to be the space of $\lfloor \alpha\rfloor$-times differentiable functions defined on a domain $\Omega$ whose $\lfloor \alpha\rfloor$th derivative is $\alpha - \lfloor \alpha\rfloor$ H{\"o}lder smooth; see \Cref{eq:def_holder} for more details. The gradient of a twice differentiable function $f$ is written as  $\nabla f$, whereas its Hessian is written as $\nabla^2 f$. If $w:\Rd\to \R$ is a nonnegative function, we let $L^\infty(w)$ be the space of functions $f:\Rd\to \R$ endowed with the norm $\|f\|_{L^\infty(w)} = \sup_{x\in \Rd} |f(x)|w(x)$. For $a\in \R$, we write $\dotp{x}^a$ for $(1+\|x\|)^a$. We also let $\log_+(x)$ denote the function $\max\{1,\log(x)\}$. For $d \in \NN$, we write $\S^d_+$ (resp. $\S^d_{++}$)  to be the space of symmetric positive semi-definite (resp. positive definite) matrices. The smallest eigenvalue of a symmetric matrix $A$ is written as $\lambda_{\min}(A)$, whereas its operator norm is $\op{A}$.  
Finally, we use $a \lesssim b$ to indicate that there exists a constant $C > 0$ such that $a \leq Cb$, and will often use e.g.,   $\lesssim_{\log n}$ to omit polylogarithmic factors in $n$. Similarly, we use $ a \asymp b$ when both $a\lesssim b$ and $b \lesssim a$ hold true.

Throughout this work, we will repeatedly consider suprema of collections of random variables, and such a supremum may not necessarily be measurable. If this is the case, the symbols $\E$ and $\PP$ have to be considered as representing the outer expectation of the supremum. All relevant results used to bound the expectation of such suprema hold for outer expectations, so that we will not make the distinction between expectation and outer expectation. We will also always assume implicitly that $\hat \phi_\cF$ is measurable: this can always be ensured by replacing $\cF$ by a countable dense subset (with respect to the $\infty$-norm). 

\section{Background on optimal transport under the quadratic cost}\label{sec: background}
For a domain $\Omega \sse \Rd$, let $\cP^2(\Omega)$ be the space of probability measures having a finite second moment on $\Omega$.  
Our analysis primarily hinges on existence results and stability bounds for transport maps, which we will later relate to empirical processes. We refer the reader to standard texts on optimal transport for more details e.g.,   \cite{villani2009optimal,San15}.

The \emph{Monge problem} in optimal transport seeks to find a transport map that minimizes the average cost of displacement between two measures
\begin{align}\label{eq: monge_p}
    \inf_{T \in \cT(P,Q)} \int \|x - T(x)\|^2 \dd P(x)\,,
\end{align}
where $\cT(P,Q) \defeq \{T :\Rd \to \Rd \ | \ T_\sharp P = Q \}$ is the set of transport maps from $P$ to $Q$, two measures in $\cP^2(\R^d)$.
The value of this optimization problem is the squared Wasserstein distance $W_2^2(P, Q)$.

 We henceforth suppose that the target measure is given by $Q \defeq (\nabla \phi_0)_\sharp P$, for some differentiable convex function $\phi_0$.
We then obtain the dual characterization
\begin{align}\label{eq: semidual}
	\begin{split}
		\tfrac12 W_2^2(P,Q) &= \int \tfrac12 \|\cdot\|^2 \dd P + \int \tfrac12 \|\cdot\|^2 \dd Q - \inf_{\varphi \in L^1(P)} S(\phi) \\
		&:= \int \tfrac12 \|\cdot\|^2 \dd P + \int \tfrac12 \|\cdot\|^2 \dd Q -\inf_{\varphi \in L^1(P)} \int \varphi \dd P + \int \varphi^* \dd Q,
	\end{split}
\end{align} 
where $\phi^*$ denotes convex conjugation.
The functional $S(\phi)$ is known as the \emph{semi-dual} of the original transport problem.
This dual characterization can be used to show that if $Q = (\nabla \phi_0)_\sharp P$ where $\phi_0$ is smooth and convex, then $\phi_0$ minimizes the semi-dual, and $\nabla \phi_0$ is the $P$-a.e.\ unique solution to~\eqref{eq: monge_p}.

The function $\varphi_0$ is called the Brenier potential. More generally, we refer to the functions $\phi$ in the minimization problem \eqref{eq: semidual}  as potentials. Note that they are defined  up to constants (indeed, shifting $\tilde{\phi}(x) = \phi(x) + c$ does not change the objective function in \cref{eq: semidual}). Hence, we will assume without loss of generality that  $0\in \dom(\phi)$ and that $\phi(0) = 0$ for any potential (where $\dom(\phi)=\{x\in \R^d:\ \phi(x)<\infty\}$ is the domain of a function $\phi$). Particular attention will be paid to functions $\phi$ which have a polynomial growth. The function $\phi$ is said to be \textit{$(\beta,a)$-smooth} (for $\beta,a\geq 0$) if it is twice differentiable and if
\begin{equation}
   \forall x\in\dom(\phi),\ \op{\nabla^2\phi(x)} \leq \beta\dotp{x}^a.
\end{equation}
When $a=0$, this implies that that the function $\phi$ is $\beta$-smooth in the classical sense. Similarly, we say that $\phi$ is \textit{$(\alpha,a)$-convex} (for $\alpha,a\geq 0$) if 
\begin{equation}
   \forall x\in\dom(\phi),\ \lambda_{\min}(\nabla^2\phi(x)) \geq \alpha\dotp{x}^{a}.
\end{equation}
 Once again, $(\alpha,0)$-convex functions correspond to $\alpha$-strongly convex functions in the usual sense. We gather properties of such functions, that we will repeatedly use in our proofs, in \Cref{app:potential}. Note that for $a=0$, twice-differentiability is not required, but only differentiability, and the usual definitions of smoothness and strong convexity can be used.

This leads us to our first proposition which states that the functional $S$ grows at least quadratically around its minimizer $\phi_0$ (up to logarithmic factors). We allow the potentials to either have bounded domains or to be smooth of some exponent $a\geq 0$, while we require $P$ to have \textit{subexponential tails}, that is $X\sim P$ is such that $\|X\|_{\psi_1}<+\infty$. 
Equivalently, a distribution is subexponential if it has tails of the form $P(\|X\|>ct)\leq C e^{-t}$, see \Cref{app: orlicz_app} for more details on subexponential distributions. A proof of \Cref{prop: stab_bound} is found in  \Cref{app:potential}.

\begin{prop}[Map stability]\label{prop: stab_bound}
Let $P$ be a probability distribution with subexponential tails. Consider one of the two following settings:
\begin{enumerate}
    \item the potentials $\phi_0$ and $\phi_1$ are $(\beta,a)$-smooth with domain $\R^d$, and let $b = a(a+1)$.
    \item $P$ is supported in $B(0;R)$ and the potentials $\phi_0$ and $\phi_1$ are twice differentiable, with $\|\nabla\phi_i \|\leq R$  on the support of $P$ and $B(0;2R) \subseteq \dom(\phi_i)$\footnote{The constant $2$ does not play any special role here, and can be replaced by any $C>1$.} ($i=0,1$). In this case, let $b=0$.
\end{enumerate}
Assume that there exists a constant $K$ such that $\|\nabla\phi_1(0)-\nabla\phi_0(0)\|\leq K$ and that $\phi_0$ is convex. Let  $Q \defeq (\nabla\phi_0)_\sharp P $ and $S(\phi_1) \defeq P(\phi_1) + Q(\phi_1^*)$. Denoting $\ell \defeq S(\phi_1) - S(\phi_0)$, we have
\begin{align}\label{eq:stab_bound_reverse}
    \|\nabla \phi_1 - \nabla \phi_0\|_{L^2(P)}^2 \lesssim \log_+(1/\ell)^{b}\ell,
\end{align}
where the suppressed constant does not depend on $d$. 
 Furthermore, if $\phi_1$ is $(\alpha,0)$-strongly convex, then
\begin{equation}\label{eq: rev_stab_bound}
    \ell \lesssim \|\nabla \phi_1 - \nabla \phi_0\|_{L^2(P)}^2.
\end{equation} 
\end{prop}

\begin{rmk}
\Cref{prop: stab_bound} is a variant of \cite[Proposition 10]{hutter2021minimax}, where they consider the case where $P$ has a bounded support, resulting in the bound
\begin{align*}
    \| \nabla \phi - \nabla \phi_0\|^2_{L^2(P)} \lesssim \ell\,.
\end{align*}
Similar stability bounds can be found in e.g.,   \cite{deb2021rates,hutter2021minimax,manole2021plugin,muzellec2021near}. 
	\Cref{prop: stab_bound} strictly generalizes these results by showing that a similar inequality holds in the unbounded case with $a > 0$ at the price of an additional logarithmic factor.
\end{rmk}

\begin{rmk}\label{rmk:stab}
If needed, one can weaken the assumptions on the regularity of the potentials: it is enough to assume that they are differentiable, with locally Lipschitz continuous gradients, and corresponding (local) Lipschitz norm that grows at most polynomially.
\end{rmk}

\section{Main results}\label{sec: main_results}

Recall the setup: we suppose that there exists a convex, $(\beta, a)$-smooth function $\phi_0$ such that $Q = (\nabla \phi_0)_\sharp P$. We have access to i.i.d.~samples $X_1,\dots,X_n\sim P$ and $Y_1,\ldots,Y_n \sim Q$. We are interested in studying the convergence rate of the plugin estimator $\nabla \hat{\phi}_\cF$ to $\nabla \phi_0$, where
\begin{equation}\label{eq:def_phi}
    \hat{\phi}_\cF \in \argmin_{\phi \in \cF} \int \phi \dd P_n + \int \phi^* \dd Q_n\,,
\end{equation}
with $\cF$ being a class of $(\beta, a)$-smooth (potentially non-convex) functions that gives a good approximation of $\phi_0$. This is similar to the setup proposed in \cite{hutter2021minimax}, though we consider \textit{any} class of possibly non-convex candidate potentials $\cF$, whereas they only consider a class of strongly convex functions with truncated wavelet expansions.

\begin{rmk}
As $\cF$ is not assumed to be compact, no minimizers of \eqref{eq:def_phi} may exist. In that case, one can pick an approximate minimizer, up to an error $\rho$. All the rates displayed in the following would then hold with an additional factor of $\rho$.
\end{rmk}

We make the following assumptions throughout this paper. 

\begin{enumerate}[start=1, label={\textbf{(A\arabic*)}}]
\item \label{cond:smooth} There exists $\beta,a\geq 0$ such that every $\phi \in \cF$ is $(\beta,a)$-smooth, as well as $\phi_0$;
\item \label{cond:poincare}  The probability measure $P$ satisfies a Poincaré inequality: for every differentiable function $f:\R^d\to \R$, it holds that
    \begin{equation}\label{eq: poincare_ineq}
        \mathrm{Var}_P(f) \leq C_{\texttt{PI}} \int \|\nabla f(x)\|^2 \dd P(x),
    \end{equation}
    where $0\leq C_{\texttt{PI}}<+\infty$.
\end{enumerate}

\begin{rmk}
The Poincaré inequality is ubiquitous in probability and analysis, being for instance connected to the phenomenon of concentration of measure \citep{bobkov1997poincare, gozlan2010poincare}. 
In general, the main obstruction to obtain a Poincar{\'e} inequality is when the measure has disconnected support. 
Examples of conditions that give rise to a probability measure satisfying \ref{cond:poincare} are:
\begin{itemize}
    \item any measure having a density bounded away from zero and infinity on a bounded Lipschitz domain (or more generally with a support satisfying the cone condition, see \cite[Section 6]{adams2003sobolev}),
    \item any measure having a density bounded away from zero and infinity on a compact, connected submanifold (in particular, $P$ need not have a density with respect to the Lebesgue measure),
    \item any log-concave probability measure (see \cite{bobkov1999isoperimetric}),
    \item $P \propto e^{-V}$, with positive constants $c$ and $a$ such that (i) $\dotp{x,\nabla V(x)} \geq c\|x\|$ for $x$ large enough, or (ii) $a\|\nabla V(x)\|^2 - \Delta V(x) \geq c$ for $x$ large enough (see \cite{bakry2008simple}).
\end{itemize}
Special cases include the standard normal distribution, and log-concave measures with $V(x) = \|x\|^s$ for $s \geq 1$.
As a last comment, let us mention that the Poincaré inequality implies that $P$ has subexponential tails \citep{bobkov1997poincare}: there exist $c_1,c_2>0$ such that
\begin{equation}\label{eq:poincare_exp}
   \forall t>c_1\|X\|_{\psi_1},\  P(\|X\|>t)\leq c_2e^{-t/\|X\|_{\psi_1}},
\end{equation}
a fact that we will repeatedly use.
\end{rmk}

\subsection{The bounded case}
Our first result deals with the case where $P$ has a bounded support. We require a control on the size of the class $\cF$ of candidate potentials as well as a uniform bound on the associated transport maps on the support of $P$. This second requirement is equivalent to having an \textit{a priori} bound $R$ on the size of the support of $P$, as well as the largest displacement endowed by the maps. 
For a subset $A$ of a metric space $(E,d)$, the covering number $\cN(h,A,E)$ is defined as the smallest number of balls of radius $h$ in $(E,d)$ that are needed to cover $A$. 

\begin{enumerate}[start=1, label={\textbf{(B\arabic*)}}]
\item \label{cond:covering_bounded}  There exists $\gamma\geq 0$  such that for every $h>0$,
\begin{equation}\label{eq:covering_1}
\log \cN(h,\cF,L^\infty) \lesssim_{\log_+(1/h)} h^{-\gamma};
\end{equation}
\item \label{cond:bounded}  $P$ is supported in $B(0;R)$. Moreover, for every $\phi \in \cF$, $\dom(\phi)=B(0;2R)$ and $\|\nabla \phi\|\leq R$ on the support of $P$.
\end{enumerate}

\begin{thm}\label{thm:bounded}
Let $P$ be a probability measure and $\cF$ be a set of convex potentials such that \ref{cond:smooth}, \ref{cond:poincare}, \ref{cond:covering_bounded} and \ref{cond:bounded} hold, and let $\ell^\star = \inf_{\phi\in \cF} (S(\phi)-S(\phi_0))$. Then,
\begin{equation}
    \E \|\nabla \hat{\phi}_\cF - \nabla \phi_0\|^2_{L^2(P)} \lesssim_{\log n} \ell^\star  +    \pran{  n^{-\frac{2}{2+\gamma}} \vee  n^{-\frac{1}{\gamma}}},
\end{equation}
where the suppressed constant does not depend on the dimension $d$.
\end{thm}
\Cref{thm:bounded} follows from standard bounds on the risk of empirical risk minimizers under a growth control of the population risk around its minimum of the type \eqref{eq:stab_bound_reverse}. It could for instance be obtained as an application of \cite[Theorem 8.3]{massart2007concentration}, although we use a slightly different route that will allow us to treat the strongly convex case at the same time.
Note that two regimes give matching lower bounds for \cref{thm:bounded}: when $\cF$ is the set of smooth convex functions (when $\gamma = d/2$, and the rate is $n^{-2/d}$) and for $\cF$ consisting of only two functions (when $\gamma = 0$, and the rate is $n^{-1}$). These lower bounds can be inferred from those provided in \cite{hutter2021minimax}.  

\subsection{The strongly convex case}\label{sec:strong_results}
For $\alpha>0$, let $\cF_\alpha \subseteq \cF$ be the set of potentials in $\cF$ that are $(\alpha,a)$-convex for the same parameter $a$ as in condition \ref{cond:smooth}. 
In the second setting, we do not assume that $P$ has a bounded support anymore, but assume instead that $\phi_0$ can be approximated by a strongly convex potential of exponent $a\geq 0$ 
in the sense that the bias term $\inf_{\phi\in \cF} (S(\phi)-S(\phi_0))$ in \Cref{thm:bounded} will now be replaced by $\inf_{\phi\in \cF_\alpha} (S(\phi)-S(\phi_0))$. If $\phi_0$ is $(\alpha,a)$-convex and belongs to $\cF$, this bias term reduces to $0$.
Since $P$ need not have bounded support, condition \ref{cond:covering_bounded} has to be modified to handle unbounded potentials. This is done by replacing the $L^\infty$-norm by a weighted $L^\infty$-norm.

\begin{enumerate}[start=1, label={\textbf{(C\arabic*)}}]
    \item \label{cond:covering_convex} There exist $\eta\geq a+2$ and $\gamma \geq 0$ such that for every $h>0$,
\begin{equation}\label{eq:covering_2}
\log \cN(h,\cF,L^\infty(\dotp{\ \cdot\ }^{-\eta})) \lesssim_{\log_+(1/h)}  h^{-\gamma}.
\end{equation}
\end{enumerate}

Condition \ref{cond:covering_convex} together with \ref{cond:smooth} and \ref{cond:poincare} is enough to obtain the same rates as in \Cref{thm:bounded}. However, strong convexity can be leveraged to obtain faster rates of estimation. 
\begin{enumerate}[start=2, label={\textbf{(C\arabic*)}}]
    \item \label{cond:covering_gradient} Let $\overline\phi\in \cF$ be a fixed potential and let $\alpha>0$. Let $\cG\subseteq  \{t\phi+(1-t)\overline \phi:\ \phi \in \cF,\ 0\leq t \leq 1\}$ be a set of $(\alpha,a)$-strongly convex potentials and let $\cG^*=\{g^*:\ g\in \cG\}$. Then, there exists $m\geq 2$ such that for any $\tau>0$ the ball $B_\tau$ centered at $\overline\phi$ (resp.~the ball $B^*_\tau$ centered at $\overline\phi^*$)  of radius $\tau$ in $\cG$ for the pseudo-norm $g\mapsto \|\nabla g\|_{L^2(P)}$ (resp.~in $\cG^*$ for the pseudo-norm $g^*\mapsto \|\nabla g^*\|_{L^2(Q)}$) satisfies for all $0<h<\tau$,
    \begin{align}
        &\log \cN(h,B_\tau,L^2(P)) \lesssim_{\log_+(1/h),\log_+(1/\tau)}  (h/\tau)^{-m}, \\
        &\log \cN( h,B_\tau^*,L^2(Q)) \lesssim_{\log_+(1/h),\log_+(1/\tau)}  (h/\tau)^{-m},
    \end{align}
    where the suppressed constant may depend on $\alpha$, $a$ and $m$.
\end{enumerate}

Before further analyzing the technical condition \ref{cond:covering_gradient} in detail, let us notice that the bounds on the covering numbers that appear in this condition correspond exactly to the covering numbers of a $H_1^2$-Sobolev ball in $L^2([0,1]^m)$. Condition \ref{cond:covering_gradient} is therefore automatically satisfied with $m=d$ whenever $P$ and $Q$ have densities bounded away from zero and infinity on some bounded Lipschitz domain.

More generally, the next lemma shows that the parameter $m$ can be picked as the dimension of the support of $P$ under mild conditions. Namely, we require two technical conditions that can be easily verified for large classes of probability measures $P$.
\begin{enumerate}[start=3, label={\textbf{(C\arabic*)}}]
    \item \label{cond:loc_poincare} The probability measure $P$ satisfies a \emph{local Poincaré inequality}: there exist $C_{\texttt{LPI}}, \theta \geq 0$, $\kappa>0$ such that for all $R>0$, $0<r\leq \kappa \dotp{R}^{-\theta}$ and  every ball $B=B(x;r)$ centered at $x\in \supp(P)\cap B(0;R)$, it holds that for every differentiable function $f:\R^d\to \R$
\begin{equation}\label{eq: local_poincare_ineq}
        \mathrm{Var}_{P_B}(f) \leq C_{\texttt{LPI}} r^2\int \|\nabla f(x)\|^2 \dd P_B(x),
    \end{equation}
    where $P_B=P_{|B}/P(B)$ is the law of $X\sim P$ conditioned on $X\in B$.
    \item \label{cond:loc_doubling} The probability measure $P$ is  a \emph{locally doubling measure}: there exist constants $C_{\texttt{D}}, \theta \geq 0$, $\kappa>0$ such that for all $R>0$, $0<r\leq \kappa \dotp{R}^{-\theta}$ and every $x\in \supp(P)\cap B(0;R)$,
     \begin{equation}
         P(B(x;r))\leq C_{\texttt{D}} P(B(x;r/2)).
     \end{equation}
\end{enumerate}
\begin{lem}\label{lem:C2_satisfied}
    Let $P$ be a probability measure that satisfies \ref{cond:poincare}, \ref{cond:loc_poincare} and \ref{cond:loc_doubling}. Suppose also that there exist $c>0$, $m\geq 2$ such that for all $R$ large enough and $\eps$ small enough, $ \cN(\eps,\supp P \cap B(0;R), \R^d) \leq (c\eps/R)^{-m}$.  Then,  \ref{cond:covering_gradient} is satisfied with exponent $m$ for any class $\cF$ of potentials satisfying \ref{cond:smooth}.
\end{lem}
A proof of \Cref{lem:C2_satisfied} is given in \Cref{sec: properties_covering_bracketing}. 
As a consequence of this lemma, \ref{cond:covering_gradient} holds  when $P$ has a density bounded away from zero and infinity on a compact manifold, with $m$ equal to the dimension of the manifold. One can easily check that the conditions of the lemma also hold with $m=d$ for measures $P\propto e^{-V}$, where $V$ is (almost-everywhere) differentiable and $\|\nabla V(x)\|\leq L\dotp{x}^\theta$ for some $\theta,L\geq 0$: indeed, it then holds that $|V(x)-V(y)|\leq L\kappa$ if $\|x-y\|\leq r\leq \kappa \dotp{R}^{-\theta}$, which implies both \ref{cond:loc_doubling} and that, if $B=B(x;r)$, then $P_B$ has a density that lies between $e^{-L\kappa}$ and $e^{L\kappa}$, which therefore satisfies a Poincaré inequality with constant of order $e^{2L\kappa}$. In particular, $V(x)=\|x\|^s$ for $s\geq 1$ is admissible, including the important case of the Gaussian distribution. 

The condition \ref{cond:covering_gradient} will also be shown to be satisfied for $m<d$ in situations where $P$ has full support, but the class $\cF$ of candidate potentials has a low dimensional structure, e.g.,  in the spiked transport model (\Cref{sec: lowdim}) or when candidate potentials belong to a Barron space (\Cref{sec:barron}).

\begin{thm}\label{thm:strongly_convex}
Let $\alpha>0$. Let $P$ be a probability measure and $\cF$ be a class of potentials such that \ref{cond:smooth}, \ref{cond:poincare}, \ref{cond:covering_convex} and \ref{cond:bounded}   hold (possibly  with $R=\infty$).  Let  $\ell^\star = \inf_{\phi\in \cF_\alpha}\|\nabla\phi-\nabla\phi_0\|^2$.
\begin{enumerate}
    \item \label{it:only_poincare} We have
\begin{equation}\label{eq:rate_strong}
    \E\|\nabla \hat{\phi}_\cF - \nabla \phi_0\|^2_{L^2(P)} \lesssim_{\log n,\log_+(1/\ell^\star)} \ell^\star +   \pran{n^{-\frac{2}{2+\gamma}} \vee  n^{-\frac{1}{\gamma}}}.
\end{equation}
\item \label{it:uniform} If $P$ also satisfies  \ref{cond:covering_gradient} and $\gamma\in [0,2)$, then
\begin{equation}\label{eq:rate_strong_more}
    \E\|\nabla \hat{\phi}_\cF - \nabla \phi_0\|^2_{L^2(P)} \lesssim_{\log n,\log_+(1/\ell^\star)} \ell^\star +   n^{-\frac{2(m-\gamma)}{2m+\gamma (m-4)}}.
\end{equation}
\end{enumerate}
\end{thm}

\begin{rmk}
We do not impose the convexity of the candidate potentials  in \Cref{thm:strongly_convex}. 
    Being able to handle class of non-convex candidate potentials gives us additional flexibility when designing classes of candidate potentials $\cF$, e.g.,  when using neural networks whose output will not necessarily be convex. There will also be several cases of interest where we will consider classes of potentials $\phi$ having a finite expansion of the form $\sum_j \lambda_j \phi_j$ for some coefficients $\lambda_j\in\R$ and some non-convex functions $(\phi_j)_j$, typically the elements of some $L^2$-orthogonal basis, be it wavelets or trigonometric functions. 
\end{rmk}

\section{Examples}\label{sec:examples}

\subsection{Transport between Location-Scale families}\label{sec:LocationScale}
We first verify our theorems on general location-scale families. For concreteness, we begin with the case of multivariate Gaussian distributions. Let $P$ be the Gaussian distribution $N(m_P,\Sigma_P)$, and let $Q$ be  $N(m_Q,\Sigma_Q)$. It is well-known that the optimal Brenier potential is given by
\begin{equation}\label{eq: quad_pot}
    \phi_0:x\mapsto \tfrac12 (x - {m}_P)^\top \left( {\Sigma}_P^{-1/2}({\Sigma}_P^{1/2} {\Sigma}_Q{\Sigma}_P^{1/2})^{1/2}{\Sigma}_P^{-1/2}   \right)(x - {m}_P) + m_Q^\top x\,.
\end{equation}
In the two-sample problem, where we have $n$ samples from both $P$ and $Q$, the plug-in estimator is given by
\begin{align*}
    \hat{T}_n(x) = \hat{m}_Q + \left( \hat{\Sigma}_P^{-1/2}(\hat{\Sigma}_P^{1/2} \hat{\Sigma}_Q\hat{\Sigma}_P^{1/2})^{1/2}\hat{\Sigma}_P^{-1/2}   \right)(x - \hat{m}_P)\,,
\end{align*}
where $\hat{m}_P$ (resp. $\hat{m}_Q$) and $\hat{\Sigma}_P$ (resp. $\hat{\Sigma}_Q$) are empirical estimates of the mean and covariance of $P$ (resp. $Q$) from samples. One can show by direct arguments that $\hat{T}_n$ achieves parametric rates; see \cite{flamary2019concentration}. More generally, let $P_0$ be a centered, isotropic base distribution on $\mathbb R^d$.
If $P = \mathrm{Law}(m_P + \Sigma_P^{1/2} X_0)$ and  $Q = \mathrm{Law}(m_Q + \Sigma_Q^{1/2} X_0)$ for $X_0 \sim P_0$, then the optimal Brenier potential is again given by~\eqref{eq: quad_pot}, and $\hat T_n$ achieves the parametric rate.

As a sanity check, we show that our procedure recovers the rates of estimation in this simple case. The location-scale setup suggests to consider a family of candidate potentials
\begin{equation}
    \cF_{\text{quad}}=\{ x\mapsto \tfrac12 x^\top B x + \dotp{b, x}:\ B \in \S^d_+,\ b\in\R^d\}.
\end{equation}
Concretely, our estimator is $\nabla \hat{\phi}_n$ with
\begin{align*}
    \hat{\phi}_n = \argmin_{\phi \in \cF_{\text{quad}}} \int \phi \dd P_n + \int \phi^* \dd Q_n\,.
\end{align*}
One can verify that $\nabla \hat{\phi}_n$ precisely matches the plug-in estimator. However, $\hat{\phi}_n$ (and its gradient) does not immediately fall into our framework since the set $\cF_{\text{quad}}$ is not precompact, and therefore its covering numbers are infinite. However, we show in \Cref{sec:affine} that it suffices to control the covering number of the smaller set
\begin{equation*}
    \cF_{\text{quad},0} = \{ x\mapsto \tfrac{1}{2}x^\top B x:\ B \in \S^d_+,\ \op{B}= 1\}\,
\end{equation*}
which grows at most logarithmically. 
\Cref{thm:strongly_convex} then yields the following bound.
\begin{prop}\label{prop:Gaussian}
Assume that $P$ and $Q$ lie in a location-scale family and that the base measure $P_0$ satisfies \ref{cond:poincare} and has a bounded density along every direction in $\mathbb R^d$.
It holds that 
\begin{equation}
    \E\|\nabla\hat\phi_{\cF_{\mathrm{quad}}}-\nabla\phi_0\|^2_{L^2(P)} \lesssim_{\log (n)} n^{-1}.
\end{equation}
\end{prop}
Full details appear in \cref{sec:affine}.
\begin{rmk} The proof of \cref{prop:Gaussian} is based on an argument showing that the minimizer of~\eqref{eq: semidual_n} over $\cF_{\mathrm{quad}}$ agrees with the minimizer over the subset of $\cF_{\mathrm{quad}}$ for which $B$ and $b$ are bounded by suitable constants. This argument applies more generally to families $\cF$ obtained by translating and rescaling a base family $\cF_0$, as long as the elements of $\cF_0$ satisfy a restricted strong convexity property. See \cref{sec:affine} for details.  \end{rmk}

\subsection{Finite set}
We now consider the setup from \cite{vacher2021convex}, who consider the problem of estimating an optimal Brenier potential over a finite class of functions $\cF=\{\phi_1,\dots,\phi_K\}$. In addition, the authors focus on the case where (i) the measure $P$ has a bounded support, (ii) the potentials $\phi_k$ are strongly convex and smooth, and (iii) the ``true'' potential $\phi_0$ does not necessarily belong to $\cF$. Their main result reads
\begin{equation}\label{eq:control_vacher}
    \E\|\nabla \hat \phi_\cF-\nabla\phi_0\|^2_{L^2(P)} \lesssim \min_{k =1,\dots, K} \|\nabla \phi_{k}-\nabla\phi_0\|^2_{L^2(P)} + \sqrt{\frac{\log n}{n}}.
\end{equation}
\Cref{thm:bounded} allows us to strengthen this result if we assume a Poincaré inequality. Indeed, the log-covering number of $\cF$ is constant for $h$ small enough, so that  condition \ref{cond:covering_bounded} holds with $\gamma=0$. Therefore, if $P$ satisfies a Poincaré inequality and \ref{cond:bounded} is satisfied, then a direct application of \Cref{thm:bounded} yields
\begin{equation}\label{eq:control_finite}
    \E\|\nabla \hat \phi_\cF-\nabla\phi_0\|^2_{L^2(P)} \lesssim_{\log n} \min_{k =1\dots K} (S(\phi_k) - S(\phi_0)) + n^{-1}\,.
\end{equation}

As in \cite{vacher2021convex}, if we also assume that the potentials in $\cF$ are strongly convex, we can apply \Cref{eq: rev_stab_bound} , resulting in
\begin{equation}\label{eq:control_finite_strongly}
    \E\|\nabla \hat \phi_\cF-\nabla\phi_0\|^2_{L^2(P)} \lesssim_{\log n} \min_{k \in [K]} \|\nabla \phi_{k}-\nabla\phi_0\|^2_{L^2(P)}  + n^{-1}\,.
\end{equation}
To put it another way, assuming a Poincaré inequality is enough to improve the excess of risk from $n^{-1/2}$ to $n^{-1}$ (up to logarithmic factors). 
\subsection{Parametric space}\label{sec:parametric}
More generally, one can consider a set of potentials parametrized by a bounded set $\Theta\subseteq \R^m$. Consider a set $\cF = \{\phi_\theta:\ \theta\in \Theta\}$ and assume that there exist constants $L,p \geq 0$ such that 
\begin{equation}\label{eq:lip_parametric}
    \forall \theta,\theta'\in\Theta,\ \forall x\in \Rd,\  |\phi_\theta(x)-\phi_{\theta'}(x)|\leq L\|\theta-\theta'\|\dotp{x}^p.
\end{equation}
Then, for $\eta>p$, the log-covering number at scale $h$ of $\cF$ for the norm $L^\infty(\dotp{\ \cdot \ }^{-\eta})$ is controlled by the log-covering number of the set $\Theta$, which scales as $m\log_+(1/h)$. Therefore, condition \ref{cond:covering_convex} is satisfied for $\gamma=0$. Let $P$ be a probability measure satisfying a Poincaré inequality and assume that all potentials in $\cF$ are $(\beta,a)$-smooth. \Cref{thm:strongly_convex} implies that 
\begin{equation}
 \E\|\nabla \hat{\phi}_\cF - \nabla \phi_0\|^2_{L^2(P)} \lesssim_{\log n} \inf_{\phi\in\cF_\alpha}\|\nabla\phi-\nabla\phi_0\|^2_{L^2(P)} + n^{-1},
\end{equation}
that is, we obtain a parametric rate of convergence.

\subsection{Large parametric spaces}\label{sec:large}
A common technique in nonparametric inference is to approximate a (nonparametric) class of functions $\cF$ by a large finite dimensional space (using e.g.,  Fourier series, wavelets, or neural networks).  If $\tilde \cF$ is such a finite dimensional space, then the log-covering numbers of $\tilde\cF$ will scale logarithmically, although with possibly a large prefactor: for some $\eta\geq a+2$ and $\gamma'\geq 0$, it will hold that
\begin{equation}\label{eq:covering_2_precise}
\forall h>0,\ \log \cN(h,\tilde\cF,L^\infty(\dotp{ \ \cdot\ }^{-\eta})) \leq D \log_+(1/h)^{\gamma'},
\end{equation}
where the constant $D$ is not necessarily small. In such a situation, \Cref{thm:strongly_convex} directly gives an excess of risk of order $n^{-1}$ (parametric rate of convergence). However, such a result is not useful as the hidden constant depends on $D$, which typically scales polynomially with $n$. Determining the exact dependency of the excess of risk as a function of $D$ is thus key to deriving tight bounds when the approximating class $\tilde \cF$ is of finite, although large, dimension. Specializing our proof to this case yields the following bound.

\begin{prop}\label{prop:rate_large_parametric}
    Let $\alpha>0$. Assume that $P$ and $\tilde \cF$ satisfy \ref{cond:smooth}, \ref{cond:poincare}, \ref{cond:bounded}  (possibly  with $R=\infty$) and \eqref{eq:covering_2_precise}.  Further assume that \ref{cond:covering_gradient} is satisfied for $m=d$ and let  $\ell^\star = \inf_{\phi\in \tilde\cF_\alpha} \|\nabla\phi-\nabla\phi_0\|^2_{L^2(P)}$.  Then
    \begin{equation}
     \E\|\nabla \hat{\phi}_{\tilde\cF} - \nabla \phi_0\|^2_{L^2(P)} \lesssim_{\log n,\log_+(1/\ell^\star),\log D} \ell^\star +   D^{1-\frac{2}d}n^{-1}.
\end{equation}
\end{prop}

\Cref{prop:rate_large_parametric} is proven in \Cref{sec: details_large_parametric}. Let us now show  cases of applications to different families of approximating functions. 
Assume for the sake of simplicity 
that $\phi_0$ is $(\alpha,a)$-convex and belongs to the (nonparametric) class of functions $\cF$. 
We define the pseudo-norm $\vertiii{\phi} = \sup_{x\in B(0,R)} \op{\nabla^2\phi(x)}\dotp{x}^{-a}$ for a $C^2$ function $\phi$ and $R$ the radius appearing in \ref{cond:bounded}. 
The class of functions $\cF$ is approximated by  a family $(\cF_j)_{j\geq 0}$ of parametric models that all satisfy condition \ref{cond:smooth} for some uniform constant $\beta$, and that satisfies the following conditions.
\begin{enumerate}[start=1, label={\textbf{(P\arabic*)}}]
    \item \label{cond:parametric} There exists $\eta \geq a+2$ and $\gamma'\geq 0$ such that for every $h>0$ and integer $j\geq 0$,
    \begin{equation}
        \log \cN(h,\cF_j,L^\infty(\dotp{\ \cdot \ }^{-\eta})) \lesssim 2^{jd} \log_+(1/h)^{\gamma'}.
    \end{equation}
    \item \label{cond:parametric_approx}  There exists $s\geq 2$ such that  for every every $j$ large enough and for every $\phi\in\cF$  there exists $\phi_j\in\cF_j$ with
    \begin{equation}
\|\nabla\phi -\nabla\phi_j\|_{L^2(P)}\lesssim 2^{-j(s-1)},\quad \vertiii{\phi_j-\phi}\leq \alpha/2.
\end{equation}
\end{enumerate}

\begin{thm}\label{thm:param}
    Consider $P$ and $\cF$ that satisfy \ref{cond:poincare},  \ref{cond:covering_gradient} (for $m=d$), and  \ref{cond:bounded} (possibly  with $R=\infty$). Assume that $\phi_0$ is both $(\beta,a)$-smooth and $(\alpha,a)$-strongly convex.  Let $(\cF_j)_j$ be an approximating family of the class of potentials $\cF$ that satisfies \ref{cond:parametric}, \ref{cond:parametric_approx}, and that also satisfies \ref{cond:smooth} for some parameter $\beta$ independent of $j$. Then, for $J$ such that  $2^{J} \asymp n^{1/(2s+d-4)}$, it holds that
    \begin{equation}
        \E\|\nabla\hat \phi_{\cF_J}-\nabla\phi_0\|^2_{L^2(P)} \lesssim_{\log n}  n^{-\frac{2(s-1)}{2s+d-4}}.
    \end{equation}
\end{thm}
\begin{proof}
    According to \ref{cond:parametric_approx}, for $j$ large enough, there exists a function $\phi_j\in \cF_j$ with $\vertiii{\phi_j-\phi_0}\leq \alpha/2$, and with $\|\nabla\phi_j -\nabla\phi_j\|_{L^2(P)}\lesssim 2^{-(j-1)s}$. The function $\phi_j$ is therefore $(\alpha/2,a)$-strongly convex, and belongs to the set $(\cF_{j})_{\alpha/2}$ (using notation from \Cref{sec:strong_results}). 
    We apply \Cref{prop:rate_large_parametric} to the set $\cF_j$ to obtain that 
    \begin{equation}
        \E\|\nabla \hat\phi_{\cF_j}-\nabla\phi_0\|^2_{L^2(P)} \lesssim_{\log n,j} 2^{-2(s-1)j} + 2^{j(d-2)}n^{-1}.
    \end{equation}
    To conclude, we choose $J$ such that $2^{J} \asymp n^{1/(2s+d-4)}$.
\end{proof}

\begin{rmk}
The introduction of a family of models $(\cF_j)_{j\geq 0}$ hints at the possibility of implementing model selection techniques for the transport map estimation problem. This can be performed using a standard penalization scheme. In fact, the formalism developed in \Cref{sec: proofs} is completely in line with Massart's book on model selection \citep{massart2007concentration}, and one can check directly that \cite[Theorem 8.5]{massart2007concentration} can be readily applied to our setting.
\end{rmk}

\subsubsection{Wavelet expansions}\label{sec:holder}
As a first application of \Cref{thm:param}, we may approximate any potential $\phi$ of regularity $s\geq 2$ by potentials having a finite wavelet expansion: this is the method proposed by \citet{hutter2021minimax}.  For a $k$-times differentiable function $\phi$, we let $d^k\phi$ be its $k^{\text{th}}$ order differential.  H\"utter and Rigollet consider the setting where the potential $\phi_0$ belongs to $\cF = \cC^s_{L}(\Omega)$, the set of functions $\phi$ defined on some domain $\Omega$ being $\lfloor s \rfloor$-times differentiable, with $\lfloor s \rfloor$-derivatives being $(s-\lfloor s \rfloor)$-H\"older continuous, and norm
\begin{equation}\label{eq:def_holder}
    \|\phi\|_{C^s(\Omega)} = \max_{0\leq i \leq \lfloor s \rfloor}\sup_{x\in \Omega} \op{d^i \phi(x)} + \sup_{x,y\in\Omega}\frac{ \op{d^{\lfloor s \rfloor} \phi(x)-d^{\lfloor s \rfloor} \phi(y)}}{\|x-y\|^{s-\lfloor s \rfloor}}
\end{equation}
smaller than $L$. When $\Omega$ is a bounded convex domain, the authors study the estimator
\begin{equation}
    \hat{\phi}_{(n,\text{W})}=\hat{\phi}_{\cF_J(\text{W}^{\alpha,\beta})} \defeq \argmin_{\phi\in \cF_J(\text{W}^{\alpha,\beta})} S_n(\phi),
\end{equation}
where $\cF_J(\text{W}^{\alpha,\beta})$ is the space of functions having a finite wavelet expansions up to some depth $J$, \textit{while also being $\alpha$-strongly convex} (details on wavelet expansions are given in \Cref{app:besov}). H\"utter and Rigollet show that, if the source measure $P$ has a density bounded away from zero and infinity on $\Omega$, then this estimator enjoys the following rate of convergence (which is minimax optimal up to logarithmic factors), for $s\geq 2$,
\begin{equation}\label{eq:rate_HR}
    \E \|\nabla \hat{\phi}_{(n,\text{W})}-\nabla\phi_0\|^2_{L^2(P)} \lesssim_{\log n} n^{-\frac{2(s-1)}{2s+d-4}}.
\end{equation}
Using \Cref{thm:param}, we are able to improve the results of H\"utter and Rigollet in several directions. First, we relax the assumptions on $P$, which does not need to have a  convex support but is only required to satisfy a Poincaré inequality \ref{cond:poincare} as well as a local Poincaré inequality \ref{cond:loc_poincare}, while being a doubling measure \ref{cond:loc_doubling}. Second, we relax the hypothesis of $P$ having a bounded support, which allows us to handle the case of smooth and strongly log-concave measures, see \Cref{cor:logconcave}. 

The third improvement we make is related to the computation of the estimator. Indeed, in their numerical implementations, H\"utter and Rigollet did not compute the estimator $\hat{\phi}_{(n,\text{W})}=\hat{\phi}_{\cF_J(\text{W}^{\alpha,\beta})} $ (where $\cF_J(\text{W}^{\alpha,\beta})$ is the set of \textit{$\alpha$-strongly convex} functions with a finite wavelet expansion), but removed the $\alpha$-strong convexity assumptions (for otherwise, the estimator $\hat{\phi}_{(n,\text{W})}$ is not computable). They actually compute the estimator $\hat \phi_{\cF_J(R_c)}$, where, for an integer $J>0$ and $R_c>0$, the set $\cF_J(R_c)$ contains all potentials having a finite wavelet expansion of depth $J$ on the cube of side-length $R_c$, with bounded coefficients, and extended by an arbitrary fixed convex function outside the cube. Their work does not supply a theoretical justification for this modification, but our \Cref{thm:param} shows it achieves the minimax rate. See \Cref{app:besov} for details.

\subsubsection{Application: Transport between log-concave measures}
When the source measure potentially has unbounded support, we additionally assumed that the true Brenier potential $\phi_0$ has upper and lower bounded Hessian (or can be well-approximated by such a potential). We want to emphasize that sometimes these assumptions can be quantitatively verified based on structural properties of the source and target measures alone.
For example, consider the case when the source and target measures are appropriately log-smooth and log-strongly concave: $P = \exp(-V)$ and $Q = \exp(-W)$ with $ 0\preceq \alpha_V I \preceq \nabla^2 V(x) \preceq \beta_V I$ and similarly $ 0\preceq \alpha_W I \preceq  \nabla^2 W(y) \preceq \beta_W I$, with $P$ and $Q$ having full support on $\Rd$.  A seminal result by \citet{caffarelli2000monotonicity} states that the optimal Brenier potential $\phi_0$ is then $\sqrt{\alpha_V/\beta_W}$-strongly convex and $\sqrt{\beta_V/\alpha_W}$-smooth, thus verifying the smoothness and strong convexity requirements of \Cref{thm:param}. This principle is known as \textit{Caffarelli's contraction theorem}; see e.g.,~\cite{chewi2022entropic} for a short proof based on entropic optimal transport.  

Recall the function class ${\cF_{J}(R_c)}$ from the previous section. The proof in \Cref{app:besov} yields the following corollary.
\begin{cor}\label{cor:logconcave}
Let  $P$ and $Q$ be log-smooth and log-strongly convex, with $Q=(\nabla\phi_0)_{\sharp} P$. Then, for $R_c \asymp J \asymp \frac{1}{d}\log(n)$, where $n$ is the number of samples from each of $P$ and $Q$, it holds that
\begin{equation}
    \E \|\nabla\hat\phi_{\cF_{J}(R_c)}-\nabla\phi_0\|^2 \lesssim_{\log n} n^{-\frac{2}{d}}\,.
\end{equation}
\end{cor}
To our knowledge, this is the first result establishing rates of estimation for transport maps between general log-concave measures.

\subsubsection{ReQU neural networks}
As already mentioned in the introduction, several empirical works have proposed  optimizing the semidual problem $S_n$ over a family of neural networks. Putting computational considerations aside for a moment (see \Cref{sec: computation}), \Cref{thm:param} shows that this approach is sound. Indeed, \citet{belomestny2023simultaneous} design families of neural networks with ReQU activation functions that are able to approximate a function $f$  for different H\"older norms. Let $\Omega=[0,1]^d$ and let $\cF=\cC^s_L(\Omega)$ for some $L\geq 0$ and $s>2$. For each $j \geq 0$, the authors build an approximating family $\cF_j$ (of $\beta$-smooth functions for some $\beta>0$) such that for $j$ large enough and for all $\phi\in \cF$, there is a function $\phi_j\in \cF_j$ with
\[ \vertiii{\phi-\phi_j}\leq \alpha/2 \text{ and } \|\phi-\phi_j\|_{C^1(\Omega)} \lesssim 2^{-j(s-1)},\]
see \cite[Theorem 2]{belomestny2023simultaneous}. The number of parameters using to define the ReQU neural network $\phi_j$ is of order $2^{jd}$, while \cite[Lemma 19]{belomestny2021rates} ensures that \ref{cond:parametric} holds. Assume that $P$ satisfies \ref{cond:poincare}, \ref{cond:loc_poincare}, \ref{cond:loc_doubling} (e.g.,  $P$ has a density bounded away from $0$ and $\infty$ on $[0,1]^d$).  We can  take $R$ large enough so that $[0,1]^d\subset B(0;R)$ and $\|\nabla\phi\|\leq R$ for all $\phi\in \cF_j$, $j\geq 0$. We then modify the family of potentials $\cF_j$ by defining the potentials to be equal to $+\infty$ outside $B(0,2R)$, so that \ref{cond:bounded} is satisfied for the modified set $\tilde\cF_j$. According to \Cref{thm:param}, for a certain choice of $J$, the family $\tilde\cF_J$ of ReQU neural networks can be used to approximate an $\alpha$-strongly convex potential $\phi_0\in \cF$, with minimax rate of estimation
\begin{equation}
    \E\|\nabla\hat\phi_{\tilde\cF_J}-\nabla\phi_0\|^2_{L^2(P)} \lesssim_{\log n} n ^{-\frac{2(s-1)}{2s+d-4}}.
\end{equation}

\subsection{Reproducing Kernel Hilbert Spaces}\label{sec: rkhs}
We turn our attention to a class of transport maps given by smooth potential functions that lie in a Reproducing Kernel Hilbert Space (RKHS). We briefly recall that a RKHS is a Hilbert space $\cH$ of functions over a domain $\cX$ where there exists a \textit{kernel} $\cK : \cX \times \cX \to \R$ such that for all $f \in \cH$, for all $x \in \cX$,
\begin{align*}
    f(x) = \langle f , \cK(\cdot,x)\rangle_\cH\,.
\end{align*}
We refer to \cite{scholkopf2002learning,wainwright2019high} for more details on the properties of such spaces.

For our application of interest, we let ${\cF}$ be the unit ball in some RKHS pertaining to some fixed kernel $\cK$. We will assume that $\cK$ is of class $\cC^4$, which is enough to ensure that all potentials $\phi\in\cF$ are $\beta$-smooth for some $\beta>0$ according to \cite{zhou2008derivative}. To complete our analysis, we use existing bounds on the $L^\infty$-covering numbers over such a set (see e.g.,  \cite{yang2020function}), which are expressed in terms on the spectral properties of the kernel operator. For simplicity, we assume that either the spectrum is finite or has eigenvalues that decay exponentially fast, 
see \cref{app: rkhs_spec} for details on these definitions and assumptions. For example, when the kernel is associated to a finite-dimensional feature mapping, the RKHS is finite-dimensional and thus exhibits a finite spectrum. The well-known Gaussian kernel
\begin{align*}
    \cK_\sigma^2(x,y) = \exp(-\|x-y\|^2/\sigma^2)\,,
\end{align*}
is an example of a kernel with exponentially fast decaying spectrum over the sphere, with $\sigma^2 > 0$. For such smooth kernels, one can show that the log-covering numbers exhibit the following bound \cite[Lemma D.2]{yang2020function} 
\begin{align*}
    \log\cN(h, \cF, L^\infty) \lesssim \log_+(1/h)^{\gamma'}\,,
\end{align*}
for some $\gamma' \geq 1$. 
Assume that $P$ satisfies a Poincaré inequality  \ref{cond:poincare} and has a bounded support. Take $R$ large enough so that the support of $P$ is included in $B(0;R)$ and that $\|\nabla\phi\|\leq R$ on the support of $P$  for all $\phi\in\cF$ (such a $R$ exists as the potentials in $\phi$ are uniformly bounded in $C^1$-norm). We then modify the potentials $\phi$ to be equal to $+\infty$ outside $B(0;2R)$, so that condition \ref{cond:bounded} is satisfied for the set of modified potential $\tilde \cF$. By Theorem \ref{thm:strongly_convex}, we  obtain parametric convergence i.e., for $\alpha>0$,
\begin{align*}
    \E \|\nabla \hat{\phi}_{\tilde\cF} - \nabla \phi_0\|^2_{L^2(P)} \lesssim_{\log(n)} \inf_{\phi\in \tilde\cF_\alpha}\|\nabla\phi-\nabla\phi_0\|^2 +  n^{-1}\,.
\end{align*}

\subsection{``Spiked'' potential functions}\label{sec: lowdim}

The \textit{spiked transport model} was recently proposed by Niles-Weed and Rigollet to model situations where two probability measures of interest living  in a high-dimensional space $\R^d$ only differ on  a low-dimensional subspace; see \cite{niles2022estimation}. Let $\cV_{k \times d}$ be the Stiefel manifold, consisting of all $k\times d$ matrices with orthonormal rows.
We assume that there exists a matrix $U\in\cV_{k \times d} $   such that $Y\sim Q$ can be obtained from $X\sim P$ through a map $T$ that can be decomposed into $T(x) =  U^\top T'(Ux)+(I-U^\top U) x$. At the level of the potentials, this is equivalent to assuming that the Brenier potential $\phi_0$ belongs to the function space
\begin{equation}
\cF_{\mathrm{spiked}} \defeq \left\{x\mapsto \phi(Ux) + \frac{1}{2}\|x-U^\top Ux\|^2:\ \phi\in \cF,\ U\in \cV_{k \times d}\right\},
\end{equation}
where $\cF$ is a class of potentials defined on $\R^k$. The covering numbers of the class $\cF_{\mathrm{spiked}}$ are controlled by the covering numbers of the class $\cF$ up to logarithmic factors.
\begin{prop}\label{prop:spiked}
For $\eta\geq 2$, it holds that
\begin{equation}
\log \cN(h,\cF_{\mathrm{spiked}},L^\infty(\dotp{\ \cdot \ }^{-\eta})) \leq \log \cN(h/2,\cF,L^\infty(\dotp{\ \cdot \ }^{-\eta})) + C_\eta dk\log_+(1/h).
\end{equation}
\end{prop}
To obtain rates of estimation that depend solely on the intrinsic dimension $k$ of the problem, it remains to prove that condition \ref{cond:covering_gradient} is satisfied for some value of $m$ depending only on $k$. We show that this is for instance the case when $P$ has a density bounded away from zero and infinity on a convex domain.

\begin{prop}\label{prop:spiked_gradient}
    Let $P$ have a density bounded away from zero and infinity on a convex domain.  
    Then, $\cF_{\mathrm{spiked}}$ satisfies  \ref{cond:covering_gradient} for $m=2k$.
\end{prop}
Proofs of Propositions \ref{prop:spiked} and \ref{prop:spiked_gradient} can be found in \cref{sec: proof_spiked}. \Cref{prop:spiked_gradient} is likely suboptimal, as one would expect the dimension parameter $m$ to be equal to the dimension $k$ of the spike.
Still, these two properties are enough to show that one can obtain rates of convergence with exponents depending only on the effective dimension $k$ of the problem, the dependency on $d$ only existing through prefactors. For instance, let $\cF$ be a class of $(\beta,a)$-smooth potentials containing some $\alpha$-strongly convex potential $\phi_0'$ and satisfying \ref{cond:covering_convex} for some $\gamma\in [0,2)$. If  $\phi_0$ is given by $x\mapsto \phi'_0(U_0x) +\frac{1}{2}\|x-U_0^\top U_0x\|^2$ for some $U_0\in \cV_{k\times d}$, then
\begin{equation}
    \E\|\nabla\hat \phi_{\cF_{\mathrm{spiked}}}-\nabla\phi_0\|^2_{L^2(P)} \lesssim_{\log n,d} n^{-\frac{(2k-\gamma)}{2k+\gamma(k-2)}}.
\end{equation}
Finding the optimal function in $\cF_{\mathrm{spiked}}$ requires to solve an optimization problem defined over the Stiefel manifold. Such a problem is non-trivial.
In fact, following a strategy from \cite{niles2022estimation}, we show that this estimation problem possesses a \emph{computational-statistical gap}~\citep{BerRig13, BanPerWei18, BreBreHul18, BreBre20}. Specifically, in \cref{sec: computation}, we demonstrate that constructing a computationally efficient estimator for Brenier potentials in the class $\cF_{\mathrm{spiked}}$ would yield a polynomial-time algorithm for the Gap Shortest Vector Problem (see \cref{sec: hardness}), which is conjectured to be impossible even for quantum computers.

\subsection{Barron spaces}\label{sec:barron}
For large-scale tasks, neural networks are often used to parametrize mappings or deformations from a set of data to another.
These deformations can be interpreted as transport maps between a source distribution (often a standard Gaussian measure) and a target distribution, which is to be learned from samples.
Several works have proposed to estimate such maps using optimal transport principles~\citep{makkuva2020optimal,huang2021convex,bunne2022supervised}, but with few accompanying guarantees.

Our general techniques allow us to obtain statistical results in this setting, namely when the Brenier potential lies in a class of shallow neural networks with unbounded width, referred to as a \emph{Barron space} \citep{e_barron}. 

Let $\cM$ be a $k$-dimensional smooth compact manifold (possibly with smooth boundary), and consider an activation function $\sigma:(x,v)\in \R^d\times \cM \to \R$. We define the $\cC^s(\cM)$-norm of a function $f:\cM\to\R$ as in the Euclidean case \eqref{eq:def_holder}, by using an arbitrary systems of charts on $\cM$.  We assume that
\begin{enumerate}[start=1, label={\textbf{(D\arabic*)}}]
    \item \label{cond:barron_smooth} for every $v\in \cM$, $x\mapsto \sigma(x,v)$ is convex with $\sigma(0,v)=0$, and is $(\beta,a)$-smooth,
    \item \label{cond:barron_holder} for every $x\in \Rd$, the function $v\mapsto \sigma(x,v)$ belongs to $\cC^s(\cM)$ for some $s>0$, with a uniformly bounded $\cC^s(\cM)$-norm.
\end{enumerate}

For a given activation function, we define the Barron space $\cF_\sigma$ as the space of functions $\phi:\R^d\to \R$ for which there exists a measure $\theta$ on $\cM$ with finite total variation  such that 
\begin{equation}\label{eq:representation_barron}
    \forall x\in \R^d,\ \phi(x) = \int \sigma(x,v)\dd \theta(v),
\end{equation}
while the Barron norm of $\phi$ is defined as the infimum of the total variations of measures $\theta$ such that the representation \eqref{eq:representation_barron} holds. We let $\cF_\sigma^1$ be the unit ball in $\cF_\sigma$. 

\begin{rmk}
Note that the optimal Brenier potential, $\phi_0$, is therefore associated with an optimal measure $\theta_0$ on the manifold $\cM$. In other words,
\begin{align*}
    \phi_0(x) = \int \sigma(x,v)\dd \theta_0(v)\,.
\end{align*}
From a practitioner's perspective, it  suffices to find the optimal weights associated to the neural network by traditional means (e.g.,  stochastic gradient descent on the objective function of interest).
\end{rmk}

\begin{thm}\label{thm:barron}
Let $P$ be a probability measure satisfying \ref{cond:poincare}, \ref{cond:loc_poincare} and \ref{cond:loc_doubling}, and let $\sigma$ be an activation function satisfying \ref{cond:barron_smooth} and \ref{cond:barron_holder}. Let $\phi_0$ be 
$(\beta,a)$-smooth. Then,
\begin{equation}
    \E[\|\nabla\hat \phi_{\cF_\sigma^1}-\nabla\phi_0\|^2_{L^2(P)}] \lesssim_{\log n, d} \inf_{\phi\in (\cF_\sigma^1)_\alpha} \|\nabla\phi-\nabla \phi_0\|^2_{L^2(P)} + n^{-\frac{2s+k(1-2/d)}{2s+2k(1-2/d)}}.
\end{equation}
In particular, the first term vanishes if $\phi_0\in \cF_\sigma^1$.
\end{thm}

The exponent in \cref{thm:barron} always lies between $-1$ and $-1/2$, in particular not depending on the ambient dimension $d$. Note however that the hidden prefactor does depend exponentially on $d$, so that finer bounds would be needed to completely break the curse of dimensionality. 
A case of   interest for applications is when $\sigma(x,v)=\sigma(\dotp{x,v})$ for some convex function $\sigma:\R \to \R$, where $v$ is in the unit sphere $\cM=\S^{d-1}$. Activation functions of the form $\sigma(u)=\sigma_s(u)=u_+^s$ are particularly studied in the literature \citep{bach2017breaking}. 
In this case, \Cref{thm:barron} applies with $k=d-1$. One can actually use the spherical symmetry of the problem to obtain tighter rates of convergence, by replacing the dimension $d$ appearing in the rate in \Cref{thm:barron} by $d-1$, that is the dimension of the sphere. We also give a near matching minimax lower bound for this problem. 
\begin{prop}\label{prop:sphere}
Let $s\geq 2$ and let $P$ be a probability measure satisfying \ref{cond:poincare}. Write $X=RU$ the polar decomposition of $X\sim P$. Assume that the law of $U$ has a density bounded away from zero and infinity on $\S^{d-1}$. Further assume that there exists two constants $M_{\min},M_{\max}>0$ such that for all $\theta\in \S^{d-1}$, $\E[R^{2s}|U=\theta]\leq M_{\max}$ and $\E[R^{2s-2}|U=\theta]\geq M_{\min}$.   Let $\sigma=\sigma_s$ and let $\phi_0\in \cF^1_{\sigma_s}$. 
Then,
\begin{equation}
    \E[\|\nabla\hat \phi_{\cF^1_{\sigma_s}}-\nabla\phi_0\|^2_{L^2(P)}] \lesssim_{\log n,d}  n^{-\frac{2s+d-3}{2s+2(d-3)}}.
\end{equation}
Furthermore, if $d$ is odd and  $\phi_0\in \cF_\sigma^1$, it holds that 
\begin{equation}
    \inf_{\hat T} \sup_{Q\in \cQ_{\sigma_s,P}} \E[\|\hat T-\nabla\phi_0\|^2_{L^2(P)}] \gtrsim_d n^{-\frac{2s+d-1}{2s+2d-3}},
\end{equation}
where $ \cQ_{\sigma_s,P}$ is the set of measures obtained as a pushforward of $P$ by $\nabla \phi$ for $\phi\in \cF^1_{\sigma_s}$.
\end{prop}
For large $d$, the exponent in both the upper and lower bounds is $- \frac{2s + d}{2s + 2d} (1 + o_d(1))$. The conditions on $P$ in the proposition are mild, and are satisfied for example for a large class of radial distributions, including the Gaussian. \Cref{thm:barron} and \Cref{prop:sphere} are proven in \Cref{app:barron}.

As a  particular case of the above theorem, when $s=2$, we have $\sigma_s(u) = \frac{u_+^2}{2}$, so that $\sigma'_s(u) = u_+$ is the ReLu activation function. The transport maps arising from the set $\cF^1_{\sigma_s}$ are of the form $x\mapsto \int \dotp{x,v}_+ v\cdot \dd \theta(v)$, which are  shallow neural networks with a ReLu activation function. \Cref{prop:sphere} gives the upper bound $n^{-\frac{1}{2}-\frac{1}{d-1}}$, with the near-matching lower bound $n^{-\frac{1}{2}-\frac{5}{4d+2}}$. For instance, for $d=11$, the exponent in the upper bound is $0.6$, whereas the exponent in the lower bound is approximately equal to $0.612$. We conjecture that the correct minimax rate for this problem is the upper bound $n^{-\frac{1}{2} - \frac{1}{d-1}}$.

\subsection{Input convex neural networks}
 As a last example, \citet{makkuva2020optimal} propose to use input convex neural networks to estimate Brenier potentials. In our language, their estimator is exactly $\hat \phi_\cF$, where $\cF$ is a class of input convex neural networks that can be parametrized using $D$ weights (for some large $D$). Assuming that all the weights of the neural network are uniformly bounded, one can check that \eqref{eq:lip_parametric} is satisfied, so that the log-covering number at scale $h$ is of order $D\log_+(1/h)$. One can then use \Cref{thm:strongly_convex} to obtain a rate of convergence of order 
 \begin{equation}
     \E\|\nabla \hat{\phi}_\cF - \nabla \phi_0\|^2_{L^2(P)} \lesssim_{\log n,D} \inf_{\phi\in \cF_\alpha} \|\nabla\phi-\nabla\phi_0\|^2_{L^2(P)} +   n^{-1}.
 \end{equation}
 To obtain a more satisfying rate of convergence, one needs to understand the approximation properties of input convex neural networks, in order to be able to control the bias term as a function of $D$ in the above upper bound.
Moreover, the dependence on $D$ in the estimation error is also likely not tight for overparameterized networks, which have good performance even when $n \ll D$. 
Improving this bound is an important question in the theory of statistical properties of input convex neural networks, which is an attractive topic for future work.

\section{Computational aspects}\label{sec: computation}
In this section, we investigate the computational properties of the optimization problem~\eqref{eq: semidual_n} defining our estimator.
We begin by noting that the objective function $S_n(\phi)$ is \emph{convex} in $\phi$: indeed, the conjugation operator $(\cdot)^*$ satisfies
\begin{equation*}
	((1-\lambda) \phi_0 + \lambda \phi_1)^* \leq (1-\lambda) \phi_0^* + \lambda \phi_1^*
\end{equation*}
for any $\phi_0, \phi_1$ and $\lambda \in [0, 1]$
(see \cref{lem:conjugate_convex} for a short proof).
Since $P_n$ and $Q_n$ are positive measures, the functional $S_n(\phi) = \int \phi \dd P_n + \int \phi^* \dd Q_n$ is therefore convex as well.

The convexity of $S_n$ suggests that solving~\eqref{eq: semidual_n} can be computationally feasible whenever the set $\cF$ is itself convex and tractable to optimize over.
In particular, we can obtain a practical algorithm in the setting where $\cF$ is given by the convex hull of a (potentially large) parametric family, as in \cref{sec:large}.
Fixing a set of functions $\{\phi_{j}\}_{j=1}^J$, define the class
\begin{align}\label{eq:ftheta_def}
	\cF_{J} = \left\{\phi_\lambda \ \bigg| \lambda \in \Delta_J \right\} \defeq \left\{\sum_{j=1}^J \lambda_j \phi_{j} \ \bigg| \lambda \in \RR_{\geq 0}^J, \sum_{j=1}^J \lambda_j = 1\right\}\,.
\end{align}

Given i.i.d. data $X_1,\ldots,X_n \sim P$ on a compact set $\Omega$ and $Y_1,\ldots,Y_n \sim (T_0)_\sharp P$, for some optimal transport map $T_0 = \nabla \phi_0$, we wish to estimate the best approximation to $\phi_0$ in the class $\cF_{J}$ by $\widehat \phi = \phi_{\hat \lambda}$, where $\hat \lambda$ solves $\min_{\lambda \in \Delta_J} S_n(\lambda)  \defeq \int \phi_\lambda \dd P_n + \int \phi^*_\lambda \dd Q_n$, which we can write explicitly as
\begin{align}
   \min_{\lambda \in \Delta_J} S_n(\lambda) = \frac{1}{n}\sum_{i=1}^n \left(\sum_{j=1}^J \lambda_j \phi_{j}(X_i)\right) + \frac{1}{n}\sum_{k=1}^n \left(\sum_{j=1}^J \lambda_j \phi_{j}\right)^*(Y_k)\,.\label{eq: Sn_lbd}
\end{align}

It is natural to solve this problem via a first-order method, iteratively updating the parameters $\lambda$ using gradient information of $S_n(\cdot)$. Note that by the envelope theorem \cite[Theorem 1]{MilSeg02}, the gradient of $S_n(\lambda)$ can be written
\begin{align*}
    [\nabla_\lambda S_n(\lambda)]_j = \frac{1}{n}\sum_{i=1}^n \phi_{j}(X_i) - \frac{1}{n}\sum_{k=1}^n \phi_{j}(x^*_k)\,,
\end{align*}
where $x^*_k$ is the solution to the following sub-problem
\begin{align}\label{eq: conj_maximizer}
    x^*_k \defeq \argmax_{x} \left\{\langle x, Y_k \rangle - \sum_{j=1}^J\lambda_j\phi_{j}(x)\right\}\,.
\end{align}

We propose to minimize~\eqref{eq: Sn_lbd} via \emph{projected gradient descent} \citep{beck2017first}: initialize at $\lambda^{(0)} \in \Delta_J$, fix a step size $\eta > 0$, and iterate the sequence
\begin{align}\label{eq: pgd_iterates}
    \lambda^{(k+1)} \gets \text{Proj}_{\Delta_J}\bigl(\lambda^{(k)} - \eta \nabla_\lambda S_n(\lambda^{(k)})\bigr)\,,
\end{align}
where the maximum operation is computed component-wise.
Each iteration of this algorithm requires $O(J)$ arithmetic operations and  solving the $n$ auxiliary problems given by~\eqref{eq: conj_maximizer}.

The following theorem establishes the convergence of the iterates given by \cref{eq: pgd_iterates}.
A proof appears in \cref{sec: comp_proofs}.
\begin{thm}\label{thm: Sn_convex_smooth}
	Assume that the elements of $\{\phi_{j}\}_{j=1}^J$ are strongly convex over $\Omega$.
	Then $S_n(\cdot)$ is $L_J$-smooth, for some $L_J > 0$.
	Moreover, if $\eta < 1/L_J$, then the iterates of projected gradient descent \eqref{eq: pgd_iterates} satisfy
\begin{align*}
    S_n(\lambda^{(k)}) - S_n(\hat \lambda) \lesssim k^{-1}\bigl(L_J \|\lambda^{(0)} - \hat \lambda\|^2_2 + S_n(\lambda^{(0)}) - S_n(\hat \lambda)\bigr)\,.
\end{align*}
\end{thm}
Implementing \cref{eq: pgd_iterates} requires being able to efficiently compute the solution to \cref{eq: conj_maximizer}.
Projected gradient descent for $S_n(\lambda)$ is therefore tractable if we assume the existence of a \emph{conjugate oracle} for functions of the form $\sum_{j=1}^J \lambda_j\phi_{j}$.
Whether this assumption is reasonable depends on the structure of the set $\cF_{J}$.

If we assume that the base functions $\{\phi_{j}\}_{j=1}^J$ are both strongly convex and smooth, then implementing a conjugate oracle is straightforward: gradient ascent on the objective in \cref{eq: conj_maximizer} computes an $\eps$-approximate optimal solution in $O(\log(1/\eps))$ time \cite[Theorem 3.10]{Bub15}.
Combined with \cref{thm: Sn_convex_smooth} and the proof of \cref{thm:bounded}, this observation yields the following theorem.
\begin{thm}
	Let $\cF = \cF_{J}$ be as in \cref{eq:ftheta_def}. 
	If the elements of $\{\phi_{j}\}_{j=1}^J$ are smooth and strongly convex over a compact set $\Omega$, then
	a solution of \cref{eq: semidual_n} satisfying the guarantees of \cref{thm:bounded} can be found in $O_J(n^2 \log n)$ time.
\end{thm}

Though projected gradient descent can be applied to $S_n(\lambda)$ even when the functions $\{\phi_{j}\}_{j=1}^J$ are not convex, polynomial-time conjugate oracles are unlikely to exist for general function classes.
Since $\phi^*(0) = - \min_{x} \phi(x)$ for any function $\phi$, computing the convex conjugate of a function, even at a single point, is as hard as general (non-convex) optimization in the worst case.
Therefore, absent further assumptions, the time required for a conjugate oracle will typically scale exponentially in $d$. 
Implementing such an oracle can therefore be practical when the dimension is very small, but typically not in the high-dimensional setting.

The estimation procedure developed in \cite{hutter2021minimax} is an example of this phenomenon.
The authors of that work consider a large parametric class $\cF$ consisting of functions with finite wavelet expansions.
Since these functions lack any convexity properties, it is not clear how to efficiently implement a conjugate oracle.
In \cite{hutter2021minimax}, the authors use the Fast Legendre Transform \citep{lucet1997faster} to compute approximate conjugates, in running time exponential in the dimension.
Computing their estimator is feasible if $d \leq 3$, but generally infeasible otherwise.

\subsection{Hardness results}\label{sec: hardness}
The forgoing discussion suggests that it may not always be possible to minimize $S_n(\phi)$ over arbitrary families $\cF$ in polynomial time.
In this section, we give evidence that this deficiency is, in fact, fundamental and arises from a \emph{computational-statistical gap}.
Concretely, we show that, under weak computational assumptions, no polynomial-time estimator can achieve the rates presented in \cref{thm:strongly_convex} for a general family $\cF$.
This result implies that the computational challenges discussed above are intrinsic to the problem, and cannot be avoided by adopting a different estimation strategy.

Our computational lower bound is based on a reduction from the \emph{Continuous Learning With Errors (CLWE)} problem.
The pathbreaking work of \citet{BruRegSon21} showed that CLWE is as hard as the worst-case \emph{Gap Shortest Vector Problem (GapSVP)}, a notoriously difficult problem in quantum computing.
Informally, this is a decision problem whose input is a $d$-dimensional lattice\footnote{A discrete additive subgroup of $\RR^d$} $L$ and a threshold $\tau > 0$, and whose output is ``YES'' if there exists a vector in $L$ whose length is at most $\tau$, and ``NO'' if every vector in $L$ has length at least $\Omega(d \tau)$.
It is widely conjectured that no polynomial-time algorithm for worst-case GapSVP exists, \emph{even for quantum computers}~\citep{Reg05,Reg10}.
In fact, this conjecture is so universally believed that GapSVP has been adopted as the standard for post-quantum cryptography, that is, for cryptographic protocols that are robust even against adversaries with access to quantum computers.
By contrast, many problems that are hard for classical computers and form the backbone of classical cryptographic systems, such as integer factoring, are known to possess polynomial-time quantum algorithms \citep{PQC09}.

The following theorem shows that estimating optimal transport maps in general function spaces is as hard as CLWE, and therefore as hard as worst-case GapSVP.
This implies that, absent additional assumptions, no polynomial-time estimators achieving the rates in \cref{thm:strongly_convex} are likely to exist, even if we permit the statistician to use quantum computers.
The class of potentials for which we show our hardness result is the simple ``spiked'' class defined in \cref{sec: lowdim}, with $k = 1$ and the base class $\cF$ consisting of smooth and strongly convex functions:
\begin{equation*}
	\cF_{\mathrm{spiked}, 1}(c, C) = \left\{x \mapsto \phi(\langle u, x\rangle) + \frac 12 \|x - \langle u, x \rangle x \|^2: \phi: \RR \to \RR, c \leq \phi''(x) \leq C , \|u\| = 1\right\}\,,
\end{equation*}
for universal constants $c, C > 0$.
As indicated in \cref{prop:spiked}, for any $\eta \geq 2$, this class satisfies
\begin{equation*}
	\log \cN(h, \cF_{\mathrm{spiked}, 1}, L^\infty(\dotp{\ \cdot \ }^{-\eta})) \lesssim_{\log_+(1/h),\delta} d h^{- \frac 12}\,,
\end{equation*}
so \cref{thm:strongly_convex} implies that for any constants $c, C$, if $\phi_0 \in \cF_{\mathrm{spiked}, 1}(c, C)$, then $\nabla \phi_0$ can be estimated at the rate $n^{-4/5}$, up to logarithmic factors.
In particular, if the estimator $\hat \phi_\cF$ could be computed in polynomial in $n$ and $d$, then it would be possible to estimate the optimal map to accuracy $O(1/\sqrt d)$ in polynomial time.
The following theorem shows that any such estimator would violate the conjectured hardness of the CLWE and GapSVP problems.
\begin{thm}\label{thm:hard}
	There exist positive constants $c, C, c'$ such that the following holds.
	Let $X_1, \dots, X_n \sim \cN(0, I_d)$ and $Y_1, \dots, Y_n \sim (T_0)_\sharp \cN(0, I_d)$, where $T_0= \nabla \phi_0$ for some $\phi_0 \in \cF_{\mathrm{spiked}, 1}(c, C)$ and where $n = \mathrm{poly}(d)$.
	If there exists a polynomial-time estimator $\hat \phi$ satisfying $\E \|\nabla \hat \phi - \nabla \phi_0\| \leq c'/\sqrt{d}$, then there exists an efficient algorithm for CLWE and an efficient quantum algorithm for GapSVP.
\end{thm}
As mentioned above, this theorem indicates that estimating transport maps from the class $\cF_{\mathrm{spiked}, 1}$ is as hard as worst-case lattice problem, and therefore widely conjectured to be intractable even for quantum computers. A proof of \cref{thm:hard} appears in \cref{sec: comp_proofs}.

\section{Proofs}\label{sec: proofs}
\subsection{Empirical risk minimizers with convex criteria}\label{sec:_one_shot}

As already stated, the estimator $\hat \phi_\cF$ is nothing but an empirical risk minimizer. As such, its risk can be bounded using standard theorems such as Massart's analysis for the risk of an empirical risk minimizer \cite[Theorem 8.3]{massart2007concentration}. Let $\xi_1,\dots,\xi_n$ be a sample of i.i.d.~observations from some law $\PP$. In full generality, Massart's theorem concerns the minimization of a criterion of the form 
\[ a\in \AA \mapsto \int \gamma(a,\xi)\dd \PP(\xi), \]
where $\AA$ is a vector space endowed with some pseudo-norm $\|\cdot\|$. Let $a^*$ be a minimizer of this functional and define the excess of risk $\ell(a,a^*)=\int \gamma(a,\xi)\dd \PP(\xi)-\!\int \gamma(a^*,\xi)\dd \PP(\xi) $. Let 
\[ \overline\gamma_n(a) = \gamma_n(a)-\int \gamma(a,\xi)\dd \PP(\xi) = \frac{1}{n} \sum_{i=1}^n \gamma(a,\xi_i) -\int \gamma(a,\xi)\dd \PP(\xi)\]
be the associated centered empirical process. 
The minimizer $a^*$ is approximated by the empirical risk minimizer $\hat a \in \argmin_{a\in \cA} \overline\gamma_n(a)$, where $\cA\subseteq \AA$ is a class of candidates. Let $\overline a \in \cA$ be a ``good candidate'' so that $\ell(\overline a,a^*)$ is small (e.g.,  let $\overline a\in \argmin_{a\in\cA}\ell(a,a^*)$). Then, it is known that estimation is possible (i) under a control of the excess of risk around its minimum of the form
\begin{equation}
   \forall a \in \AA,\ \|a-a^*\| \leq w(\ell(a,a^*))
\end{equation}
for some appropriate function $w:[0,\infty)\to [0,\infty)$, and (ii) under a control of local fluctuations given by a bound on
\begin{equation}\label{eq:control_local_fluct_abstract}
    \sqrt{n} \E\left[ \sup_{a\in \cA,\ \|a-\overline a\|\leq \tau} |\overline\gamma_n(a)-\overline\gamma_n(\overline a)| \right].
\end{equation}
Using this general theorem would be enough for obtaining the right bound in the bounded case (\Cref{thm:bounded}), where $\AA$ is equal to the set of twice differentiable functions $\phi$ with $\nabla\phi\in L^2(P)$, that we endow with the pseudo-norm $\phi\mapsto \|\nabla\phi\|_{L^2(P)}$, and $\cA=\cF$ is the class of candidate potentials. 
This strategy, however, needs to be refined for our second scenario, where the measure $P$ is not bounded. Indeed, in that case, as the proof will show, a tight control of the local fluctuations \eqref{eq:control_local_fluct_abstract} is only possible when the potentials $a=\phi$ in the supremum are strongly convex. As such, to apply Massart's theorem, we need to restrict ourselves to families of candidate potentials containing only strongly convex functions (as it is the case in \cite{hutter2021minimax}). 
Such a limitation is not natural, as most families of approximation functions (neural networks, wavelets, Fourier functions) are not convex, hindering the applicability of our theorem.

To overcome this issue, we introduce an auxiliary norm $\vertiii{\cdot}$, and make the key observation that, under a convexity assumption of  the  criterion $\gamma$, it is sufficient to control the local fluctuations $ |\overline\gamma_n(a)-\overline\gamma_n(\overline a)|$ for $\|a-\overline a\|\leq \tau$ while  $\vertiii{a-\overline a}$  is also arbitrarily small. In our setting, $\vertiii{\cdot}$ will be the $C^2$-norm, while $\overline a=\overline\phi$ will be an $\alpha$-strongly convex potential attaining the infimum $\inf_{\phi\in\cF_\alpha}(S(\phi)-S(\phi_0))$. In particular, any potential $\phi\in \cF$ satisfying $\vertiii{\phi-\overline\phi} \leq \alpha/2$ is $(\alpha/2)$-strongly convex. The supremum now being restricted to strongly convex functions can be tightly controlled.

We  state a general theorem for controlling the risk of an empirical risk minimizer with convex criterion. Besides introducing the auxiliary norm $\vertiii{\cdot}$, this theorem is nothing new, and follows from van de Geer's ``one-shot'' localization technique \citep{van2002m}. Let us remark that  \citet{hutter2021minimax} already successfully used this method for dealing with optimal transport map estimation with wavelet approximations,  with the key difference that they  only treat the case of a class of \textit{strongly convex} candidate potentials $\cF$ having a finite wavelet expansions.

\begin{thm}\label{thm:abstract_one_shot}
Let $\xi_1,\dots,\xi_n$ be independent observations taking their values in some measurable space $\Xi$ and with common distribution $\PP$ and let $(\AA,\|\cdot\|)$ be a pseudo-normed vector space. Let $\gamma: \AA\times \Xi\to (-\infty,+\infty]$ be a function such that for every $a\in \AA$, $\xi\mapsto \gamma(a,\xi)$ is measurable and such that for  $\PP$-almost all $\xi\in \Xi$, $a\in \AA \mapsto \gamma(a,\xi)$ is a convex function. Assume that there exists some minimizer $a^*$ of $a\in \AA\mapsto \int \gamma(a,\xi)\dd \PP(\xi)$, and denote by $\ell(a,a^*)$ the nonnegative quantity 
\[ \ell(a,a^*)= \int \gamma(a,\xi)\dd \PP(\xi) -\int \gamma(a^*,\xi)\dd \PP(\xi).\]
 Let $\overline \gamma_n$ be the centered empirical process defined by
\[ \overline\gamma_n(a) = \gamma_n(a)-\int \gamma(a,\xi)\dd \PP(\xi) = \frac{1}{n} \sum_{i=1}^n \gamma(a,\xi_i) -\int \gamma(a,\xi)\dd \PP(\xi),\quad \text{for every $a\in \AA$}.\]
Assume that there exists a nondecreasing function $w:[0,\infty)\to [0,\infty)$ such that
\begin{equation}\label{eq:choice_w}
    \forall a\in \AA,\ \|a-a^*\|\leq w(\ell(a,a^*)).
\end{equation} 
Let $\cA\subseteq \AA$ be a model, and consider an estimator $\hat a\in \argmin_{a\in \cA} \gamma_n(a)$. 

We assume that there is another pseudo-norm $\vertiii{\cdot}$ on $\AA$ such that the set $\cA$ is bounded by some constant $B>0$ for this norm. Let $0<r\leq B$ and let $\overline a \in \cA$ be any point. Let $\overline\cA$ be the set of points $a$ belonging to a segment linking $\overline a$ to some element in $\cA$, with $\vertiii{a-\overline a}\leq r$.  Assume that for some $\tau>0$
\begin{equation}\label{eq:control_fluct_abstract}
\sqrt{n}\E\left[ \sup_{a\in \overline\cA, \|a-\overline a \|\leq \tau}|\overline \gamma_n(a)-\overline\gamma_n(\overline a)|\right] \leq \Psi(\tau).
\end{equation}
for some function $\Psi$.  Let $F$ be an envelope function, that is a measurable function such that for all $a\in \overline \cA$, $\xi\in \Xi$,$|\gamma(a,\xi)|\leq F(\xi)$.
Assume that the random variable $F(\xi)$ for $\xi\sim \PP$ belongs to the Orlicz space of index $p\in (0,1]$, with Orlicz norm bounded by $L$. 

Then, with probability $1-e^{-y}$,
\[ \|\hat a-a^*\|\leq w(\ell(\overline a,a^*)) + \pran{\frac{2B}{r}-1} \tau\]
whenever $\tau$ satisfies
\begin{equation}\label{eq:abstract_tau}
    \tau \geq 4w\Big(\ell(\overline a,a^*)+ \kappa_p y^{1/p}\pran{\Psi(\tau)n^{-1/2} + L(\log n)^p n^{-1}}\Big)
\end{equation} 
for some constant $\kappa_p$ depending only on $p$.
\end{thm}

\begin{proof}
Let $y>0$ and let $\tau$ be as in the statement of the theorem. 
    Let $\hat a_t=t\hat a + (1-t)\overline a$, where $t= \frac{\tau}{\tau+\|\hat a-\overline a\|}$. Note that 
    \[ \|\hat a_t-\overline a\| = t\|\hat a-\overline a\|=  \frac{\tau \|\hat a-\overline a\|}{\tau+\|\hat a-\overline a\|}\leq \tau.\]
    If $t\geq r/(2B)$, then $\|\hat a-\overline a\|\leq \pran{\frac{2B}{r}-1} \tau$, and the conclusion holds. Otherwise, we have $t< r/(2B)$, which implies that 
    \begin{equation}
    \vertiii{\hat a_t-\overline a}=t\vertiii{\hat a-\overline a} \leq r.
    \end{equation}
      By convexity of $\gamma$ and as $\gamma_n(\hat a)=  \min_{a\in \cA}\gamma_n(a)$, it holds that
    \[ \gamma_n(\hat a_t) \leq t\gamma_n(\hat a) + (1-t)\gamma_n(\overline a) \leq \gamma_n(\overline a).\]
    Hence,
    \begin{align*}
        \ell(\hat a_t, a^*) &= \ell(\overline a,a^*) + \gamma_n(\hat a_t)-\gamma_n(\overline a) + \overline \gamma_n(\overline a) - \overline \gamma_n(\hat a_t) \\
        &\leq \ell(\overline a,a^*)  + \overline \gamma_n(\overline a) - \overline \gamma_n(\hat a_t).
    \end{align*}
    Let  $Z_n=  \sup_{a\in \overline\cA, \|a-\overline a\|\leq \tau}|\overline \gamma_n(a)-\overline\gamma_n(\overline a)|$, so that $\overline \gamma_n(\overline a) - \overline \gamma_n(\hat a_t) \leq Z_n$. As $w$ is nondecreasing, we have
    \begin{align*}
         \frac{\tau \|\hat a-\overline a\|}{\tau+\|\hat a-\overline a\|}&= \|\hat a_t-\overline a\| \leq \|\hat a_t-a^*\|+\|\overline a-a^*\| \\
         &\leq w(\ell(\hat a_t,a^*)) + w(\ell(\overline a,a^*)) \\
         &\leq w(\ell(\overline a,a^*) +\overline \gamma_n(\overline a) - \overline \gamma(\hat a_t) ) + w(\ell(\overline a,a^*))\\
         &\leq 2w(\ell(\overline a,a^*)+Z_n).
    \end{align*}
    Thus, if $w(\ell(\overline a,a^*)+Z_n) \leq \tau/4$, then $\|\hat a-\overline a\|\leq \tau$. 
     According to \cite[Theorem 2.14.23]{vaart2023empirical},
    \begin{align*}
        \PP(\|\hat a -\overline a\| >\tau \text{ and } t<r/(2B))&\leq \PP(w(\ell(\overline a,a^*)+Z_n)> \tau/4) \\
        &\leq \PP( Z_n \geq \kappa_p y^{1/p}\pran{\frac{\Psi(\tau)}{\sqrt{n}} + \frac{(\log n)^p L}{n}}) \\
        &\leq \PP( Z_n\geq \kappa_p y^{1/p}\pran{\E[Z_n]+ \frac{(\log n)^p L}{n}}) \\
        &\leq \exp\pran{- y}
    \end{align*}
    for a certain constant $\kappa_p$. Hence, with probability $1-e^{-y}$, if $ t<r/(2B)$, then
    \[ \|\hat a -\overline a\| \leq  \tau \leq \pran{\frac{2B}{r}-1} \tau \]
     and  otherwise, if $t\geq r/(2B)$, we have already proved that $\|\hat a -\overline a\| \leq \pran{\frac{2B}{r}-1} \tau$.  Hence, by the triangle inequality, with probability $1-e^{-y}$,
    \begin{align*}
        \|\hat a-a^*\|&\leq\|\overline a-a^*\|+ \|\overline a-\hat a\| \\
        &\leq w(\ell(\overline a,a^*)) + \pran{\frac{2B}{r}-1} \tau. \qedhere
    \end{align*}
\end{proof}

Let us  apply \Cref{thm:abstract_one_shot} to prove \Cref{thm:bounded} and \Cref{thm:strongly_convex}. Let $\Omega=B(0,2R)$ for $R$ possibly equal to ${+\infty}$. 
We consider $\AA$ the set of functions $\phi$ that are $(\beta,a)$-smooth on $\Omega$, with $\phi(0)=0$ and such that $\|\nabla\phi\|_{L^2(P)}<\infty$. We endow $\AA$ with the pseudo-norm $\|\phi\|=\|\nabla\phi\|_{L^2(P)}$. We let $\Xi=\supp(P)\times \supp(Q)$, $\PP=P\otimes Q$, and $\xi_i=(X_i,Y_i)$. For $\phi\in\AA$ and $\xi\in \R^d\times \R^d$, we let $\gamma(\phi,\xi)=\phi(x)+\phi^*(y)$, which is convex as a function of $\phi$. By assumption, the minimum of $\int \gamma(\phi,\xi)\dd \PP(\xi) = S(\phi)$ is $\phi_0$, with $\ell(\phi,\phi_0)=S(\phi)-S(\phi_0)$.

We  can  assume without loss of generality that $\|\nabla\phi(0)\|$ is uniformly bounded over $\phi\in\cF$ by some constant $K>0$: this is clear in the bounded case, and we explain in \Cref{lem:assume_gradient_bounded} why this is possible in the strongly convex case. Hence, the constant in \Cref{prop: stab_bound} (which depends on $\|\nabla\phi(0)\|$) can be chosen  uniformly over $\cF$. For $b$, the exponent appearing in \Cref{prop: stab_bound}, let 
\[
w(x)= C_0\begin{cases}
    \sqrt{x}\log(1/x)^{b/2} &\text{ if $x\leq c$} \\
    \sqrt{x}& \text{ else}
\end{cases} 
\]
with a choice of $c>0$ such that the function is nondecreasing. According to \Cref{prop: stab_bound}, the function $w$ satisfies \eqref{eq:choice_w} for some constant $C_0>0$ (depending on $K$). We let $\cF=\cA$ be the class of candidate potentials. In the bounded case, let $\overline a=\overline \phi$ that attains $\inf_{\phi\in\cF}S(\phi)$ (or an approximate minimizer should the infimum be not attained). Likewise, in the strongly convex case, we consider a (potentially approximate) minimizer  $\overline a=\overline \phi$ of $\inf_{\phi\in\cF_\alpha}S(\phi)$. Let $\ell^\star= S(\overline\phi)-S(\phi_0)$. We define the auxiliary pseudo-norm $\vertiii{\phi} = \sup_{x\in \Omega}\dotp{x}^{-a}\op{\nabla^2 \phi(x)}$. Condition \ref{cond:smooth} asserts that $\cF$ is bounded by $B=\beta$ for $\vertiii{\cdot}$.

In both scenarios, we will be able to (i) find an envelope function $F$ with a finite Orlicz norm, and (ii) an appropriate function $\Psi(\tau)$ in \eqref{eq:control_fluct_abstract}  for some parameter $r$. Namely, when the parameter $\gamma\geq 0$ appearing in \ref{cond:covering_bounded} and \ref{cond:covering_convex} is below $2$, we will give a bound of the form 
\begin{equation}\label{eq:the_good_psi}
\forall \tau\geq \sqrt{\ell^\star},\ \Psi(\tau) \lesssim_{\log_+(1/\tau),\log(n)} \tau^{1-\gamma/2}+  n^{\frac 1 2- \frac 1 \gamma} +n^{-\frac 1 2}
\end{equation}
(where by convention, we let $n^{\frac 1 2- \frac 1 \gamma}=0$ if $\gamma=0$). 
\Cref{thm:abstract_one_shot} then applies (with $y=\log(1/n)$). Indeed, one can check that the condition \eqref{eq:abstract_tau} holds for $\tau$ of the form $C_1(n^{-\frac{1}{2+\gamma}}\log(n)^q + w(\ell^\star))$ for some constants $C_1,q\geq 0$. Then, with probability $1-n^{-1}$,
\begin{equation}
    \|\nabla\hat\phi_{\cF}-\nabla\phi_0\|_{L^2(P)} \leq w(S(\overline\phi)-S(\phi_0)) +C_2 n^{-\frac{1}{2+\gamma}} \log(n)^q,
\end{equation}
concluding the proof.

When the parameter $\gamma$ is larger than $2$, we will apply \Cref{thm:abstract_one_shot} for controls of the form
\begin{equation}\label{eq:the_good_psi_gamma_large}
\forall \tau\geq \sqrt{\ell^\star},\ \Psi(\tau) \lesssim_{\log(n)}n^{\frac{1}{2}-\frac{1}{\gamma}}.
\end{equation}
When such a control holds, one can pick $\tau$ of the form $C_3(n^{-\frac{1}{2\gamma}}\log(n)^q+ w(\ell^\star))$ for some constants $C_3,q\geq 0$ to obtain a final bound of the form
\begin{equation}
    \|\nabla\hat\phi_{\cF}-\nabla\phi_0\|_{L^2(P)} \leq w(S(\overline\phi)-S(\phi_0)) + C_4n^{-\frac{1}{2 \gamma}} \log(n)^q,
\end{equation}
concluding the proof in the case $\gamma\geq 2$.

At last, let us mention that in the strongly convex case, the bias $\ell^\star$ is bounded up to constants by $\inf_{\phi\in\cF_\alpha}\|\nabla\phi-\nabla\phi_0\|^2_{L^2(P)}$ according to \Cref{prop: stab_bound}, so that $\ell^\star$ can be replaced by this quantity in \Cref{thm:strongly_convex}.

To summarize, all that is left to do is finding an appropriate envelope function $F$ and an appropriate bound $\Psi(\tau)$ on \eqref{eq:control_fluct_abstract}, both in the bounded case and 
in the strongly convex case. To do so, we will use the following proposition to bound the expectation of the supremum of an empirical process.
\begin{prop}\label{prop:van_wellner}
Let $\PP$ be a probability distribution on some Euclidean space such that $\xi\sim\PP$ belongs to the Orlicz space of index $p\in (0,1]$. 
    Let $\cG$ be a  class of functions defined on $\supp(\PP)$ with $\sup_{g\in\cG} \|g\|_{L^2(\PP)}\leq \sigma$ and $\sup_{g\in \cG} \|g\|_{L^\infty(\dotp{\ \cdot \ }^{-\eta})} \leq M$  for some $q\geq 0$. Let $H_2$ (resp.~$H_\infty$) be a decreasing function  that diverges at $h=0$, and such that for all $h>0$, $\log_+ \cN(h,\cG,L^2)\leq H_2(h)$ (resp.~$\log_+ \cN(h,\cG,L^\infty(\dotp{\ \cdot \ }^{-\eta}))\leq H_\infty(h)$). 
     There exists $C$ depending on $\eta$, $p$, and  the $p$-Orlicz norm of $\xi\sim\PP$ such that the following holds. For $\eps \in (0,M)$, let $L=(C\log_+(1/\eps))^{1/p}$. Let $K$ be such that $H_\infty(K\eps) \leq H_\infty(\eps)/2$ 
     and let $\tilde\eps<\sigma$ be such that $H_2(\tilde \eps )\geq  H_\infty(\eps)$ (with possibly $\tilde\eps=0$). Then
    \begin{equation}\label{eq:van_wellner_mixed_entropy}
    \begin{split}
        \sqrt{n}\E \sup_{g\in \cG} |(\PP_n-\PP)(g)| &\lesssim \sqrt{n}(K+M)\eps\dotp{L}^{\eta}+  \int_{\tilde\eps/2}^\sigma \sqrt{H_2(h)
        }\dd h  + \frac{\dotp{L}^{\eta}}{\sqrt{n}}\int_{\eps/2}^{M}H_\infty(h)\dd h
        \end{split}
    \end{equation}
    up to a numerical constant.  The same inequality holds for $\eps=0$ with $L=0$.
\end{prop}

This proposition is a generalization of \cite[Theorem 2.14.21]{vaart2023empirical} to classes of unbounded functions, with a truncation of both entropy integrals. 
A proof of \Cref{prop:van_wellner} is given in \Cref{sec: properties_covering_bracketing}.

\begin{rmk}
    \Cref{thm:abstract_one_shot} can be used with minor modifications to obtain high probability bounds on $\|\nabla\hat\phi_\cF-\nabla\phi_0\|_{L^2(P)}$. We leave the details to the interested reader.
\end{rmk}

\subsection{Bounded case}\label{sec: main_bounded_case}

Let us assume that \ref{cond:covering_bounded} and \ref{cond:bounded} hold. Proofs of additional lemmas in this section are given in \Cref{app: lemmas_bounded}. Let $r=2B=2\beta$ in \Cref{thm:abstract_one_shot}. Define $\overline\cF$ as in \Cref{thm:abstract_one_shot} by
\[ \overline\cF=\{\phi_t=t\phi+(1-t)\overline\phi:\ \phi\in\cF, t\in [0,1] \text{ and }\vertiii{\phi_t-\overline\phi}\leq 2B\}.\]
Note that the last constraint in the definition of the set is actually always satisfied as $\vertiii{\phi_t-\overline\phi}\leq \vertiii{\phi}+\vertiii{\overline\phi}\leq 2B$. 
We first show that the constant function equal to $8R^2$ is an envelope function of $\overline \cF$.

\begin{lem}[Envelope for $\overline \cF$, bounded case]\label{lem:bounded_envelope}
Let $\phi\in \overline\cF$ and let $\xi=(x,y)\in \Xi$. Then,  $|\gamma(\phi,\xi)| \leq  8 R^2$.
\end{lem}

The constant envelope function belongs to the Orlicz space of index $p=1$. Let $\tau>0$. We apply \Cref{prop:van_wellner} to the set $\cG=\{\gamma(\phi,\cdot)-\gamma(\overline\phi,\cdot):\ \phi\in \overline\cF,\ \|\nabla\phi-\nabla\overline\phi\|_{L^2(P)} \leq \tau\}$, for which we require the next two lemmas.

\begin{lem}[Uniform second moment on $\cG$, bounded case]\label{lem:bounded_second_moment}
 There exists $C\geq 0$ such that for all $\tau\geq \sqrt{\ell^\star}$,
\begin{equation}
    \sup_{g\in\cG} \|g\|_{L^2(\PP)} \leq \sigma\defeq C \tau.
\end{equation}
\end{lem}
\begin{lem}[Covering numbers of $\cG$, bounded case]\label{lem:bounded_brackets}
For $h>0$, it holds that
\begin{equation}
    \log \cN(h,\cG,L^\infty) \leq \log \cN(h/4,\cF,L^\infty) + C\log_+(1/h).
\end{equation}
\end{lem}
Assume first that $\gamma\in [0,1)$ and let $q=0$. We define for some $m,c\geq 0$ the functions $H_\infty(h)=H_2(h) = 1+ c\log_+(1/h)^{m} h^{-\gamma}$ for all $h>0$. According to condition \ref{cond:covering_bounded}, these functions are valid upper bounds on the log-covering numbers in \Cref{prop:van_wellner}. According to \Cref{lem:bounded_envelope}, the constant function $16R^2$ is an upper bound on the $L^\infty$-norm of functions in $\cG$. \Cref{prop:van_wellner} for $\eps=0$ implies 
\begin{align*}
   &\sqrt{n} \E \sup_{g\in \cG} |(\PP_n-\PP)(g)| \lesssim \int_0^{C\tau} \sqrt{H_2(h)}\dd h + \frac{1}{\sqrt{n}} \int_0^{16R^2} H_\infty(h)\dd h\\
   &\lesssim_{\log_+(1/\tau)}   \tau^{1-\gamma/2} + n^{-1/2} \lesssim_{\log_+(1/\tau),\log(n)}  \tau^{1-\gamma/2}+ n^{\frac{1}{2}-\frac{1}{\gamma}} + n^{-\frac 12}.
\end{align*}    
This defines a function $\Psi(\tau)$ that is bounded as in \eqref{eq:the_good_psi}. Applying the argument explained at the end of \Cref{sec:_one_shot} gives the conclusion. When $\gamma\in [1,2)$, the function $H_\infty$ is no longer integrable, and we apply \Cref{prop:van_wellner} with $\eps=n^{-1/\gamma}$ and $\tilde\eps=0$. Remark that the parameter $K$ in \Cref{prop:van_wellner} can be chosen independently of $\eps$. We obtain
\begin{align*}
   \sqrt{n} \E \sup_{g\in \cG} &|(\PP_n-\PP)(g)| \lesssim \sqrt{n}\eps+ \int_0^{C\tau} \sqrt{H_2(h)}\dd h + \frac{1}{\sqrt{n}} \int_{\eps/2}^{16R^2} H_\infty(h)\dd h\\
   &\lesssim_{\log_+(1/\tau),\log(n)}  \tau^{1-\gamma/2}+n^{\frac{1}{2}-\frac{1}{\gamma}} \lesssim_{\log_+(1/\tau),\log(n)}  \tau^{1-\gamma/2}+ n^{\frac{1}{2}-\frac{1}{\gamma}} + n^{-\frac 12}.
\end{align*}    
Once again, we conclude as explained in \Cref{sec:_one_shot}.

 When $\gamma\geq 2$, we let once again $\eps=n^{-1/\gamma}$ and apply \Cref{prop:van_wellner} with $\tilde\eps=\eps$ to obtain
\begin{equation}\label{eq:bound_gamma_large}
    \begin{split}
        \sqrt{n} \E \sup_{g\in \cG} |(\PP_n-\PP)(g)| &\lesssim \sqrt{n}\eps+ \int_{\eps/2}^{C\tau} \sqrt{H_2(h)}\dd h + \frac{1}{\sqrt{n}} \int_{\eps/2}^{16R^2} H_\infty(h)\dd h \lesssim_{\log n}  n^{\frac{1}{2}-\frac{1}{\gamma}},
    \end{split}
\end{equation}
   and we conclude thanks to the considerations found after \Cref{eq:the_good_psi_gamma_large}.

\subsection{Strongly convex case}\label{sec: main_strongly_case}
In the strongly convex case, we pick $\overline \phi= \argmin_{\phi\in \cF_\alpha}S(\phi)$ (we suppose that $\overline \phi$ exists, for otherwise we use a standard approximation scheme). The proofs of the additional lemmas in this section are found in \Cref{app: lemmas_unbounded}.

As in the bounded case, we first give an envelope function of $\cF$.

\begin{lem}[Envelope for $\cF$, strongly convex case]\label{lem:strict_envelope}
Let $\phi\in \cF$ and let $\xi=(x,y)\in \Xi$. Then, $|\gamma(\phi,\xi)|\leq  M\dotp{\xi}^{a+2}$, where $C$ depends on $\alpha$, $\beta$, $a$ and on $K$ a uniform bound on $\|\nabla\phi(0)\|$ over $\phi\in \cF$.
\end{lem}
Recall that $P$ satisfying a Poincaré inequality implies that it has an exponential tail. As $\phi_0$ is $\beta$-smooth with exponent $a$ and $Q=(\nabla\phi_0)_\sharp P$, this implies that $Y\sim Q$ belongs to the Orlicz space of exponent $\min(1,1/a)$. In particular, the random variable $\dotp{\xi}^{a+2}$ where $\xi\sim \PP$ belongs to the Orlicz space of exponent $\min(1,1/a)/(2+a)$. Hence, \Cref{thm:abstract_one_shot} can be applied. It remains to bound the expectation \eqref{eq:control_fluct_abstract} using \Cref{prop:van_wellner}.

Recall that $\vertiii{\phi} = \sup_{x\in \R^d} \dotp{x}^{-1}\op{\nabla^2\phi(x)}$.  Let $r=\alpha/2$ in \Cref{thm:abstract_one_shot}. Define $\overline\cF$ as in \Cref{thm:abstract_one_shot} by
\[ \overline\cF=\{\phi_t=t\phi+(1-t)\overline\phi:\ \phi\in\cF, t\in [0,1] \text{ and }\vertiii{\phi_t-\overline\phi}\leq r\}.\]
As $\overline\phi$ is $(\alpha,a)$-strongly convex, the constraint $\vertiii{\phi_t-\overline\phi}\leq \alpha/2$ entails that $\phi_t$ is $(\alpha/2,a)$-strongly convex. Let $\tau>0$. We apply \Cref{prop:van_wellner} to the set $\cG=\{\gamma(\phi,\cdot)-\gamma(\overline\phi,\cdot):\ \phi\in \overline\cF, \|\nabla\phi-\nabla\overline\phi\|_{L^2(P)} \leq \tau\}$. If the radius $R$ in \ref{cond:bounded} is finite, we can apply \Cref{lem:bounded_second_moment} and \Cref{lem:bounded_brackets} to apply \Cref{prop:van_wellner}. When $R=\infty$, the potentials are not bounded and we require the two following lemmas.

 \begin{lem}[Uniform second moment on $\cG$, strongly convex case]\label{lem:strict_second_moment}
 There exists $C,c\geq 0$ such that for all $\tau\geq \sqrt{\ell^\star}$,
\begin{equation}
    \sup_{g\in\cG} \|g\|_{L^2(\PP)} \leq \sigma\defeq C \log_+(1/\tau)^c  \tau.
\end{equation}
\end{lem}

\begin{lem}[Covering numbers of $\cG$, strongly convex case]\label{lem:strict_brackets}
There exists $c_0,c_1,C>0$ such that for $h>0$,
\begin{equation}
\begin{split}
    \log \cN(h,\cG,L^2(\PP)) &\leq  \log \cN(c_0h,\cG,L^\infty(\dotp{\ \cdot \ }^{-\eta}) \\
    &\leq \log \cN(c_1h,\cF,L^\infty(\dotp{\ \cdot \ }^{-\eta})) + C\log_+(1/h).
    \end{split}
\end{equation}
\end{lem}
Having these two lemmas at hand, the proof can be concluded exactly as in the bounded case, with the same estimates holding up to logarithmic factors. This yields the first rate of convergence in \Cref{thm:strongly_convex}. Under the additional condition \ref{cond:covering_gradient}, we are able to improve the rates by using better bounds on the $L^2$-covering numbers of the set $\cG$. 

\begin{lem}[Covering numbers of $\cG$ under condition \ref{cond:covering_gradient}]\label{lem:improved_covering}
Under the additional condition \ref{cond:covering_gradient},  it holds that for every $h>0$,
\begin{equation}
    \log \cN(h,\cG,L^2(\PP)) \lesssim_{\log_+(1/h),\log_+(1/\tau)} (h/\tau)^{-m}\,. 
\end{equation}
\end{lem}

Assume  that $\gamma\in [0,2)$ and let $\gamma'=2-m\frac{2-\gamma}{m-\gamma}\in [0,2)$. According to the previous considerations, we can apply \Cref{prop:van_wellner} to $\cG$ with $p=\min(1,1/a)/(2+a)$,  $H_\infty(h)$ of the form $1+c_0\log_+(1/h))^{c_1} h^{-\gamma}$ for all $h>0$ and some $c_0,c_1>0$. We let 
\[ H_2(h) = \min(H_\infty(h), c_2\log_+(1/h)^{c_3}\log_+(1/\tau)^{c_4} (h/\tau)^{-m}),\]
which is a valid upper bound on $\log_+\cN(h,\cG,L^2(\PP))$ for some choice of constants $c_3,c_4\geq 0$ according to Lemmas \ref{lem:strict_brackets} and \ref{lem:improved_covering}. We also let $\eta\geq a+2$ be such that Condition \ref{cond:covering_convex} holds, let $\eps=n^{-1/\gamma}$ and $\tilde \eps=0$. We obtain
\begin{align*}
  & \sqrt{n} \E \sup_{g\in \cG} |(\PP_n-\PP)(g)| \lesssim_{\log(n)} \sqrt{n}\eps + \int_0^{\sigma} \sqrt{H_2(h)} \dd h + \frac{1}{\sqrt{n}} \int_{\eps/2}^{M} H_\infty(h)\dd h \\
    &\lesssim_{\log n}n^{\frac{1}{2}-\frac{1}{\gamma}} +  \int_0^{\tau^{\frac{m}{m-\gamma}}} (\log_+(1/h))^{c_3/2}h^{-\gamma/2}\dd h   + \int_{\tau^{\frac{m}{m-\gamma}}}^\sigma (\log_+(1/h))^{c_3/2}\pran{\frac h \tau}^{-m/2}\dd h.\\
    &\lesssim_{\log n, \log_+(1/\tau)}  \tau^{\frac{m}{2} \frac{2-\gamma}{m-\gamma}} +  n^{\frac{1}{2}-\frac{1}{\gamma}} \lesssim_{\log n, \log_+(1/\tau)}  \tau^{1-\gamma'/2}+  n^{\frac{1}{2}-\frac{1}{\gamma}} +  n^{-\frac 1 2}.
\end{align*}   
This defines a function $\Psi(\tau)$ of the form given in \eqref{eq:the_good_psi}. Applying the argument explained at the end of \Cref{sec:_one_shot} gives the conclusion.

\appendix

\section{Orlicz spaces}\label{app: orlicz_app}

We gather in this section results on random variables in Orlicz spaces relevant to us. Let $p\in (0,1]$.  Recall that we let $\psi_p:[0,+\infty)\to [0,+\infty)$ be a given convex increasing function with $\psi_p(0)=0$ and $\psi_p(x)=e^{x^p}-1$ for $x\geq c_p$, where the constant $c_p$ is chosen so that such a convex increasing function exists. We say that a random variable $X$ belongs to the Orlicz space of exponent $p$ if there exists $c>0$ with $\E[\psi_p(\|X\|/c)]\leq 1$. The smallest number $c$ satisfying this condition is called the $p$-Orlicz norm of $X$, denoted by $\|X\|_{\psi_p}$. Having a finite Orlicz norm directly implies a control on the bound of $X$:
\begin{equation}\label{eq:tail_orlicz}
    \forall t>c_p \|X\|_{\psi_p},\ \PP(\|X\|>t)  \leq \pran{\psi_p\pran{\frac{t}{\|X\|_{\psi_p}}}}^{-1} \leq \kappa_p \exp\pran{-\pran{\frac{t}{\|X\|_{\psi_p}}}^p},
\end{equation}
for some constant $\kappa_p$. 
By integrating this tail bound, subexponential variables are seen to have controlled moments.

\begin{lem}\label{lem:orlicz_moment}
Let $p\in (0,1]$ and let $X$ be a random variable with $\|X\|_{\psi_p}\leq m<+\infty$. Let $a\geq 1$. Then, there exists $C_{p,a}>0$ such  that for all $r\geq c_pm$,
\begin{equation}
\begin{split}
&\E[\|X\|^a] \leq C_{p,a} m^a\\
    &\E[\|X\|^a\ones\{\|X\|\geq r\}] \leq  C_{p,a} r^a \pran{\frac{r}{m}}^{-p} e^{- \pran{\frac{r}{m}}^p}.
\end{split}
\end{equation}
\end{lem}

\begin{proof}
Let us prove the first inequality. According to \eqref{eq:tail_orlicz},
\begin{align*}
    \E[\|X\|^a] &\leq (c_pm)^a + \int_{(c_pm)^a}^{+\infty} \PP(\|X\|^a>t)  \dd t\leq (c_pm)^a + \kappa_p\int_{(c_pm)^a}^{+\infty} \exp(-t^{p/a}/m^p)  \dd t \\
    &\leq (c_pm)^a + \kappa_p m^a \frac{a}{p}\int_{c_p^p}^{+\infty}u^{\frac{a}{p}-1} \exp(-u)  \dd u \leq C_{p,a}m^a.
\end{align*}
The second inequality is obtained in a similar fashion.
\end{proof}

We conclude this section by a technical lemma, that shows that the integral of a function $f$ against a subexponential distribution $P$ is stable under polynomial perturbations of $f$.

\begin{lem}\label{lem:matching_moments}
Let $f:\Rd\to \R$ be a function with  $0\leq f(x)\leq L\dotp{x}^a$ for every $x\in \Rd$ and some $L,a\geq 0$. Let $P$ be a subexponential distribution (that is such that $\|X\|_{\psi_1}<+\infty$ if $X\sim P$) and let $I = \E_P[ f(X)]$. Then, for every $b\geq 0$,
\begin{equation}
    \begin{split}
        &\E_P[f(X)\dotp{X}^{-b}] \geq C I\log_+( 1/I)^{-b}\\
    &\E_P[f(X)\dotp{X}^b ] \leq CI\log_+( 1/I)^{b},
    \end{split}
\end{equation}
where $C$ depends on $L,a,b$ and $\|X\|_{\psi_1}$.
\end{lem}

\begin{proof}
We only prove the first inequality, the second one being proven similarly.  Fix a threshold $r\geq 1$ and write
\begin{align*}
    \int f(x)\dotp{x}^{-b} \dd P(x)\geq \frac{1}{(2r)^b}\int_{\|x\|\leq r} f(x)\dd P(x).
\end{align*}
According to \Cref{lem:orlicz_moment} with $m=\|X\|_{\psi_1}$, if $r\geq c_1m$, 
\[\int_{\|x\|\geq r} |f(x)|\dd P(x) \leq L C_a r^a e^{-r/m}.\]
Pick $r=m \lambda \log_+(1/I)$ for some $\lambda >0$. For $\lambda$ large enough with respect to the different constants involved, the above display is smaller than $I/2$. Therefore,
\[ \int_{\|x\|\leq r} |f(x)|\dd P(x) = I - \int_{\|x\|\geq r} |f(x)|\dd P(x)\geq I/2,\]
concluding the proof.
\end{proof}

\section{Properties of strongly convex and smooth potentials}\label{app:potential}

We gather in this section the different growth properties of $(\alpha,a)$-strongly convex and $(\beta,a)$-smooth potentials (as well as their convex conjugates). We start with a simple lemma that we give without proof, and that follows easily from writing Taylor expansions. Remark that the results given in this section also hold for $a=0$ for strongly convex or smooth differentiable functions in the classical sense, without requiring second differentiability.

\begin{lem}\label{lem:smooth_basic}
Let $\phi$ be a $(\beta,a)$-smooth function for some $\beta,a\geq 0$. Let $x,y\in B(0;r)$ for some $r\geq 0$. Then,
\begin{align}
    &\|\nabla\phi(x)-\nabla\phi(y)\| \leq \beta \dotp{r}^a\|x-y\| \leq 2\beta \dotp{r}^{a+1},\\
    &|\phi(x)-\phi(y)| \leq (2\beta \dotp{r}^{a+1} +  \|\nabla\phi(0)\|) \|x-y\| \leq 2\beta' \dotp{r}^{a+2},
\end{align}
where $\beta'=2\beta + \|\nabla\phi(0)\|$. 
Furthermore, if $\nabla\phi$ is invertible with $\|\nabla\phi^{-1}(x)\|,\|\nabla\phi^{-1}(y)\|\leq t$,  then
    \begin{equation}
            \|\nabla\phi^{-1}(x)-\nabla\phi^{-1}(y)\| \geq \frac{\|x-y\|}{\beta} \dotp{t}^{-a}.
    \end{equation}
\end{lem}

\begin{lem}\label{lem:strong_convex_basic}
Let $\phi$ be a $(\alpha,a)$-strongly convex function for some $\alpha,a\geq 0$. Let $x,y\in\Rd$ with $r = \max\{\|x\|,\|y\|\}$. Then, $\nabla\phi$ is bijection on $\Rd$, and
\begin{align}
    &\|\nabla\phi(x)-\nabla\phi(y)\| \geq \frac{\alpha}{2^{a+2}} \dotp{r}^a \|x-y\|, \\
    &\|\nabla\phi^{-1}(x)\| \leq 4\alpha^{-\frac{1}{a+1}} (\|x\| + \|\nabla\phi(0)\|)^{\frac{1}{a+1}},\\
    &\|\nabla\phi^{-1}(x)-\nabla\phi^{-1}(y)\| \leq \frac{2^{a+2}}{\alpha} \|x-y\|
\end{align}
\end{lem}
\begin{proof}
First, note that $\phi$ is in particular $\alpha$-strongly convex of exponent $0$ (i.e., strongly convex in the classical sense), so that $\nabla\phi$ is indeed a bijection. 
We prove the first inequality. Assume without loss of generality that $\|y\|=r$. It holds that
\[ \|\nabla\phi(x)-\nabla\phi(y)\| \geq \int_0^1 \lambda_{\min}(\nabla^2\phi(tx + (1-t)y)) \|x-y\| \dd t.\]
For $0\leq t\leq 1/4$, we have $\|tx+(1-t)y\| \geq r -2t r\geq r/2$. Therefore,
\[ \|\nabla\phi(x)-\nabla\phi(y)\| \geq \frac{\alpha}{4} \dotp{\frac r 2}^{a} \|x-y\| \geq \frac{\alpha}{2^{a+2}} \dotp{r}^a \|x-y\|.\]
The second inequality is obtained by inverting the first one and using that $\dotp{r}\geq r$. The third one is also obtained by inverting the first one, and using that $\dotp{r}\geq 1$.
\end{proof}

We now turn to bounds on the convex conjugates.
\begin{lem}\label{lem:conjugate_strict_convex}
Let $\phi$ be a $(\alpha,a)$-strongly convex function  for some $\alpha,a\geq 0$, with $\phi(0)=0$. Then, $\nabla\phi$ is invertible, and, for $y\in\Rd$,
\begin{equation}
    \phi^*(y) \leq \alpha^{-\frac{1}{1+a}}\|\nabla\phi(0)-y\|^{\frac{2+a}{1+a}}.
\end{equation}
\end{lem}

\begin{proof}
We write a Taylor expansion around $0$. For every $z\in\Rd$,
\begin{align*}
       \phi(z) &\geq \phi(0) + \dotp{\nabla\phi(0),z} + \int_0^1 \lambda_{\min}(\nabla^2\phi(tz)) \|z\|^2t\dd t \\
       &\geq \dotp{\nabla\phi(0),z} + \int_0^1 \alpha \dotp{t\|z\|}^a \|z\|^2t\dd t \geq \dotp{\nabla\phi(0),z} + \int_0^1 \alpha (t\|z\|)^a \|z\|^2t\dd t \\
       &\geq \dotp{\nabla\phi(0),z} + \frac{\alpha}{a+2} \|z\|^{a+2}.
\end{align*}
 Therefore, 
 \[ \phi^*(y) \leq \sup_{z\in\Rd} \dotp{z,y-\nabla\phi(0)} - \frac{\alpha}{a+2} \|z\|^{a+2}.\]
 We conclude thanks to the formula for the convex conjugate of the function $x\mapsto \|x\|^{a+2}$.
\end{proof}

The following lemma states  that, for points $y$ that stays at a bounded distance from $\nabla\phi(x)$ the convex conjugate  $\phi^*(y)$ is lower bounded by a quadratic approximation of $\phi$ at $x$.

\begin{lem}[Bounds of convex conjugates: quadratic behavior]\label{lem:lower_bound_conjugate}
Let $\phi$ be a $(\beta,a)$-smooth function for some $\beta,a\geq 0$. Let $\ell\geq 0$, $x\in \R^d$ and $y$ be such that $\|y-\nabla\phi(x)\|\leq \ell$. Then
\begin{equation}\label{eq:key_proof_stab}
    \phi^*(y) \geq  -\phi(x) + \dotp{y,x} +\frac{1}{2D_x}\|y-\nabla\phi(x)\|^2,
\end{equation}
where $D_x =  \beta\dotp{\|x\|+\frac{\ell}{\beta}}^a$.
\end{lem}

\begin{proof}
Let $x\in \Rd$ and $z\in \Rd$, with $\|x\|,\|z\|\leq r$. A Taylor expansion yields
\begin{equation}\label{eq:convex_lower_bound}
\begin{split}
    \phi(z) &\leq \phi(x) + \dotp{\nabla\phi(x),z-x} + \frac{\beta}{2} \dotp{r}^a\|z-x\|^2 =q_x(z).
\end{split}
\end{equation}
 Let $z = x+ \frac{y-\nabla\phi(x)}{D_x}$, so that we may pick $r = \|x\| + \frac{\ell}{D_x}$.  We obtain 
\begin{equation}
\begin{split}
    \phi^*(y) &\geq \dotp{z,y} -\phi(z) \geq \dotp{z,y} - q_x(z)  \\
    &\geq  -\phi(x) + \dotp{y,x} +\frac{1}{D_x}\|y-\nabla\phi(x)\|^2\pran{1-\frac{\beta}{2D_x}\dotp{r}^a}.
    \end{split}
\end{equation}
A little bit of algebra yields $1-\frac{\beta}{2D_x}\dotp{r}^a \geq \frac{1}{2}$, 
giving the result.
\end{proof}

We are now in position to prove \Cref{prop: stab_bound}.

\begin{proof}[Proof of \Cref{prop: stab_bound}]
We write the proof when the potentials are smooth and defined on $\Rd$, the case where the support of $P$ is contained in $B(0;R)$ being a straightforward adaptation of the proof of \cite[Proposition 10]{hutter2021minimax}.   
We use \Cref{lem:smooth_basic} to get
\begin{equation}\label{eq:bound_diff_nabla_01}
\|\nabla\phi_1(x)-\nabla\phi_0(x)\| \leq M + 4\beta \dotp{x}^{a+1}.
\end{equation}

We apply \Cref{lem:lower_bound_conjugate} to $\phi=\phi_1$, with $y=\nabla\phi_0(x)$ and $\ell$ given by the upper bound in \eqref{eq:bound_diff_nabla_01}. Note that $S(\phi_0) = \int \dotp{\nabla\phi_0(x),x}$. Integrating \eqref{eq:key_proof_stab} against $P$ yields
\begin{align*}
    S(\phi_1) &\geq \int \dotp{\nabla\phi_0(x),x}\dd P(x) + \int \frac{1}{2D_x}\|\nabla\phi_0(x)-\nabla\phi_1(x)\|^2 \dd P(x) \\
    &=S(\phi_0) + \int \frac{1}{2D_x}\|\nabla\phi_0(x)-\nabla\phi_1(x)\|^2 \dd P(x) 
\end{align*}
Note that $D_x$ grows at most polynomially with $\|x\|$, with degree $a(a+1)$. The  inequality then follows from \Cref{lem:matching_moments}: we obtain that, for $I = \|\nabla\phi_0-\nabla\phi_1\|_{L^2(P)}^2$,
\begin{equation}
    I\log_+(1/I)^{-b} \leq C \ell.
\end{equation}
One can invert this inequality to obtain an inequality of the form $I\leq C'\ell \log_+(1/\ell)^b$.

It remains to prove the second inequality.  If $\phi_1$ is $(\alpha,0)$-strongly convex (with domain $\R^d$), then for all $x\in \R^d$,
\begin{align*}
   \phi_1(x)+ \phi_1^*(\nabla\phi_0(x)) \leq \dotp{x,\nabla\phi_0(x)} + \frac{1}{2\alpha}\|\nabla\phi_0(x)-\nabla\phi_1(x)\|^2. 
\end{align*}
We obtain the second inequality by integrating the two sides of this equation with respect to $P$ and using $\int\dotp{x,\nabla\phi_0(x)} \dd P(x)=S_0(\phi)$.
\end{proof}

When $y$ is far away from $\nabla\phi(x)$, as $\phi$ grows polynomially (of order $2+a$), we expect the convex conjugate to behave like the convex conjugate of the function $z\mapsto\|z-\nabla\phi(x)\|^{a}$, which scales like $\|y-\nabla\phi(x)\|^{\frac{2+a}{1+a}}$, where $\frac{2+a}{1+a}$ is the conjugate exponent of $2+a$.

\begin{lem}[Lower bound of convex conjugates: polynomial behavior]\label{lem:lower_bound_conjugate_bis}
Under the same assumptions as \Cref{lem:lower_bound_conjugate}, if  $\|y-\nabla\phi(x)\|\geq \beta\max\{1,\|x\|^{a+1}\}$, then
\begin{equation}\label{eq:poly_behavior}
    \phi^*(y) \geq  -\phi(x) + \dotp{y,x} +2^{-2a-1} \beta^{-\frac{1}{a+1}} \|y-\nabla\phi(x)\|^{\frac{2+a}{1+a}}.
\end{equation}
\end{lem}

\begin{proof}
Let $v=\|y-\nabla\phi(x)\|$.
We start again from \eqref{eq:convex_lower_bound}.  Let $e=\frac{y-\nabla\phi(x)}{\|y-\nabla\phi(x)\|}$ and choose $z=x + e \cdot t \cdot (v/\beta)^{\frac{1}{1+a}}$ for some parameter $t$ to fix.  Thanks to the condition on $v$, we may choose the upper bound $r$ on the norm of $x$ and $z$ as  $r= (t+1) (v/\beta)^{\frac{1}{1+a}}$. As $v\geq \beta$, it holds that $\dotp{r}^a \leq (t+2)^a (v/\beta)^{\frac{a}{1+a}}$.  We obtain
\begin{align*}
    \phi^*(y) &\geq \dotp{z,y}-q_x(z) \\
    &\geq -\phi(x) + \dotp{x,\nabla\phi(x)} +  \dotp{z,y-\nabla\phi(x)} - \frac{\beta}{2} \dotp{r}^a t^2 (v/\beta)^{\frac 2{1+a}} \\
    &\geq -\phi(x) +\dotp{y,x} + t v^{1+\frac{1}{1+a}}\beta^{-\frac{1}{1+a}} - \frac{\beta}{2} t^2 (2+t)^a (v/\beta)^{\frac{2+a}{1+a}} \\
    &\geq -\phi(x) +\dotp{y,x} + \beta^{-\frac{1}{1+a}} v^{\frac{2+a}{1+a}}(t - \frac{ t^2(2+t)^a}{2}).
\end{align*}
We let $t = 2^{-2a}$ to obtain that $t - \frac{ t^2(2+t)^a}{2}\geq \frac {t} 2$, concluding the proof.
\end{proof}

If one does not care about the tight exponent $\frac{2+a}{1+a}$, it is always possible to lower bound the last term of \eqref{eq:poly_behavior} by a term proportional to $\|y-\nabla\phi(x)\|$ (recall that we assume that  $\|y-\nabla\phi(x)\|\geq \beta$ in this lemma).

We can also gather \Cref{lem:lower_bound_conjugate} and \Cref{lem:lower_bound_conjugate_bis} together to obtain, with $\ell = \beta\max\{1,\|x\|^{a+1}\}$ and $v=\|y-\nabla\phi(x)\|$,
\begin{equation}\label{eq:master_lower_bound}
     \phi^*(y) \geq  -\phi(x) + \dotp{y,x} +\min\{\frac{1}{2D_x}v^2, C_{a,\beta} v\},
\end{equation}
where $D_x$ is bounded by a polynomial expression of degree $a(a+1)$ in $\|x\|$.

\section{Covering numbers and suprema of empirical processes}\label{sec: properties_covering_bracketing}

We give in this section proofs related to covering numbers.  We start by a simple lemma that we will repeatedly use.

\begin{lem}[Covering of lines]\label{lem:covering_lines}
    Let $\cF$ be a subset of a normed space $(E,\|\cdot\|)$, bounded by $R$. Let $g\in \cF$ and consider the set of lines $\overline \cF=\{tf+(1-t)g:\ f\in\cF,\ 0\leq t\leq 1\}$. Then, there exists a constant $c$ depending on $R$ such that for $h>0$,
    \begin{equation}\label{eq:covering_lines}
        \log \cN(h,\overline \cF,E) \leq \log \cN(h/2,\cF,E) + c\log_+(1/h).
    \end{equation}
\end{lem}
\begin{proof}
    Let $A$ be a minimal $(h/2)$-covering of $\cF$ and let $T$ be a $(h/(4R))$-covering of $[0,1]$. For $tf+(1-t)g\in \cF$, we let $f_0$ be the closest element to $f$ in the covering $A$, and define  $t_0$ likewise. Then,
    \begin{align*}
        \|tf+(1-t)g - (t_0f_0 + (1-t_0)g)\| \leq \|f-f_0\| +  2R|t-t_0| \leq h.
    \end{align*}
    Therefore, the set $\{t_0f_0+(1-t_0)g:\ f_0\in A,\ t_0\in T\}$ is a $h$-covering of $\overline \cF$ whose size satisfies the inequality in \eqref{eq:covering_lines}.
\end{proof}

\begin{lem}[Covering numbers under affine transformations and dilations]\label{lem:bracket_affine}
 Let $\cF$ be a class of functions on $\R^d$, endowed with a norm $\|\cdot\|$. Assume that linear functions are bounded for this norm and  that $\sup_{f\in \cF}\|f\|\leq \sigma$. For $R\geq 1$, let
\begin{equation}\label{eq:lem_transfo}
    \tilde\cF \subseteq \{x\mapsto b\phi(x)+\dotp{t,x}: \phi \in \cF, \ 0\leq b \leq R, \|t\|\leq R\},
\end{equation}
Then, for $h>0$,
\begin{equation}\label{eq:brack_affine}
    \log \cN(h, \tilde\cF, \|\cdot\|) \leq \log \cN(h/(2R), \cF, \|\cdot\|) + c \log_+(1/h),
\end{equation}
where $c$ depends on $\sigma$, $R$, $d$ and on $\|\cdot\|$.
\end{lem}

\begin{proof}
Assume that the norm $\|\cdot\|$ of the function $x\mapsto\dotp{t,x}$ is smaller than $L$ for every $\|t\|\leq R$. 
Let $h>0$. Let $A$ be a minimal $h/(2R)$-covering of $\cF$, let $B$ be a minimal $h/(4\sigma)$-covering of $[0,R]$, and let $T$ be a minimal $h/(4L)$-covering of $B(0;R)$.  Let $b\phi +\dotp{t,\cdot} \in \tilde\cF$, and consider the function $b_0\phi_0+\dotp{t_0,\cdot}$ obtained by considering the closest $\phi_0$ to $\phi$, $b_0$ to $b$, and $t_0$ to $t$ in their respective nets. Then,
\begin{align*}
    \|b\phi +\dotp{t,\cdot}-(b_0\phi_0+\dotp{t_0,\cdot})\| \leq |b-b_0|\sigma + R\|\phi-\phi_0\| + L\|t-t_0\| \leq h.
\end{align*}
Hence, the set $\{b_0\phi_0+\dotp{t_0,\cdot},\ \phi_0\in A, b_0\in B, t_0\in T\}$ is a $h$-net of $\tilde \cF$, whose size satisfies the inequality in \eqref{eq:brack_affine}.
\end{proof}

We now prove \Cref{lem:C2_satisfied}, that gives sufficient conditions for Condition \ref{cond:covering_gradient} to be satisfied.
\begin{proof}[Proof of \Cref{lem:C2_satisfied}]
Let $\cF$ be any class of $(\beta,a)$-smooth potentials satisfying \ref{cond:smooth}. We assume without loss of generality that $\supp(P)\subseteq\dom(\phi)$ for all $\phi\in\cF$. We can actually even assume that $\dom(\phi)=\R^d$, up to minor modifications in the general case. 
Let $\cG\subseteq  \{t\phi+(1-t)\overline \phi:\ \phi \in \cF,\ 0\leq t \leq 1\}$ as in \ref{cond:covering_gradient} and let $\tau>0$. We first bound the covering numbers of $B_\tau$. 

As a preliminary remark, let us notice that if $g\in B_\tau$, then $\|\nabla g(0)-\nabla\overline\phi(0)\|\lesssim 1+\tau$. Indeed, let $\|\nabla g(0)-\nabla \overline \phi(0)\|= K$. As $g$ and $\overline\phi$ are $(\beta,a)$-smooth,  it holds according to \Cref{lem:smooth_basic} that for all $x$ such that $4\beta \dotp{x}^{a+1}\leq K/2$, we have $\|\nabla g(x)-\nabla \overline\phi(x)\|\geq K/2$. Hence,
\[ \tau^2 \geq \int \|\nabla g-\nabla \overline\phi \|^2 \dd P\geq \int_{4\beta \dotp{x}^{a+1}\leq K/2} \dd P(x) \frac{K^2}{4}, \]
an inequality that is only possible if $K \lesssim 1+\tau$.
Hence, it holds that 
\begin{equation}
    \sup_{g\in \cG }\|\nabla g(0)-\nabla \overline\phi(0)\| \leq K
\end{equation}
for some constant $K$ satisfying $K\lesssim 1+\tau$.

Let $h>0$. As the potentials in $\overline\cF$ are $(\beta,a)$-smooth, it holds thanks to \Cref{lem:orlicz_moment} and \Cref{lem:smooth_basic} that for all $g\in \cG$, $\int_{|x|>R} |g(x)-\overline\phi(x)|^2\dd P(x)<h^2/16$
for some choice of $R$ of the form $R=c_0\log_+(K/h)^q$ for some $c_0,q>0$.  According to \ref{cond:loc_poincare}, the local Poincaré inequality holds for all balls centered at a point in $B(0,R)$ of radius smaller than $r_{\max} =\kappa \dotp{R}^{-\theta}$. Let $B_1,\dots ,B_I$ be a packing of $\supp P \cap B(0;R)$ with balls  of radius $\lambda h/(2\tau)$, for some parameter $\lambda$ to be specified. In particular, the balls $\tilde B_1,\dots,\tilde B_I$ obtained by doubling the radius of the balls $B_i$ cover $\supp P \cap B(0;R)$. Assume first that $\lambda h/\tau\leq r_{\max}$, so that the local Poincaré inequality can be applied to these balls. 
We let $(\chi_i)_{1\leq i \leq I}$ be a partition of unity adapted to this covering. Let $f=g-\overline\phi$ and let $\tilde f=\sum_i (P_{\tilde B_i} f)\chi_i$. As the balls $B_i$ are pairwise disjoint, each ball $\tilde B_i$ can intersect at most $C_d$ balls $\tilde B_j$ for $i\neq j$, where $C_d$ depends only on $d$. Hence
    \begin{align*}
        \|f-\tilde f\|_{L^2(P)}^2 &\leq C_d\sum_{i=1}^I \|\chi_i(f-P_{\tilde B_i}f)\|_{L^2(P)}^2 + \frac{h^2}{16} \\
        &\leq C_d\sum_{i=1}^I  P(\tilde B_i) \int_{\tilde B_i} |f-P_{\tilde B_i}f|^2 \dd P_{\tilde B_i} +\frac{h^2}{16}\\
        &\leq C_dC_{\texttt{LPI}}  \sum_{i=1}^I \frac{\lambda^2 h^2}{\tau^2} \int_{\tilde B_i} \|\nabla f\|^2 \dd P +\frac{h^2}{16},
    \end{align*}
    where we use the local Poincaré inequality at the last step.  Using once again that each $\tilde B_i$ intersects at most $C_d$ balls $\tilde B_j$, we have
    \begin{equation}\label{eq:localpoincare_cover}
        \|f-\tilde f\|_{L^2(P)}^2\leq  C_d^2C_{\texttt{LPI}} \lambda^2  \frac{h^2}{\tau^2} \int \|\nabla f\|^2 \dd P +\frac{h^2}{16} \leq \frac{ h^2}{4}
    \end{equation} 
    for $\lambda= 1/\sqrt{4C_{\texttt{LPI}} C_d^2}$.
    
    Furthermore, as $\sup_{g\in \cG}\|\nabla g(0)-\nabla \overline \phi(0)\|<K$, and as the functions in $\cG$ are smooth of order $(a,\beta)$, it holds thanks to \Cref{lem:smooth_basic} that $L=\sup_{g\in \cG} \sup_{x\in B(0;R)} |g(x)-\overline\phi(x)|$ is at most of order polynomial in $\tau$. The vector $(P_{\tilde B_i} f)_{1\leq i\leq I}$ is a subset of $[-L,L]^I$.
    \Cref{eq:localpoincare_cover} shows that a $(h/2)$-covering of $[-L,L]^I$ 
    for the $\infty$-norm induces a $h$-covering of $B_\tau$ for the norm $L^2(P)$. This covering has a log-size of order $I\log_+(L/h)$, 
    where $I \lesssim (h/(\tau R))^{-m}$. This gives the desired bound when $\lambda h\leq r_{\max}\tau$. 

It remains to bound the covering numbers for large values of $h<\tau$, that is when the inequality $h\lambda \leq r_{\max}\tau$ is not satisfied.
Let $h_{\max}$ be the largest value of $h$ such 
that this inequality is satisfied, and remark that $h_{\max}>c_3\tau/\log(K)^{q'}$ 
for some $c_3,q'>0$. 
    We use the bound
    \[ \log \cN(h,B_\tau,L^2(P)) \leq \log \cN(h_{\max},B_\tau,L^2(P)) \lesssim_{\log_+(1/\tau)} 1. \]
    As $h<\tau$, this bound is once again (up to polylogarithmic factors), of order $(h/\tau)^{-m}$, concluding the proof.
\medskip

We now succinctly explain how to bound the covering numbers of $B_\tau^*$. As $\nabla\phi_0$ is locally Lipschitz continuous and $P$ is a locally doubling measure satisfying a local Poincaré inequality, the probability measure $Q=(\nabla\phi_0)_\sharp P$ also  satisfies a local Poincaré inequality, see \cite[Lemma 8.3.18]{heinonen2015sobolev}. Furthermore, $(\alpha,a)$-strong convexity of the potentials $g\in \cG$ ensures that the potentials $g^*\in \cG^*$ have sublinear growth (\Cref{lem:conjugate_strict_convex}). As for the previous bound, we obtain that $\sup_{g^*\in \cG^*}\|\nabla g^*(0)-\nabla \overline\phi^*(0)\|\lesssim 1+\tau$ and that $\int_{|T_0(x)|>R} |g^*(T_0(x)) - \nabla \overline \phi^*(0)|^2\dd P(x) \leq h^2/32 $ for $R\lesssim_{\log(K/h)} 1$. Furthermore, as $T_0$ is $\beta\dotp{R}^\theta$-Lipschitz continuous on $B(0;R)$, the covering number of $T_0(\supp(P)\cap B(0;R))$ is bounded (up to polylogarithmic constants)  by the covering number of $\supp(P)\cap B(0;R)$. We may now proceed as in the proof of the bound of the covering numbers of $B_\tau$ by consdering a packing number of $T_0(\supp(P)\cap B(0;R))$ and using that $Q$ satisfies a local Poincaré inequality. 
\end{proof}

At last, we prove \Cref{prop:van_wellner}, which gives bound on the suprema of empirical processes.

\begin{proof}[Proof of \Cref{prop:van_wellner}]
Consider first the case $\eta=0$. For $\eps=0$, the result is given by \cite[Theorem 2.14.21]{vaart2023empirical}. The case $\eps>0$ is obtained  by a  truncation scheme. 
For $g\in \cG$, let $X_g=\sqrt{n}(\PP_n-\PP)(g)$. As in the proof of  \cite[Theorem 2.14.21]{vaart2023empirical}, we remark that by Bernstein's inequality the process $(X_g)_{g\in \cG}$ satisfies a property of the form
\[ \forall x>0,\ f,g\in\cG,\ \PP(|X_f-X_g|> c_1 \sqrt{x}  \|f-g\|_{L^2(P)}+c_2x\frac{\|f-g\|_\infty}{\sqrt{n}}) \leq 2e^{-x}\]
where $c_1$ and $c_2$ are numerical constants. 
Let $\eps_{0,2}=\sigma$ and $\eps_{0,\infty}=M$. 
Assume without loss of generality that $0\in \cG$, and, for $i\in \{2,\infty\}$, let $\cG_{0,i}=\{0\}$.  For $j\geq 1$,  let $\eps_{j,i}$ be the largest $h\leq \eps_{0,i}$ such that $H_i(h)\geq 2^j$. As the function $H_i$ is increasing and converges to $+\infty$ as $h \to 0$, the sequence $(\eps_{j,i})_{j\geq 0}$ is a nonincreasing sequence converging to $0$. More precisely, this sequence is constant equal to $\eps_{0,i}$ until the largest value $j=j_0$ such that $H_i(\eps_{0,i})\geq 2^{j_0}$, and is given for $j>j_0$ by $\eps_{j,i}=H_i^{-1}(2^j)$. Let $J \geq 1$ be such that $\eps_{J-1,\infty}> \eps \geq \eps_{J,\infty}$. The condition $\eps< M=\eps_{0,\infty}$ and the fact that the sequence converges to $0$ ensures the existence of such an index. Moreover, we have $\eps_{J,\infty} \leq \eps <\eps_{0,\infty}$, so necessarily $\eps_{J,\infty}=H_\infty^{-1}(2^J)$.

For $i\in\{2,\infty\}$ and $j\geq 1$, let $\cG_{j,i}$ be a minimal $(2\eps_{j,i})$-covering of the set $\cG$ for the metric $\|\cdot\|_i$. If $\eps_{j,i}=\eps_{0,i}$, we can simply pick $\cG_{j,i}=\{0\}$. Otherwise, we have, $\log |\cG_{j,i}|\leq H_i(2\eps_{j,i})\leq  2^j$. By a standard argument, we may assume that the sets
$\{\cG_{j,i}\}_{j=1,\dots,J}$ are nested, at the price of increasing their logarithmic size by a constant factor. 
The net $\cG_{j,i}$ induces a partition $P_{j,i}$ of $\cG$ of the same cardinality, where each cell in the partition is given by the set of points in $\cG$ whose nearest neighbor (in the appropriate distance) is a certain element of $\cG_{j,i}$.  For all pairs of cells $Q_1$, $Q_2$ from $P_{j,2}$, $P_{j,\infty}$, respectively, consider the  cells of the form $Q_1\cap Q_2$. Let $\pi_j$ be a piecewise constant function on each of the cells, equal to an arbitrary point of the given cell. Let $\cG_j=\{\pi_j g:\ g\in\cG\}$. Note that by construction, $d_i(g,\pi_j g)\leq 2\eps_{j,i}$ for $i\in \{2,\infty\}$, while $\log |\cG_j|\leq \log(|\cG_{j,1}||\cG_{j,2}|)\leq c_32^{j}$ for some numerical constant $c_3>0$. The nestedness of the partitions $\{\cG_{j,i}\}_{j=1,\dots,J}$ implies that for $1\leq j\leq J$, $\pi_{j-1}\pi_jg=\pi_{j-1}g$. As $\pi_0g=0$, we write
\begin{align*}
    |X_g|=|X_g-X_{\pi_{0} g}| \leq \sum_{j=1}^{J-1} \sup_{g\in \cG} |X_{\pi_jg}-X_{\pi_{j-1}g}| + \sup_{g\in\cG}|X_{\pi_{J-1} g}-X_g|.
\end{align*}
According to \cite[Lemma 2.2.13]{vaart2023empirical}, it holds that
\[ \E[\sup_{g\in \cG} |X_{\pi_jg}-X_{\pi_{j-1}g}|] \leq c_4 (\eps_{j,2} 2^{j/2} +  \frac{\eps_{j,\infty}2^{j}}{\sqrt{n}})\]
for some numerical constant $c_4>0$. Note also that $|X_{\pi_{J-1} g}-X_g| \leq 2\sqrt{n}\|g-\pi_{J-1} g\|_\infty \leq 4\sqrt{n} \eps_{J-1,\infty}$. By assumption, it holds that 
\[ H_\infty(K\eps) \leq \frac{1}{2}H_\infty(\eps) \leq \frac{1}{2} H_\infty(\eps_{J,\infty}) = 2^{J-1}. \]
Hence, $\eps_{J-1,\infty}\leq K\eps$. We obtain
\begin{equation}
     \sqrt{n}\E \sup_{g\in \cG} |(\PP_n-\PP)(g)| = \E[\sup_{g\in \cG}|X_g|] \leq 4K\sqrt{n} \eps+ c_4\sum_{j=1}^{J-1}(\eps_{j,2}2^{j/2} + \frac{\eps_{j,\infty}2^j}{\sqrt{n}}  ).
\end{equation}
Let us assume that $J\geq 2$, for otherwise the sum is empty and the conclusion holds. By construction, for all $j\geq 1$, $H_\infty(h)\geq 2^j$ for $h\leq \eps_{j,\infty}$. 
Hence, it holds that
\[ \sum_{j=1}^{J-1}(\eps_{j-1,\infty}-\eps_{j,\infty})2^j \leq \sum_{j=1}^{J-1}\int_{\eps_{j,\infty}}^{\eps_{j-1,\infty}}H_\infty(h)\dd h \leq \int_{\eps_{J-1,\infty}}^{\eps_{0,\infty}}H_\infty(h)\dd h. \]
As $\sum_{j=1}^{J-1}\eps_{j-1,\infty}2^j = 2\sum_{j=0}^{J-2}\eps_{j,\infty}2^j $, 
it holds that
\[\sum_{j=1}^{J-1}\eps_{j,\infty}2^j = \sum_{j=1}^{J-1}(\eps_{j-1,\infty}-\eps_{j,\infty})2^j +\eps_{J-1,\infty}2^J - 2\eps_{0,\infty}\leq  \sum_{j=1}^{J-1}(\eps_{j-1,\infty}-\eps_{j,\infty})2^j +\eps_{J-1,\infty}2^J.\]
We have $2^{J-1} \leq H_\infty(h)$ for $h\leq \eps_{J-1,\infty}$, hence $\eps_{J-1,\infty}2^{J} \leq 4\int_{\eps_{J-1,\infty}/2}^{\eps_{J-1,\infty}}H_\infty(h)\dd h$. This yields
\[ \sum_{j=1}^{J-1} \eps_{j,\infty}2^j \leq 4\int_{\eps_{J-1,\infty}/2}^{\eps_{0,\infty}}H_\infty(h)\dd h.\] 
 Likewise, we find
\[ \sum_{j=1}^{J-1}\eps_{j,2}2^{j/2} \leq c_5  \int_{\eps_{J-1,2}/2}^{\eps_{0,2}}\sqrt{H_2(h))}\dd h\]
for some constant $c_5$. In total, we have
\begin{equation}
    \sqrt{n}\E \sup_{g\in \cG} |(\PP_n-\PP)(g)| \lesssim  K\sqrt{n} \eps+ + \int_{\eps_{J-1,2}/2}^{\eps_{0,2}}\sqrt{H_2(h)}\dd h+ \frac{1}{\sqrt{n}}\int_{\eps_{J-1,\infty}/2}^{\eps_{0,\infty}}H_\infty(h)\dd h.
\end{equation}
By construction, $\eps_{J-1,\infty}\geq \eps$. Let $0<\tilde\eps<\sigma=\eps_{0,2}$ be any number such that  $H_2(\tilde \eps)\geq H_\infty(\eps) \geq 2^{J-1}$, so that $\eps_{J-1,2}\geq \tilde\eps$.   We therefore obtain
\begin{equation}\label{eq:the_case_q=0}
     \sqrt{n}\E \sup_{g\in \cG} |(\PP_n-\PP)(g)| \lesssim   K\sqrt{n}\eps + \int_{\tilde\eps/2}^{\sigma}\sqrt{H_2(h)}\dd h+ \frac{1}{\sqrt{n}}\int_{\eps/2}^{M}H_\infty(h)\dd h.
\end{equation}
 This proves the result in the case $\eta=0$.
\medskip

For the case $\eta>0$, we let $L=(C\log_+(1/\eps))^{1/p}$  and write $g^L$  the function $g\in \cG$ restricted to $B(0,L)$. Then, one can check that according to \Cref{lem:orlicz_moment}, for $C$ large enough with respect to the $p$-Orlicz norm of $\xi\sim\PP$, $\eta$ and $p$, 
\begin{equation}\label{eq:remainder}
    \sqrt{n}\E \sup_{g\in \cG} |(\PP_n-\PP)(g-g^{L})| \leq 2M\sqrt{n}\int_{\|\xi\|>L} \dotp{\xi}^\eta \dd \PP(\xi) \leq 2M\sqrt{n}\eps \dotp{L}^\eta.
\end{equation} 
Let $\cG^L=\{g^L:\ g\in \cG\}$. 
The $L^\infty$-norm on $B(0,L)$ and the $L^\infty(\dotp{\ \cdot \ }^{-\eta})$-norm are related by
\[ \forall f:B(0,L)\to \R,\  \|f\|_{L^\infty(\dotp{\ \cdot \ }^{-\eta})} \leq \|f\|_{L^\infty} \leq  \dotp{L}^{\eta} \|f\|_{L^\infty(\dotp{\ \cdot \ }^{-\eta})}.\]
Hence, for $h>0$,
\[ \log \cN(h,\cG_L, L^\infty) \leq H_\infty(h\dotp{L}^{-\eta})= H'_\infty(h). \]
Furthermore, any function $g^L\in \cG_L$ satisfies $|g^L(\xi)|\leq M'=M\dotp{L}^{\eta}$. We apply the case $\eta=0$ to $\cG^L$ and $\eps'=\eps \dotp{L}^{\eta}$, with $H_2$ the upper bound on the $L^2$-covering numbers and $H'_\infty$ the upper bound on the $L^\infty$-covering number. It holds that $2H_\infty'(K\eps')=2H_\infty(K\eps)\leq H_\infty'(\eps')=H_\infty(\eps)$. Furthermore,  the number $\tilde\eps=\eps$ is such that $H_2(\tilde \eps) \geq H_\infty(\eps)=H_\infty'(\eps')$. Hence, according to \eqref{eq:the_case_q=0},
\begin{align*}
    \sqrt{n}\E \sup_{g\in \cG} |(\PP_n-\PP)(g^L)| &\lesssim   K\sqrt{n}\eps \dotp{L}^\eta + \int_{\tilde\eps  }^{\sigma}\sqrt{H_2(h)}\dd h+ \frac{1}{\sqrt{n}}\int_{\eps \dotp{L}^\eta}^{M\dotp{L}^\eta}H_\infty(h\dotp{L}^{-\eta})\dd h \\
    &\lesssim K\sqrt{n}\eps \dotp{L}^\eta + \int_{\eps  }^{\sigma}\sqrt{H_2(h)}\dd h+ \frac{\dotp{L}^\eta}{\sqrt{n}}\int_{\eps }^{M}H_\infty(h)\dd h.
\end{align*}
This concludes the proof together with \eqref{eq:remainder}.
\end{proof}

\section{Affine transformations of potentials}\label{sec:affine}
In \Cref{sec:LocationScale}, we claimed that as a particular case of our general theorems, we cover the case where $P$ is the normal distribution and 
\begin{equation}
    \cF_{\text{quad}}=\{ x\mapsto \tfrac12 x^\top B x + b^\top x:\ B \in \S^d_+, b\in\R^d\}
\end{equation}
is the set of quadratics. However, as the set $\cF_{\text{quad}}$ is not precompact, the covering numbers of $\cF_{\text{quad}}$ are infinite. Thus, condition \ref{cond:covering_convex} is not satisfied and \Cref{thm:strongly_convex} cannot be readily applied. Despite this, all potentials in $\cF_{\text{quad}}$ can be obtained by scaling and translating potentials from the set
\begin{align*}
    \cF_{\text{quad},0} = \{ x\mapsto \tfrac12 x^\top B x:\ B \in \S^d_+,\ \op{B}=1\}\,,
\end{align*}
whose covering numbers can easily be controlled, thus satisfying condition \ref{cond:covering_convex}. The next proposition asserts that if one performs such transformations on an arbitrary ``base set" $\cF_0$ satisfying condition \ref{cond:covering_convex}, 
then \Cref{thm:strongly_convex} also holds on the larger class $\cF$ that is induced by scaling and translating. 
To this end, given a class of potentials $\cF$ and $0\leq r_1,r_2\leq +\infty$, we introduce the modified class
\begin{equation}
    \cF_{\mathrm{aff}}^{r_1,r_2}\defeq \{x\mapsto b\phi(x) + \dotp{t,x}:\ \phi\in\cF,\ 0 \leq b \leq r_1,\ \|t\|\leq r_2\}.
\end{equation}

\begin{prop}\label{prop:affine}
Let $P$ be a subexponential distribution and let $Q=(\nabla\phi_0)_{\sharp} P$ for some $(\beta,a)$-smooth convex potential $\phi_0$  with $\beta,a\geq 0$. Let $\cF$ be a class of potentials satisfying \ref{cond:smooth} with $\nabla\phi(0)=0$ for every $\phi\in\cF$, that contains at least one $\alpha$-strongly convex potential for some $\alpha>0$. Let $\tilde \cF\subseteq \Faff^{r_1,+\infty}$ for some $r_1\geq 1$.  We consider two scenarios:
\begin{enumerate}
    \item either $r_1<+\infty$;
    \item or either $r_1=+\infty$ and every potential $\phi\in \cF$ is convex, and $\alpha$-strongly convex in at least one direction in the sense that there exists a unit vector $e$ such that for all $x,y\in \R^d$, $\phi(x)\geq \phi(y)+\dotp{x-y,\nabla\phi(y)} + \frac{\alpha}{2}|\dotp{x-y,e}|^2$. Furthermore, the marginal of $P$ in the direction $e$ has a bounded density.
\end{enumerate}
  Then, there exists $r>0$ and $\ell>0$ such that
\begin{equation}\label{eq:affine}
\E[\|\nabla\hat \phi_{\tilde\cF}-\nabla\phi_0\|^2_{L^2(P)}\ones\{\hat \phi_{\tilde\cF} \neq \hat \phi_{\tilde\cF^r}\}] \leq c\exp\pran{-Cn^\ell},\end{equation}
where $\tilde\cF^r = \tilde\cF\cap\Faff^{r,r}$, with constants $r$, $\ell$, $c$ and $C$ depending on the different parameters involved and on $P$.
\end{prop}

\begin{proof}
Let us write the sample $Y_1,\dots,Y_n\sim Q$ as $Y_i=\nabla\phi_0(X'_i)$, where $X'_i\sim P$. We let $P'_n$ be the empirical measure associated with the points $X'_i$. We write $m_n=\frac{1}{n}\sum_{i=1}^n X_i$ the empirical mean of $(X_1,\dots,X_n)$.
 We can write $\hat\phi_{\tilde\cF}(x) = b\phi(x)+\dotp{t,x}$ for $x\in\R^d$, where $\phi\in\cF$, and $b$ and $t$ are two parameters.  Let $\overline\phi$ be an $\alpha$-strongly convex potential in $\cF$.
 Note that $S_n(\hat\phi_{\tilde\cF})-S_n( \overline\phi)\leq 0$. We have for $y\in \R^d$
\begin{equation*}
 \hat\phi_{\tilde\cF}^*(y) =b\phi^*\pran{\frac{y-t}{b}},
\end{equation*}
so that
\begin{align*}
    S_n(\hat\phi_{\tilde\cF})=P_n(\hat\phi_{\tilde\cF}) + Q_n(\hat\phi_{\tilde\cF}^*) &= bP_n(\phi) + \dotp{t,m_n}   + \frac{b}{n}\sum_{i=1}^n \phi^*\pran{\frac{\nabla\phi_0(X'_i)-t}{b}}.
\end{align*}
Our goal is to show that the condition $S_n(\hat\phi_{\tilde\cF})-S_n(\overline\phi)\leq 0$ implies that both $b$ and $\|t\|$ are upper bounded by some parameter $r$, which implies that $\hat\phi_{\tilde\cF} \in \tilde\cF^r$. In particular, we then have $\hat\phi_{\tilde\cF}=\hat\phi_{\tilde\cF^r}$. 
\medskip

\textbf{Bound on $b$:}
In the first scenario, $b\leq b_{\max}=r_1$. In the second scenario, by definition of the convex conjugate, it holds that
\begin{equation}
    b\phi^*\pran{\frac{\nabla\phi_0(X'_i)-t}{b}} \geq -b\phi(m_n) + \dotp{m_n,\nabla\phi_0(X'_i)-t}.
\end{equation}
Hence, $S_n(\overline\phi)\geq S_n(\hat\phi_{\tilde\cF})\geq b(P_n(\phi)-\phi(m_n))+\dotp{m_n,P'_n(\nabla\phi_0)}$. Assume without loss of generality that $\phi$ is $\alpha$-strongly convex in the direction $x_1$. We have by strong convexity that $P_n(\phi)-\phi(m_n)   \geq \frac{\alpha}{2}\frac{1}{n}\sum_{i=1}^n |X_{i,1}-m_{n,1}|^2 = u_n$. Hence,
\begin{equation}\label{eq:def_bmax}
    b \leq b_{\max} =\max\{1, \frac{S_n(\overline\phi)-\dotp{m_n,P'_n(\nabla\phi_0)}}{u_n}\}.
\end{equation}
Note  that we have $b_{\max}\geq 1$, a property which will be used.

\medskip

\textbf{Bound on $\|t\|$:}
Finding a bound on the norm of the parameter $t$ proves to be more delicate, and relies on lower bounds on the convex conjugate given in \Cref{app:potential}. Apply \eqref{eq:master_lower_bound} to $x=m_n$   and $y=\frac{\nabla\phi_0(X'_i)-t}{b}$. The parameter $D_x$ appearing in \eqref{eq:master_lower_bound} can be upper bounded by a polynomial of order $a(a+1)$ with respect to $m_n$. In total, we obtain a lower bound of the form
\begin{align*}
    &b\phi^*\pran{\frac{\nabla\phi_0(X'_i)-t}{b}} \geq -b\phi(m_n) + \dotp{m_n,\nabla\phi_0(X'_i)-t} + \frac{C}{\dotp{m_n}^{a(a+1)}} \min\{v_i^2,v_i\},
\end{align*}
where $v_i=\|\frac{\nabla\phi_0(X'_i)-t}{b}-z\|$, and we define $z=\nabla\phi(m_n)$. Assuming without loss of generality that $\|t\|\geq 1$, we lower bound $v_i^2$,
\begin{align*}
     v_i^2 &\geq \frac{\|t\|^2 -2\|t\|(\|\nabla\phi_0(X'_i)\| +b_{\max}\|z\|)}{b_{\max}^2} \geq \frac{\|t\| -2\|\nabla\phi_0(X'_i)\| -2b_{\max}\|z\|}{b_{\max}^2} \eqdef U_i,
\end{align*}
as well as $v_i$,
\begin{align*}
    v_i &\geq \frac{\|t\|-\|\nabla\phi_0(X'_i)\|-b_{\max}\|z\|}{b_{\max}}  \geq  U_i,
\end{align*}
where we use that $b_{\max}\geq 1$. 
We obtain
\begin{equation}\label{eq:lower_bound_S}
    \begin{split}
       S_n(\hat\phi_{\tilde\cF})    &\geq bP_n(\phi)+\dotp{t,m_n}-b\phi(m_n)+\dotp{m_n,P'_n(\nabla\phi_0)-t}+ \frac{C}{n\dotp{m_n}^{a(a+1)}} \sum_{i=1}^n U_i\\
    \end{split}
\end{equation}

As $\phi$ is $(\beta,a)$-smooth with $\nabla\phi(0)=0$, we have $|\phi(x)|\leq 4\beta\dotp{x}^{a+2}$ according to \Cref{lem:smooth_basic}. 
Therefore, applying Jensen's inequality,
    \begin{align*}
         S_n(\hat\phi_{\tilde\cF})    &\geq -4\beta b_{\max}(P_n(\dotp{x}^{a+2})+P(\dotp{x}^{a+2})) + \dotp{m_n,P'_n(\nabla\phi_0)}+ \frac{C}{n\dotp{m_n}^{a(a+1)}} \sum_{i=1}^n U_i.
    \end{align*}
    As $b_{\max}\geq 1$, we obtain a bound of the form
\begin{align*}
    \frac{C\|t\|}{\dotp{m_n}^{a(a+1)}} &\leq 2C  P'_n(\|\nabla\phi_0\|) +2C b_{\max}\|z\| +b_{\max}^2 \dotp{m_n,P'_n(\nabla\phi_0)} \\
    &\qquad + b_{\max}^3 (S_n(\overline\phi) + 4\beta (P_n(\|x\|^{a+2})+P(\|x\|^{a+2}))).
\end{align*}
By \Cref{lem:smooth_basic}, \Cref{lem:strong_convex_basic}, and using that $\overline\phi$ is $\alpha$-strongly convex, we find that   $\|t\|$ is bounded by a polynomial expression $T_{\max}$ with variables given by  $b_{\max}$ and the  moments of $P_n$ and $P'_n$.
\medskip

\textbf{Conclusion:}
Let $r \geq 1$. One can make the key observation that if $\hat\phi_{\tilde\cF}\neq \hat \phi_{\tilde\cF^r}$, then either $b_{\max}>r$ or $T_{\max}>r$. Therefore,
\begin{align*}
    \E[\|\nabla\hat \phi_{\tilde\cF}-\nabla\phi_0\|^2_{L^2(P)}\ones\{\hat \phi_{\tilde\cF} \neq \hat \phi_{\tilde\cF^r}\}] &\leq 2\E[\|\nabla\hat \phi_{\tilde\cF}\|_{L^2(P)}^2\ones\{\hat \phi_{\tilde\cF} \neq \hat \phi_{\tilde\cF^r}\}]\\
    &\quad +2\|\nabla\phi_0\|^2_{L^2(P)}\PP(\hat \phi_{\tilde\cF} \neq \hat \phi_{\tilde\cF^r}).
\end{align*}
 By \Cref{lem:smooth_basic}, it holds that 
\[\| \nabla\hat\phi_{\tilde \cF}(x)\|\leq b_{\max}\|\nabla\hat\phi(x)\| + T_{\max}\|x\| \leq (b_{\max}\beta\dotp{x}^a +T_{\max})\|x\|.\]
Therefore,
\begin{align*}
     \E[\|\nabla\hat \phi_{\tilde\cF}-\nabla\phi_0\|^2_{L^2(P)}\ones\{\hat \phi_{\tilde\cF} \neq \hat \phi_{\tilde\cF^r}\}]\leq  c\E[b_{\max}^2\ones\{b_{\max}>r\}] + c \E[T_{\max}^2 \ones\{T_{\max}>r\}]
\end{align*}
for some constant $c>0$.
In the first scenario, $b_{\max}$ is constant and it remains to bound quantities of the form $\E[\frac{1}{n}\sum_{i=1}^n \|X_i\|^q \ones\{\frac{1}{n}\sum_{i=1}^n \|X_i\|^q>r\}]$, which is done in \Cref{lem:second_moment_orlicz}. In the second scenario, $b_{\max}$ is given in \eqref{eq:def_bmax}. As $S_n(\overline\phi)-\dotp{m_n,P'_n(\nabla\phi_0)}$ is bounded by a polynomial expression in the moments of $P_n$ and $P'_n$, it remains to use Young's inequality to bound quantities of the form
\[ \E[u_n^{-q}\ones\{u_n^{-1}>r\}],\]
where $u_n$ is defined above \eqref{eq:def_bmax}. This is done in \Cref{lem:bound_un}.

\end{proof}

\begin{lem}\label{lem:second_moment_orlicz}
Let $P$ be a subexponential distribution with $X_1,\dots,X_n\sim P$ a sample of i.i.d.~random variables. Then, for all $q\geq 1$, there exists $r>0$ such that 
\begin{equation}
  \E[\frac{1}{n}\sum_{i=1}^n \|X_i\|^q \ones\{\frac{1}{n}\sum_{i=1}^n \|X_i\|^q>r\}]\leq c\exp\pran{-Cn^{1/(2q)}}.
\end{equation}
\end{lem}

\begin{proof}
Let $Z=\frac{1}{n}\sum_{i=1}^n \|X_i\|^q$. According to \cite[Theorem 2.14.23]{vaart2023empirical}, as the random variable $X^q$ for $X\sim P$ belongs to the Orlicz space of order $1/q$, we have $\|Z\|_{\psi_{1/q}} \lesssim n^{-1/2}$. This implies that for $t$ large enough 
\[ \PP(Z> t)\leq \exp(-ct^{1/q}n^{1/(2q)}).\]
We obtain the conclusion by integrating this bound.
\end{proof}

\begin{lem}\label{lem:bound_un}
Under the assumptions of the second scenario in \Cref{prop:affine}, for all $q\geq 1$ there exists $r>0$ such that $\E[u_n^{-q}\ones\{u_n^{-1}>r\}] \leq Ce^{-cn^{1/4}}$.
\end{lem}
\begin{proof}
We have
\begin{align*}
    \E[u_n^{-q}\ones\{u_n^{-1}>r\}] = \int_{r^q}^\infty \PP(u_n^{-q}>x) \dd x = q \int_0^{r^{-1}} \PP(u_n<t) t^{1-q} \dd t.
\end{align*}
    We bound $\PP(u_n<t)$ differently depending on whether $t>n^{-2}$ or $t\leq n^{-2}$. 

    When $t\leq n^{-2}$, we use the crude bound $u_n >\frac{1}{2n} A_n^2$, where $A_n=\max_i X_i-\min_i X_i$ is the amplitude of the sample. Hence, $\PP(u_n<t)\leq \PP(A_n< \sqrt{2nt})$. Consider intervals of the form $B_k = [\sqrt{2nt}k, \sqrt{2nt}(k+2)]$ for $k\in \ZZ$. If $A_n<\sqrt{2nt}$, then all the points $X_i$ are in the same interval $I_k$, so that
    \[ \PP(u_n<t) \leq \sum_{k\in \ZZ} P(B_k)^n.\]
    We use two different bounds on $P(B_k)$. First, $P(B_k)\leq 2p_{\max}\sqrt{2nt}$. Second, as $P$ is subexponential, $P(B_k)\leq c\exp(-C\sqrt{2nt}|k|)$ when $\sqrt{2nt}|k|>c_0$. Hence, 
    \begin{align*}
       \PP(u_n<t) &\lesssim  \frac{n\log(1/t)}{\sqrt{nt}} (2p_{\max}\sqrt{2nt})^n + \sum_{C\sqrt{2nt}|k|>n\log(1/t)/2} e^{-C\sqrt{2nt}|k|}  \\
        &\lesssim n\log(1/t) (C\sqrt{nt})^{n-1} + t^{n/2} \lesssim n\log(1/t) (C\sqrt{nt})^{n-1},
    \end{align*}
    where we compute the geometric sum that appears at the second line. Thus, 
    \begin{equation}\label{eq:crude_bound_un}
         \int_0^{n^{-2}} \PP(u_n< t) t^{q-1} \dd t \leq   n(C\sqrt{n})^{n-1}\int_0^{n^{-2}} \log(1/t) t^{n/2+q-3/2} \dd t\lesssim_{\log n}n^{-2q}\pran{\frac{C}{\sqrt{n}}}^{n-1}.
    \end{equation}
       
When $t>n^{-2}$, we use that $u_n$ is concentrated around the variance $v = \Var_P(X_1)>0$. Indeed, $u_n=\frac{1}{n} \sum_{i=1}^n X_i^2 - \pran{\frac 1n\sum_{i=1}^n X_i}^2$. Each sum is concentrated around its expectation. More precisely, as $X_i^2$ belongs the Orlicz space of exponent $1/2$, \cite[Theorem 2.14.23]{vaart2023empirical} implies that $\PP(u_n< v/2) \lesssim e^{-Cn^{1/4}}$ for some $C>0$. In particular, as long as $r^{-1}<v/2$,
\[\int_{n^{-2}}^{r^{-1}} \PP(u_n<t) t^{1-q} \dd t \leq  r^{-1}n^{2q-2} e^{-Cn^{1/4}}.\]
Hence, using \eqref{eq:crude_bound_un} and the bound $n^{2q-2}e^{-Cn^{1/4}}\leq e^{-C'n^{1/4}}$ for some $C'>0$, it holds that when $r>2/v$, $ \E[u_n^{-q}\ones\{u_n^{-1}>r\}]\lesssim e^{-C''n^{1/4}}$ for some $C''>0$.
\end{proof}

As a consequence of \Cref{prop:affine}, we can show that  \Cref{thm:strongly_convex} holds for affine transformations of a base set $\cF$.

\begin{cor}\label{cor:affine}
Besides the assumptions of \Cref{prop:affine}, assume that $P$ and $\cF$ satisfy \ref{cond:smooth}, \ref{cond:poincare}, \ref{cond:covering_convex} and \ref{cond:bounded} for $R=+\infty$. 
Let $\tilde \cF\subseteq \Faff^{r_1,+\infty}$. 
  Then, letting $\ell^\star=\inf_{\phi\in\tilde\cF_\alpha}\|\nabla\phi-\nabla\phi_0\|^2_{L^2(P)}$ for some $\alpha>0$, in either of the two scenarios of \Cref{prop:affine}, it holds that
\begin{equation}
    \E \|\nabla \hat{\phi}_{\tilde\cF} - \nabla \phi_0\|^2_{L^2(P)}  \lesssim_{\log n,\log_+(1/\ell^\star)} \ell^\star +   \pran{n^{-\frac{2}{2+\gamma}} \vee  n^{-\frac{1}{\gamma}}}.
\end{equation}
 If $P$ also satisfies  \ref{cond:loc_poincare} and \ref{cond:loc_doubling} with support of dimension $m\geq 2$ and if $\gamma\in [0,2)$, then
\begin{equation}
    \E\|\nabla \hat{\phi}_{\tilde\cF} - \nabla \phi_0\|^2_{L^2(P)} \lesssim_{\log n,\log_+(1/\ell^\star)} \ell^\star +   n^{-\frac{2(m-\gamma)}{2m+\gamma (m-4)}}.
\end{equation}
\end{cor}

\begin{proof}
Let $\cF_0 = \{x\mapsto \phi(x)-\dotp{\nabla\phi(0),x}:\ \phi\in\cF\}$. Then, $\tilde \cF \subseteq (\cF_0)_{\mathrm{aff}}^{1,+\infty}$. Note that we may assume without loss generality that $\tilde\cF$ contains a strongly convex potential (otherwise $\ell^\star=+\infty$, and there is nothing to prove). \Cref{prop:affine} implies that for some $r>0$,
\begin{equation}
    \E[\|\nabla\hat \phi_{\tilde\cF}-\nabla\phi_0\|^2_{L^2(P)}\ones\{\hat \phi_{\tilde\cF} \neq \hat \phi_{\tilde\cF^r}\}] \leq \frac 1n.
\end{equation}
To conclude, it suffices to check that the set $\tilde \cF^r$ satisfy the assumptions of \Cref{thm:strongly_convex}. \Cref{lem:bracket_affine} ensures that condition  \ref{cond:covering_convex} is satisfied when $\cF$ satisfies it. It is also clear that $\tilde\cF^r$ satisfies \ref{cond:smooth} and \ref{cond:bounded} for $R=+\infty$. 
We then apply \Cref{thm:strongly_convex} to conclude.
\end{proof} 

Let us come back to the Gaussian case. Recall that we introduced the ``base'' set
\begin{align*}
    \cF_{\text{quad},0} = \{ x\mapsto \tfrac12 x^\top B x:\ B \in \S^d_+,\ \op{B}=1\}\,,
\end{align*}
whereas $(\cF_{\text{quad},0})_{\mathrm{aff}}^{+\infty,+\infty}=\cF_{\text{quad}}$. We are therefore in position to apply \Cref{cor:affine}. Note that every $\phi \in \cF_{\text{quad},0}$ is smooth, convex, and $1$-strongly convex in at least one direction. Furthermore, $P = N(0,I_d)$ has a density on $\R^d$ and satisfies  the Poincaré inequality \ref{cond:poincare}.  
Condition \ref{cond:covering_convex}  is verified for $\eta=2$ and $\gamma=0$ in \Cref{lem: Gaussian_bracket}. 
All in all, with the absence of a bias term, \Cref{cor:affine} implies that the estimator $\nabla \hat\phi_{\cF_{\mathrm{quad}}}$ attains the rate $n^{-1}$.

\begin{lem}\label{lem: Gaussian_bracket} It holds that $\log \cN(h,\cF_{\mathrm{quad},0},L^\infty(\dotp{\ \cdot \ }^{-2})) \leq c\log_+(1/h)$.
\end{lem}
\begin{proof}
If $\Sigma_1,\Sigma_2\in \S^d_+$, then $ \sup_{x} x^\top (\Sigma_2-\Sigma_1) x \cdot \dotp{x }^{-2} \leq  C\op{\Sigma_2-\Sigma_1}$. 
 As the set of matrices is finite dimensional, the covering number scales polynomially with $h$.
\end{proof}

\section{Lemmas: bounded case}\label{app: lemmas_bounded}
\begin{proof}[Proof of \Cref{lem:bounded_envelope}]
 Let us first bound $\phi\in \overline\cF$. Let $x\in B(0;R)$. Then \[ |\phi(x)| \leq \int_0^1 |\dotp{\nabla\phi(tx),x}| \dd t \leq  R^2.\] Remark that $Q$ is supported in $B(0; R)$. Let $y\in B(0;R)$. We have
\begin{equation}
    \phi^*(y) = \sup_{\|x\|\leq 2R} \dotp{x,y}-\phi(x) \leq 2 R^2 +  R^2 = 3 R^2.
\end{equation}
Also, choosing $x=0$, we obtain $\phi^*(y) \geq 0$.  Therefore, for $\xi=(x,y)\in \Xi$,
\[ |\gamma(\phi,\xi)|\leq |\phi(x)|+|\phi^*(y)| + |\overline\phi(x)|+|\overline\phi^*(y)| \leq 8R^2. \qedhere\]
\end{proof}

\begin{proof}[Proof of \Cref{lem:bounded_second_moment}]
Let $\phi$ be any potential in $\overline \cF$ and $g=\gamma(\phi,\cdot)-\gamma(\overline\phi,\cdot)\in \cG$. 
 Let $\phi_c=\phi-\int \phi\dd P$ be the centered version of $\phi$. Remark  that for $\xi=(x,y)\in \Xi$, $\gamma(\phi,\xi)=\phi(x)+\phi^*(y) = \phi_c(x)+\phi_c^*(y)$, and the same holds for $\gamma(\overline\phi,\cdot)$. 
 As $Q=(\nabla\phi_0)_{\sharp}P$, we may write $y$ in the support of $Q$ as $y=\nabla\phi_0(x)$ for $x$ in the support of $P$. Therefore, $\phi_0^*(y) = \dotp{y,x}-\phi_0(x)$. Let $x'$ be a point attaining the supremum in the definition of $\phi^*(y)$. Note that $\|x'\|\leq 2R$ and that $\nabla\phi(x')=y$. Then,
\begin{align*}
    \phi_c^*(y)-\phi_{0,c}^*(y) &= \dotp{y,x'-x} - \phi_c(x') + \phi_{0,c}(x) \\
    &= \dotp{y,x'-x} + (\phi_c(x)- \phi_c(x')) +  (\phi_{0,c}(x)-\phi_c(x)).
\end{align*}
As $\phi_c$ is convex, it holds that $\phi_c(x') \geq \phi_c(x) + \dotp{\nabla\phi(x),x'-x}$, so that
\begin{align*}
    \phi_c^*(y)-\phi_{0,c}^*(y) &\leq \dotp{\nabla\phi_0(x)-\nabla\phi(x),x'-x} +  (\phi_{0,c}(x)-\phi_c(x)).
\end{align*}
Likewise, $\phi_c(x) \geq \phi_c(x') + \dotp{y,x-x'}$. Therefore,
\begin{align*}
    \phi_c^*(y)-\phi_{0,c}^*(y) &\geq \phi_{0,c}(x)-\phi_c(x).
\end{align*}
In total, as $\|x-x'\|\leq R+2R\leq 3R$,
\begin{equation}
    |\phi_c^*(y)-\phi_{0,c}^*(y)| \leq 3R \|\nabla\phi_0(x)-\nabla\phi(x)\|+ |\phi_{0,c}(x)-\phi_c(x)|.
\end{equation}
By the Poincaré inequality, $\|\phi_{0,c}-\phi_c\|_{L^2(P)} \leq C_{\texttt{PI}}^{1/2} \|\nabla\phi_0-\nabla\phi\|_{L^2(P)}$. In total, we obtain that
\begin{equation}
    \|\phi_c^*-\phi_{0,c}^*\|_{L^2(Q)}\leq (3R+ C_{\texttt{PI}}^{1/2})\|\nabla\phi_0-\nabla\phi\|_{L^2(P)}.
\end{equation}

Then,
\begin{align*}
    \|g\|_{L^2(\PP)} &\leq \|\phi_c-\overline\phi_c\|_{L^2(P)} + \|\phi_c^*-\overline\phi_c^*\|_{L^2(Q)} \\
    &\leq \|\phi_c-\phi_{0,c}\|_{L^2(P)}+\|\overline\phi_c-\phi_{0,c}\|_{L^2(P)}+\|\phi_c^*-\phi_{0,c}^*\|_{L^2(Q)} + \|\overline\phi_c^*-\phi_{0,c}^*\|_{L^2(Q)} \\
    &\leq (3R+ 2C_{\texttt{PI}}^{1/2})(\|\nabla\phi_0-\nabla\phi\|_{L^2(P)}+\|\nabla\phi_0-\nabla\overline\phi\|_{L^2(P)}) \\
    &\leq (3R+ 2C_{\texttt{PI}}^{1/2})(\|\nabla\overline\phi-\nabla\phi\|_{L^2(P)}+2\|\nabla\phi_0-\nabla\overline\phi\|_{L^2(P)}) \\
    &\leq (3R+ 2C_{\texttt{PI}}^{1/2})(\tau+2\sqrt{C_0\ell^\star}),
\end{align*}
where we apply \Cref{prop: stab_bound} at the last line. 

As $\ell^\star\leq \tau^2$ by assumption, we obtain the conclusion.
\end{proof}

\begin{proof}[Proof of \Cref{lem:bounded_brackets}]
The application $\phi\mapsto \phi^*$ is an isometry for the $\infty$-norm. Hence, for $\phi,\phi'\in \overline\cF$,
\[ \|\gamma(\phi,\cdot)-\gamma(\phi',\cdot)\|_\infty \leq 2\|\phi-\phi'\|_\infty.\]
Hence, for $h>0$,
\[ \log \cN(h,\cG,L^\infty)\leq \log \cN(h/2,\overline\cF,L^\infty).\]
We then conclude thanks to \Cref{lem:covering_lines}.
\end{proof}

\section{Lemmas: strongly convex case}\label{app: lemmas_unbounded}
Our first lemma states that we may without loss of generality assume that every potential $\phi$ in $\cF$ satisfies $\|\nabla\phi(0)\|\leq K$ for some constant $K$.

\begin{lem}\label{lem:assume_gradient_bounded}
Let $\cF^r = \{\phi\in\cF:\ \|\nabla\phi(0)\|\leq K\}$. Then, for $K$ large enough,
\begin{equation}
    \E[\|\nabla\hat\phi_\cF - \nabla\phi_0\|^2_{L^2(P)}\ones\{\hat\phi_\cF \neq \hat\phi_{\cF^K}\}]\leq \frac{1}{n}.
\end{equation}
\end{lem}

\begin{proof}
Let $\cF_0 = \{x\mapsto \phi(x)-\dotp{\nabla\phi(0),x}:\ \phi\in\cF\}$. Then, $\cF \subseteq (\cF_0)_{\mathrm{aff}}^{1,+\infty}$. 
We may therefore apply \cref{prop:affine} (with $r_1=1<+\infty$), which implies that there exists $K>0$ such that the statement of the lemma holds.
\end{proof}
This preliminary remark shows that we can safely assume that $\|\nabla\phi(0)\|$ is uniformly bounded over $\phi\in \cF$, up to increasing the risk by a negligible factor of order $n^{-1}$: we will now simply write $\cF$ instead of $\cF^K$ and keep this assumption in mind. To ease notation, we will now write $T$ instead of $\nabla\phi$ for a general potential $\phi\in\cF$, with $T_0=\nabla\phi_0$ and $\overline T=\nabla\overline\phi$.

\begin{proof}[Proof of \Cref{lem:strict_envelope}]
Let $y\in \R^d$. 
 As $\phi$ is $(\alpha/2,a)$-strongly convex, we have by  \Cref{lem:conjugate_strict_convex} that $\phi^*(y) \leq (\alpha/2)^{-\frac{1}{1+a}}\|T(0) - y\|^{\frac{2+a}{1+a}} \leq C(K^2 + \|y\|^2)$, where we used the fact that $\|T(0)\|\leq K$. As $\overline\phi$ is also $(\alpha/2)$-strongly convex, the same inequality holds for $\overline\phi$ (where we pick $K$ larger than $\|\overline T(0)\|$). 
This yields
\begin{align*}
    |\phi^*(y) - \overline\phi^*(y)| &\leq 2C(K^2 + \|y\|^2).
\end{align*}
For $x\in \R^d$, we also have $|\phi(x)-\overline\phi(x)|\leq 2(2\beta + K)\dotp{x}^{a+2}$ according to \Cref{lem:smooth_basic}. Therefore, for $\xi=(x,y)\in \Xi$,
\[ |\gamma(\phi,\xi)|\leq 2C(K^2 + \|y\|^2) + 2(2\beta + K)\dotp{x}^{a+2} \leq C_0 \dotp{\xi}^{a+2}, \]
where $C_0$ depends on $\alpha$, $\beta$, $a$ and $K$.
\end{proof}

\begin{proof}[Proof of \Cref{lem:strict_second_moment}]
Let $\phi$ be any potential in $\overline \cF$ with $\|\nabla\phi-\nabla\overline\phi\|_{L^2(P)}=\|T-\overline T\|\leq\tau$ and let $g=\gamma(\phi,\cdot)-\gamma(\overline\phi,\cdot)\in \cG$. As in the proof of \Cref{lem:bounded_second_moment}, we let $\phi_c=\phi-\int \phi\dd P$ be the centered version of $\phi$ and remark  that for $\xi=(x,y)\in \Xi$, $\gamma(\phi,\xi)=\phi(x)+\phi^*(y) = \phi_c(x)+\phi_c^*(y)$. As $\phi$ is strongly convex, $T=\nabla\phi$ is invertible. By definition of a convex conjugate, it holds that
\begin{align*}
    \phi^*_c(y) = \langle y, T^{-1}(y)\rangle - \phi(T^{-1}(y)) + P(\phi) = \langle y, T^{-1}(y)\rangle - \phi_c(T^{-1}(y))\,.
\end{align*}
Omitting the argument $y$ for notational brevity, we obtain
\begin{align*}
    \phi_c^* - \phi_{0,c}^* &= \langle \id,T^{-1} - T_0^{-1} \rangle + \phi_{0,c}\circ T_0^{-1}  - \phi_c \circ T^{-1} \\
    &= \langle \id,T^{-1} - T_0^{-1} \rangle + (\phi_{0,c}\circ T_0^{-1} -\phi_c \circ T_0^{-1}) + (\phi_c \circ T_0^{-1}-\phi_c \circ T^{-1} )\\
    &= A_1 + A_2+A_3.
\end{align*}
We bound the norms of each of the three terms separately: 
\begin{itemize}

    \item According to \Cref{lem:strong_convex_basic}, $\|T^{-1}(y)-T^{-1}(y')\| \leq \frac{2^{a+3}}{\alpha} \|y'-y\|$ for every $y,y'$. Therefore
    \begin{align*}
       \|A_1\|_{L^2(Q)}^2 &=\int \langle \id , T^{-1} - T_0^{-1} \rangle ^2 \dd Q =  \int \langle T_0, T^{-1}\circ T_0 - \id\rangle^2 \dd P  \\
        &\leq c\int \|T_0\|^2\|T_0 - T\|^2 \dd P\leq C\int \|T_0(x) - T(x)\|^2\dotp{x}^{2a+2} \dd P(x),
    \end{align*}
   where we use \Cref{lem:smooth_basic} at the last line. Let $I =\int \|T_0(x) - T(x)\|^2\dd P(x)$. We may apply \Cref{lem:matching_moments} to obtain $\|A_1\|_{L^2(Q)}^2 \leq C I\log_+(1/I)^{2a+2}$.
    
    \item By the Poincaré inequality, $\|A_2\|_{L^2(Q)} = \|\phi_{0,c}\circ T_0^{-1}-\phi_c\circ T_0^{-1}\|_{L^2(Q)} = \|\phi_{0,c}-\phi_c\|_{L^2(P)} \leq \sqrt{C_{\texttt{PI}}}\|T-T_0\|_{L^2(P)}$.

    \item By \Cref{lem:smooth_basic}, it holds that
    \begin{align*}
    \|A_3\|_{L^2(Q)}^2&=\int \|\phi\circ T_0^{-1}-\phi \circ T^{-1}\|^2 \dd Q \leq c \int  \dotp{\|T_0^{-1}\|+\|T^{-1}\|}^{2a+2} \|T_0^{-1}-T^{-1}\|^2 \dd Q \\
    &\leq c \int  \dotp{\|\id\|+\|T^{-1}\circ T_0\|}^{2a+2} \|\id-T^{-1}\circ T_0\|^2 \dd P \\
    &\leq C\int  \dotp{x}^{2a+2} \|T(x)- T_0(x)\|^2 \dd P(x),
    \end{align*}
    where we use \Cref{lem:strong_convex_basic} at the last line. 
    As for the bound on $A_1$, we can use \Cref{lem:matching_moments} to obtain that this term is controlled by $C I\log_+(1/I)^{2a+2}$.
\end{itemize}
In total, $\|\phi_{c}^*-\phi_{0,c}^*\|_{L^2(Q)}^2 \leq C  \log_+(1/I)^{2a+2}I$ for some constant $C$. By \Cref{prop: stab_bound}, the integral $I$ is bounded by
\begin{align*}
   & 2 \|T-\overline T\|_{L^2(P)}^2+2\|T_0-\overline T\|_{L^2(P)}^2  \leq 2\tau^2 +  C\ell^\star \log_+(1/\ell^\star)^{a(a+1)}\leq C\tau^2 \log_+(1/\tau)^{a(a+1)},
\end{align*} 
where we use that $\tau^2\geq \ell^\star$. We obtain
\begin{align*}
     \|\phi_{c}^*-\phi_{0,c}^*\|_{L^2(Q)}^2 &\leq C \tau^2 \log_+(1/\tau)^{a^2+3a+2}.
\end{align*}
The same inequality holds for $\overline\phi$ (as $\overline\phi\in \overline \cF$),  so that we get our final bound on $\|\phi_c^*-\overline\phi_{c}^*\|_{L^2(Q)}$ by using the triangle inequality. 
\begin{align}
    \|\phi_c-\overline\phi_c\|_{L^2(P)} \leq \sqrt{C_{\texttt{PI}}} \|\nabla \phi-\nabla\overline\phi\|_{L^2(P)} \leq \sqrt{C_{\texttt{PI}}}\tau.
\end{align}
Hence,
\[ \|g\|_{L^2(\PP)} \leq \sqrt{C_{\texttt{PI}}}\tau+ C \tau \log_+(1/\tau)^{(a^2+3a+2)/2}.\qedhere\]
\end{proof}

\begin{proof}[Proof of \Cref{lem:strict_brackets}]
Let $y\in \R^d$ and let $\phi_1,\phi_2\in \overline\cF$, with $T_1=\nabla\phi_1$, $T_2=\nabla\phi_2$. It holds that $\phi_1^*(y) = \dotp{x_1,y} - \phi_1(x_1)$ for $x_1=T_1^{-1}(y)$. Therefore,
\begin{align*}
\phi_1^*(y)-\phi_2^*(y) &= (\dotp{x_1,y} - \phi_1(x_1))-\phi_2^*(y) \leq (\dotp{x_1,y} - \phi_1(x_1)) - (\dotp{x_1,y}-\phi_2(x_1)) \\
&\leq |\phi_1(x_1)-\phi_2(x_1)|.
\end{align*}
By symmetry, it holds that $|\phi_1^*(y)-\phi_2^*(y)|\leq |\phi_1(x_1)-\phi_2(x_1)|+|\phi_1(x_2)-\phi_2(x_2)|$, with $x_2=T_2^{-1}(y)$.  Then,
\begin{equation}\label{eq:bound_l2_conjugate}
    \begin{split}
         |\phi_1(x_1)-\phi_2(x_1)| &\leq |\phi_1(T_1^{-1}(y))-\phi_2(T_1^{-1}(y))|\dotp{T_1^{-1}(y)}^{-\eta} \dotp{T_1^{-1}(y)}^{\eta}  \\
    &\leq  \|\phi_1-\phi_2\|_{L^\infty(\dotp{\ \cdot \ }^{-\eta})} \dotp{T_1^{-1}(y)}^{\eta}.
    \end{split}
\end{equation}
By \Cref{lem:strong_convex_basic}, $ \|T_1^{-1}(y)\|\leq 4\alpha^{-\frac{1}{a+1}}( \|y\| + K)^{\frac{1}{a+1}}$, where we recall that $K$ is an upper bound on $\|\nabla\phi(0)\|$ over $\phi\in\cF$. This is smaller than $C_0\dotp{y}^\eta$ for some constant $C_0>0$. Hence, we obtain that
\begin{equation}
    \|\phi_1^*-\phi_2^*\|_{L^\infty(\dotp{\ \cdot \ }^{-\eta})}\leq 2C_0  \|\phi_1-\phi_2\|_{L^\infty(\dotp{\ \cdot \ }^{-\eta})}.
\end{equation}
Thus,
\[\|\gamma(\phi_1,\cdot)-\gamma(\phi_2,\cdot)\|_{L^\infty(\dotp{\ \cdot \ }^{-\eta})} \leq (2C_0+1)  \|\phi_1-\phi_2\|_{L^\infty(\dotp{\ \cdot \ }^{-\eta})}.\]
This implies that for $h>0$
\[ \log \cN(h,\cG,L^\infty(\dotp{\ \cdot \ }^{-\eta}))\leq \log \cN(h/(2C_0+1),\overline\cF,L^\infty(\dotp{\ \cdot \ }^{-\eta}))\]
We then conclude that the second inequality in \Cref{lem:strict_brackets} holds thanks to \Cref{lem:covering_lines}. It remains to prove the first one. We remark that
\begin{align*}
    \|\phi_1-\phi_2\|^2_{L^2(P)} &= \int |\phi_1(x)-\phi_2(x)|^2 \dd P(x) =\int \dotp{x}^{2\eta} \dotp{x}^{-2\eta}|\phi_1(x)-\phi_2(x)|^2  \dd P(x)  \\
    &\leq \int \dotp{x}^{2\eta}\dd P(x) \|\phi_1-\phi_2\|_{L^\infty(\dotp{\ \cdot \ }^{-\eta}}^2
\end{align*} 
and that thanks to \eqref{eq:bound_l2_conjugate}
\begin{align*}
    \|\phi_1-\phi_2\|^2_{L^2(Q)} \leq \|\phi_1-\phi_2\|_{L^\infty(\dotp{\ \cdot \ }^{-\eta})}^2 C_1 \int ( \|y\| + K)^{\frac{2\eta}{a+1}}\dd Q(y)
\end{align*}
for some constant $C_1>0$. 

These two inequalities imply that $\log \cN(h,\cG,L^2(\PP)) \leq \log \cN(C_2h,\cG, L^\infty(\dotp{\ \cdot \ }^{-\eta}))$ for some constant $C_2>0$.
\end{proof}
\begin{proof}[Proof of \Cref{lem:improved_covering}]
Let $\overline \cF_\tau =\{\phi\in\overline\cF:\ \|\nabla\phi-\nabla\overline\phi\|_{L^2(P)}\leq \tau\}$. 
    Let $\phi_1$, $\phi_2\in \overline \cF_\tau$, and write $T_i=\nabla\phi_i$ for $i=1,2$. 
    We have according to \Cref{lem:strong_convex_basic}
    \begin{align*}
        \|\nabla\phi_i^*-\nabla\phi_0^*\|^2_{L^2(Q)}&= \int \|T_i^{-1}\circ T_0(x)-x\|^2 \dd P(x) \\
        &\leq C \int\|T_i(x)- T_0(x)\|^2 \dd P(x) \leq C\tau^2
    \end{align*}
    for some positive constant $C>0$. 
   Hence, the set $\{\phi^*:\ \phi\in \overline\cF_\tau\}$ is included in a ball $B^*_{\tau'}$ of radius $\tau'=\sqrt{C}\tau$  with respect to the pseudo-norm $f\mapsto \|\nabla f\|_{L^2(Q)}$, while $\overline\cF$ is included in a ball $B_\tau$ of radius $\tau$ with respect to the pseudo-norm $f\mapsto \|\nabla f\|_{L^2(P)}$. 
    Consider a $(h/2)$-covering $\phi_1,\dots,\phi_N$ of $B_\tau$ with respect to $L^2(P)$, and a $(h/2)$-covering $\psi_1,\dots,\psi_{N^*}$ of  $B_{\tau'}^*$ with respect to $L^2(Q)$. For $1\leq i \leq N$ and $1\leq j \leq N^*$, let $g_{i,j}:(x,y)\mapsto \phi_i(x)+\psi_j(y)$.  The set of functions $(g_{i,j})_{i,j}$ is a $h$-covering of $\cG$ in $L^2(\PP)$,  that is
    \begin{equation}
        \log\cN(h,\cG,L^2(\PP))\leq\log\cN(h/2,B_\tau,L^2(P))+ \log\cN(h/2,B_{\tau'}^*,L^2(Q)).
    \end{equation}
    We conclude using Condition \ref{cond:covering_gradient}.
\end{proof}

\section{Large parametric spaces}\label{sec: details_large_parametric}
\begin{proof}[Proof of \Cref{prop:rate_large_parametric}]
We write $\cF=\tilde \cF$. Recall that we assume that for some $\eta\geq a+2$ and $\gamma'\geq 0$,
\begin{equation}\label{eq:covering_2_repeat}
\forall h>0,\ \log \cN(h,\cF,L^\infty(\dotp{ \ \cdot\ }^{-\eta})) \leq D \log_+(1/h)^{\gamma'},
\end{equation}
    The proof is exactly the same as the proof of \Cref{thm:strongly_convex}, but we keep track of the dependency with respect to $D$ in the computations. As in \Cref{sec: proofs}, we use \Cref{thm:abstract_one_shot} to bound $\E\|\nabla\hat\phi_{\cF}-\nabla\phi_0\|^2$. Using notations from this section, introduce $\overline\cF=\{\phi_t=t\phi+(1-t)\overline\phi:\ \phi\in\cF, t\in [0,1] \text{ and }\vertiii{\phi_t-\overline\phi}\leq r\}$ for an appropriate value of $r$, and define $\cG=\{\gamma(\phi,\cdot)-\gamma(\overline\phi,\cdot):\ \phi\in\overline\cF, \|\nabla\phi-\nabla\overline\phi\|_{L^2(P)}\leq \tau\}$. To be able to use \Cref{thm:abstract_one_shot}, we require a bound on $\E[|\sup_{g\in \cG}(\PP_n-\PP)(g)|]$.
    
    According to condition \ref{cond:smooth}, the set $\cF$ is included in a ball of radius $\beta$ for the pseudo-norm $\vertiii{f} = \sup_{x\in \R^2} \op{\nabla^2 f(x)}\dotp{x}^{-a}$. According to \cite{triebelcovering} (see Theorem 4.2, region I, and their comments at the beginning of Section 2.4), it holds that for $\eta>a+2$,
    \begin{equation}\label{eq:covering_holder_ball}
    \log \cN(h,\cF,L^\infty(\dotp{\ \cdot \ }^{-\eta}))\leq h^{-\frac{d}{2}}.
\end{equation}
Of course, this bound is very bad for $h$ small (it is non-integrable). However, when $D$ is very large, it might be better than \eqref{eq:covering_2_repeat}. Precisely, we use \Cref{eq:covering_holder_ball} for $h\geq D^{-2/d}$ and \Cref{eq:covering_2_repeat} for $0<h\leq D^{-2/d}$.

Similarly, we have two bounds for the $L^2$-covering numbers of $\cG$: a bound  of order $(h/\tau)^{-d}$ (up to polylogarithmic factors) given by \Cref{lem:improved_covering}, and a bound of order $D\log_+(1/h)^{\gamma'}+\log_+(1/h)$ given by \Cref{lem:strict_brackets}. We will use the latter bound for $0<h\leq \tau D^{-1/d}$, and the former for $h\geq \tau D^{-1/d}$. In total \Cref{prop:van_wellner} yields for all $\tau\geq \sqrt{\ell^\star}$ (and for some $q\geq 0$)
\begin{align*}
  \sqrt{n}\E[|\sup_{g\in \cG}(\PP_n-\PP)(g)|] &\lesssim_{\log_+(1/\tau),\log D} \tau D^{-1/d}D^{1/2} + \int_{\tau D^{-1/d}}^\infty \log_+(1/h)^q\pran{\frac h \tau}^{-d/2} \dd h \\
  &\qquad+ \frac{1}{\sqrt{n}}\pran{ D^{1-2/d} + \int_{D^{-2/d}}^\infty h^{-d/2} \dd h} \\
  &\lesssim_{\log_+(1/\tau),\log D} D^{1/2-1/d}\tau + \frac{D^{1-2/d}}{\sqrt{n}}.
\end{align*}
To conclude, we let $\tau=C(D^{1/2-1/d}n^{-1/2}\log(D)^{q_1}\log(n)^{q_2} + w(\ell^\star))$ for some constants $C,q_1,q_2\geq 0$, so that \eqref{eq:abstract_tau} applies, and gives an excess of risk of order $D^{1/2-1/d}n^{-1/2}$ up to polylogarithmic factors.
\end{proof}

\section{Besov spaces}\label{app:besov}

We now prove the results in \Cref{sec:holder} by recalling relevant facts about Besov spaces, which can be found in many textbooks (see e.g., \cite{gine2021mathematical}). The Besov spaces $B^s_{p,q}$ are a family of functional spaces, with $s$ representing a regularity index, that can be defined using a wavelet basis. A wavelet basis is an orthonormal basis $(\Gamma_{j,k})_{j\in \ZZ,\ k\in\ZZ^d}$ of $L^2(\Rd)$ that satisfy some desirable properties. Let $s>0$. We will always assume that the regularity of the wavelet basis is large enough with respect to $s$ and $d$. Let $1\leq p,q\leq\infty$. If a function $\phi$ admits a wavelet representation
\begin{equation}
    \phi = \sum_{j\in\ZZ} \sum_{k\in \ZZ^d}\gamma_{k,j}\Gamma_{k,j},
\end{equation}
then, the Besov norm $B^s_{p,q}$ of $\phi$ is defined as
\begin{equation}\label{eq:def_besov}
    \|\phi\|_{B^s_{p,q}} = \|\gamma\|_{b^s_{p,q}} = \pran{ \sum_{j\in \ZZ} 2^{jq(s+\frac{d}{2}-\frac{d}{p})} \pran{\sum_{k\in\ZZ^d} |\gamma_{j,k}|^p}^{\frac{q}{p}}}^{\frac{1}{q}},
\end{equation}
with generalizations when $p$ or $q$ equal to $\infty$.

Let $R>0$. Each wavelet $\Gamma_{k,j}$ has compact support of diameter of order $2^{-j}$, and, for a fixed $j$, the number of wavelets $\Gamma_{k,j}$  that intersect the ball $B(0;R)$ is of order $R^d2^{jd}$. Let $I_j(R)$ be the set of indexes $k$ such that this intersection is nonempty. Let $J_0\in\ZZ$ be such that $B(0;R)$ is included in the support of a single wavelet $\Gamma_{k,J_0}$. We can pick $J_0$ such that $J_0 \gtrsim -\log R$. 
Let $\eta(x)=\alpha\|x\|^2/8$. 
We define the set $\cF_J(R)$ as the set of functions of the form 
\[ \eta + \sum_{j=J_0}^J \sum_{k\in I_j(R)}\gamma_{k,j}\Gamma_{k,j}, \]
where, $\|\gamma\|_{b^s_{1,\infty}}\leq L$ for some $L>0$. 

To put it another way, functions $\phi\in \cF_J(R)$ are functions with finite wavelet expansions up to depth $J$ and bounded $B^s_{1,\infty}$ norm on a set of size roughly $R$, whereas we add a quadratic term to ensure that they are strongly convex at infinity. Note however that functions in $\cF_J(R)$ are not necessarily convex. 

Let us show that we can apply \Cref{thm:param} to obtain the rate displayed in \eqref{eq:rate_HR}. By hypothesis, we suppose that $P$ satisfies \ref{cond:poincare}, \ref{cond:loc_poincare} and \ref{cond:loc_doubling}. In particular, \ref{cond:covering_gradient} is satisfied. Furthermore, as the $\cC^2$-norm is controlled by the $B^2_{1,\infty}$-norm and that $s\geq 2$, all potentials in $\cF_J(R)$ are $(\beta,0)$-smooth for some $\beta$ of order $L$, so that \ref{cond:smooth} is satisfied. It remains to check the conditions \ref{cond:parametric} and \ref{cond:parametric_approx}.

Let $\phi_1,\phi_2\in\cF_J(R)$, with associated coordinates $\gamma^1$ and $\gamma^2$. We have for all $x\in\R^d$
\begin{align*}
    |\phi_1(x)-\phi_2(x)| &= |\sum_{j=J_0}^J \sum_{k\in I_j(R)} (\gamma_{j,k}^1-\gamma_{j,k}^2)\Gamma_{k,j}(x)| \\
    &\lesssim R^{d/2}2^{Jd/2}\|\gamma^1-\gamma^2\|_{\infty},
\end{align*}
where we use a standard bound on the $L^\infty$-norm in terms of wavelet coefficients (see e.g.,  \cite[Lemma 24]{hutter2021minimax}). Note that only the coefficients $\gamma^1_{k,j}-\gamma^2_{k,j}$ with $k\in I_j(R)$ are non zero. Therefore, the bracketing number of $\cF_J(R)$ at scale $h$ with respect to $L^\infty(\dotp{\ \cdot \ }^{-\eta})$ (for $\eta>2$) is controlled by the covering number at scale of order $h R^{-d/2}2^{-Jd/2}$ of the $\ell_\infty$-ball in dimension $\sum_{j=-J_0}^J| I_j(R)|\lesssim R^{d}2^{Jd}$. By \cite{schutt1984entropy}, the logarithm of such a covering number is of order
\[ R^{d}2^{Jd} \log(R^{d/2}2^{Jd/2}/h)\lesssim R^{d}2^{Jd} (\log(R) + J)\log_+(1/h),\]
that is condition \ref{cond:parametric} holds.

It remains to check \ref{cond:parametric_approx}.  Let $\gamma^0$ be the vector of coefficients of $\phi_0$ in the wavelet basis, and let $\phi_{0,J} = \eta + \psi_{0,J}$, where
\begin{equation}
    \psi_{0,J}= \sum_{j=J_0}^J \sum_{k\in I_j(R)} \gamma_{j,k}^0\Gamma_{jk}.
\end{equation}
As $s\geq 2$, the norm of the Hessian of $\psi_{0,J}-\phi_0$ at $x\in B(0;R)$, is bounded by $2^{(J_0-J)2}\|\phi_0\|_{\cC^s} \leq \alpha/4$ for $J$ large enough. Also, as $\vertiii{\eta} \leq \alpha/4$, we obtain that $\vertiii{\phi_{0,J}-\phi_0}\leq \alpha/2$.

Furthermore
\begin{align*}
    &\|\nabla\phi_0-\nabla\phi_{0,J}\|_{L^2(P)}^2= \int \|\nabla\phi_0(x)-\nabla\phi_{0,J}(x)\|^2 \dd P(x)\\
    &\leq \|p\|_\infty\int_{\|x\|\leq R} \|\nabla\phi_0(x)-\nabla\phi_{0,J}(x)\|^2 \dd x +\int_{\|x\|> R} \|\nabla\phi_0(x)-\nabla\phi_{0,J}(x)\|^2 \dd P(x).
\end{align*} 
The first integral is bounded by a quantity of order $2^{2(J_0-J)(s-1)}\|\phi_0\|_{\cC^s}^2$ (see e.g.,  \cite[Lemma 13]{hutter2021minimax}). The second integral is of order $e^{-cR}$ as $P$ is subexponential. We pick $R =2J(s-1)/c$ to conclude that \ref{cond:parametric_approx} holds.

\section{On the spectrum of kernels for RKHS}\label{app: rkhs_spec}
Here we provide some details about the spectrum of kernels, following the discussion of \cref{sec: rkhs}. To this end, recall that a kernel $\cK : \cX \times \cX \to \R$ defines a Reproducing Kernel Hilbert Space (RKHS) if for all $f \in \cH$ (the Hilbert space over $\cX$) and all $x \in \cX$, it holds that 
\begin{align*}
    f(x) = \langle f, \cK(\cdot,x)  \rangle_\cH\,.
\end{align*}
In \cref{sec: rkhs}, we provided log-covering bounds from \cite{yang2020function} that required that the kernel have either finite spectrum, or exponentially decaying eigenvalues; the following is borrowed from \cite[Section 2.2 and Appendix D]{yang2020function}. To this end, we note that $\cK$ induces an integral operator $T_\cK : L^2(\cX) \to L^2(\cX)$ defined as
\begin{align*}
    T_\cK f(y) = \int \cK(y,x) f(x) \dd x\,,
\end{align*}
for all $f \in L^2(\cX)$. Mercer's theorem (see e.g.,  \cite{scholkopf2002learning}) states that this operator has countably many, positive eigenvalues with corresponding eigenfunctions $\{(\sigma_i,\vartheta_i)\}_{i=1}^\infty$ which form an orthonormal basis of $L^2(\cX)$, whence the kernel admits the following spectral decomposition
\begin{align}
    \cK(x,y) = \sum_{i=1}^n \sigma_i \vartheta_i(x)\vartheta_i(y)\,.
\end{align}
The two conditions we presented are the following:
\begin{enumerate}
    \item \textit{$\gamma$-finite spectrum}: $\sigma_j = 0$ for all $j \geq \gamma$,
    \item \textit{$\gamma$-exponential decay}: there exist absolute constants $C_1$ and $C_2$ such that for all $j \geq 1$, $\sigma_j \leq C_1 \exp(-C_2 j^\gamma)$,
\end{enumerate}
where for the latter, we also require that there exist constants $\tau \in [0,\tfrac12), C_\vartheta > 0$ such that $\sup_{x \in \cX} \sigma_j^\tau |\vartheta_j(x)| \leq C_\vartheta $ for all $j \geq 1$. In fact, over the sphere in $\Rd$, it holds that the Gaussian kernel exhibits $\gamma$-exponential decay for any $\tau  > 0$ \cite[Appendix B.3]{yang2020function}, wherein several more examples of kernels are presented. The following is a proposition from \cite{yang2020function}, which is an upper bound on the covering numbers we employ in \Cref{sec: rkhs}.
\begin{prop}\label{prop: rkhs_prop}
Let $\cF_\cK$ be the unit ball in an RKHS corresponding to some kernel $\cK$ i.e.,  $\cF_{\cK} \defeq \{ f : \|f\|_{\cH} \leq 1\}$, and let $\cK$ be a kernel with $\gamma$-finite spectrum, or $\gamma$-exponentially decaying spectrum, as described above, such that $\sup_z \cK(z,z)\leq 1$. Then, 
\begin{align*}
    \log \cN(h,\cF_\cK,\|\cdot\|_\infty) \leq
    \begin{cases}
    C \gamma (\log(1/h) + c) & \gamma\text{-finite spectrum}\,, \\
    C (\log(1/h) + c)^{1+1/\gamma} & \gamma\text{-exponentially decaying spectrum}\,,
    \end{cases}
\end{align*}
where the absolute constants $c,C > 0$ depend on $C_\vartheta, C_1,C_2,\gamma$ and $\tau$.
\end{prop}

\section{Spiked transport model}\label{sec: proof_spiked}
\begin{proof}[Proof of \Cref{prop:spiked}]
Recall that
\begin{align*}
    \cF_{\text{spiked}} = \left\{ x \mapsto \phi(Ux) + \tfrac12 x^\top(I - U^\top U)x \ : \ \phi \in \cF, U \in \cV_{k \times d} \right\}\,,
\end{align*}
where $\cV_{k \times d}$ is the Stiefel manifold, and $\cF$ is some function class. Log-covering number bounds for Stiefel manifold are already known (see \cite[Lemma 4]{niles2022estimation}):
\begin{align}\label{eq: covering_stiefel}
    \log \cN(h, \cV_{k\times d}, \|\cdot\|_{\text{op}}) \leq dk \log(c \sqrt{k}/h)\,,
\end{align}
where $c > 0$ is some absolute constant. The same bound holds for the space $L^\infty(\dotp{\ \cdot \ }^{-\eta})$ for $\eta\geq 2$ up to a constant depending on $\eta$, as in our previous calculations. To complete the proof, we can decompose a function $\phi' \in \cF_{\text{spiked}}$ into two parts; $\phi \circ U$ and a quadratic. The quadratic part contributes the same degree of complexity as a parametric class (see \cref{sec:parametric}): indeed we have that
\begin{equation}\label{eq:stiefel_bound}
\begin{split}
     \| x^\top(I - U^\top U)x - x^\top(I - V^\top V)x\| &= \|x^\top(U+V)^\top(U-V)x\| \\
    &\leq \|x\|^2 \|U+V\|_F\|U-V\|_F \\
    &\leq C_d \|x\|^2 \op{U-V},
\end{split}
\end{equation}
using that $\op{U},\op{V}\leq 1$ and that the Frobenius and operator norms are equivalent. 
Thus, we can apply known bounds on the covering numbers for the Stiefel manifold to determine the size of the (log-)covering numbers that the quadratic contributes to the class $\cF_{\text{spiked}}$. Finally, log-covering numbers for the composition $\phi \circ U$ can simply be bounded by considering the sum of the log-covering numbers of the two classes.  Thus, applying \cref{eq: covering_stiefel}, completes the proof.
\end{proof}

\begin{proof}[Proof of \Cref{prop:spiked_gradient}]
For $U\in \cV_{k\times d}$, let $P_U=U_\sharp P$. Remark the uniform measure $P_0$ on  a convex domain is a log-concave measure. Therefore, so is its  projection $(P_0)_U$ on a subspace. The distribution $P_U$ has a density bounded away from zero and infinity with respect to $(P_0)_U$. This implies that $P_U$ satisfies \ref{cond:loc_poincare} and \ref{cond:loc_doubling} for constants independent of $U\in \cV_{k\times d}$. Hence, the measure $P_U$ satisfies \ref{cond:covering_gradient} with $m=k$ for the class of functions $\cF$. To ease notation, we write $A(\phi,U): x\mapsto \phi(Ux) + \tfrac12 x^\top(I - U^\top U)x$ for $\phi\in\cF$ and $U\in\cV_{k\times d}$, so that 
\[ \cF_{\mathrm{spiked}} = \{A(\phi,U):\ \phi\in\cF, U\in\cV_{k\times d}\}.\]
 As in \ref{cond:covering_gradient}, let $\overline f\in \cF_{\mathrm{spiked}}$, let $\cG\subseteq  \{tf+(1-t)\overline f:\ f\in \cF_{\mathrm{spiked}},\ 0\leq t \leq 1\}$ be a set of $(\alpha,a)$-strongly convex potentials and let $\cG^*=\{g^*:\ g\in \cG\}$.  Let $B_\tau$ (resp.~$B^*_\tau$) be a ball of radius $\tau$ in $\cG$ for the pseudo-norm $g\mapsto \|\nabla g\|_{L^2(P)}$ centered at $\bar{f}$ (resp.~in $\cG^*$ for the pseudo-norm $g^*\mapsto \|\nabla g^*\|_{L^2(Q)}$ centered at $\bar{f}^*$). First, note that as in the proof of \Cref{lem:C2_satisfied}, we have for all $g\in B_\tau$, $\|\nabla g(0)-\nabla\overline f(0)\|\leq K$ where $K\lesssim 1+\tau^2$. 

 By assumption, if $tA(\phi,U)+(1-t)\overline f\in B_\tau$, then
 \[ t^2\E[\|U^\top \nabla\phi(UX)+(I-U^\top) X - \nabla \overline f(X)\|^2 ]\leq \tau^2.\]
 Let $\overline f_U$ be defined by $\overline f_U(y)= \E[\overline f(X)-\tfrac12 X^\top(I - U^\top U)X|UX=y]$. Then, using properties of the conditional expectation,
\begin{align*}
 &t^2\E[\|U^\top \nabla\phi(UX)-U^\top \nabla\overline f_U(UX)\|^2] \\
 &\qquad \leq t^2\E[\|U^\top \nabla\phi(UX)-\nabla(\overline f(X)-\tfrac12 X^\top(I - U^\top U)X) \|^2] \leq \tau^2. 
\end{align*}
 Hence, the set of functions $\{t\phi+(1-t)\overline f_U:\ tA(\phi,U)+(1-t)\overline f\in B_\tau\}$ is included in the ball $B_\tau^U$ centered at $\overline f_U$ in $L^2(P_U)$ of radius $\tau$. 
For each $U\in \cV_{k\times d}$, let $\cC_U$ be a minimal $(h/2)$-covering of $B^U_{\tau}$ with respect to $L^2(P_U)$. Also, for some parameter $\lambda>0$ to fix, let $\cU$ be a minimal $\lambda h$-covering of the Stiefel manifold $\cV_{k\times d}$ with respect to the operator norm and let $\cT$ be a  minimal $\lambda h$-covering of $[0,1]$.

Let $tA(\phi,U)+(1-t)\overline f\in B_\tau$. Let $t_0$ be the element in $\cT$ the closest to $t$, let $U_0$ be the element in $\cU$ the closest to $U$, and let $\phi_0$ be the element in $\cC_{U_0}$ the closest to $\phi$. Then, according to \Cref{eq:stiefel_bound},
\begin{align*}
   & \|tA(\phi,U)+(1-t)\overline f - t_0A(\phi_0,U_0)+(1-t_0)\overline f\|_{L^2(P)} \leq |t-t_0|\|A(\phi,U)-\overline f\|_{L^2(P)}\\
    &\qquad + \|\phi(U\cdot)-\phi(U_0 \cdot)\|_{L^2(P)} + \|\phi-\phi_0\|_{L^2(P_{U_0})} +  \op{U-U_0} \pran{\int C_d^2 \|x\|^4 \dd P(x)}^{1/2}.
\end{align*}
According to \Cref{lem:smooth_basic}, $\|\phi(U\cdot)-\phi(U_0 \cdot)\|_{L^2(P)} \lesssim (1+K)\op{U-U_0}$ and $\|A(\phi,U)-\overline f\|_{L^2(P)}\lesssim (1+K)$. Hence,
\begin{align*}
   & \|tA(\phi,U)+(1-t)\overline f - t_0A(\phi_0,U_0)+(1-t_0)\overline f\|_{L^2(P)} \lesssim (1+K)\lambda h+ h/2 \leq h.
\end{align*}
if $\lambda \leq c/(1+K)$ for some $c>0$ small enough. We have created a $h$-covering of $B_\tau$, of size
\[ |\cT|\times \prod_{U \in \cU} |\cC_{U}|.\]
  Recall that the measure $P_U$ satisfies \ref{cond:covering_gradient} with $m=k$ for the class of functions $\cF$. Hence, by \eqref{eq: covering_stiefel}, the log-size of this covering satisfies
  \[ \log\pran{|\cT|\times \prod_{U \in \cU} |\cC_{U}|}\lesssim_{\log_+(1/\tau),\log_+(1/h)} \pran{\frac{h}\tau}^{-k}.\]
\medskip

Let us now bound the covering numbers of $B_\tau^*$. Let $f=A(\phi,U)$. Then, for $y\in \R^d$,
\begin{align*}
    f^*(y) &= \sup_{x\in\R^d} \dotp{x,y}-\phi(Ux) -\frac 12 x^\top (I-U^\top U)x \\
    &=\sup_{x\in\R^d} \dotp{U^\top Ux,y}-\phi(Ux) + \dotp{(I-U^\top U)x,y}-\frac 12 x^\top (I-U^\top U)x \\
    &= \phi(U\cdot)^*(Uy) + \frac{1}{2} y^\top  (I-U^\top U)y,
\end{align*}
where we use the expression for the convex conjugate of a quadratic function. Hence, if $f=A(\phi,U)$ is a spiked potential, then so is $f^*$, with the same spike. Unfortunately, the ball $B_\tau^*$ also contains dual potentials of the form $(tA(\phi,U)+(1-t)A(\overline\phi,\overline U))^*$. There is no direct expression for this dual potential, although we can use the fact that $tA(\phi,U)+(1-t)A(\overline\phi,\overline U)$ is a spiked potential with spike of dimension $\dim(U+\overline U)\leq 2k$. Hence, $B_\tau^*$ is a family of spiked potentials with spikes of dimension $2k$.
Furthermore, for any $U\in \cV_{2k\times d}$, as the law $P_U$ satisfies \ref{cond:loc_poincare} and \ref{cond:loc_doubling}, \Cref{lem:C2_satisfied} ensures that the law $Q_U = (\nabla\phi_0)_{\sharp} P_U$ satisfies the estimates of condition \ref{cond:covering_gradient}. This is enough to conclude using the same strategy as the one used for bounding the covering numbers of $B_\tau$.
\end{proof}

\section{Barron spaces}\label{app:barron}

\begin{proof}[Proof of \Cref{thm:barron}]
By construction, every $\phi\in \cF_\sigma^1$ is $(\beta,a)$-smooth, so that assumptions \ref{cond:smooth}. Furthermore, \ref{cond:covering_gradient} is satisfied thanks to \Cref{lem:C2_satisfied}.  To conclude, it remains to check that \ref{cond:covering_convex} is satisfied by bounding the covering numbers of $\cF$.
\end{proof}

\begin{lem}\label{lem:bracket_barron}
Let $\eta>a+2$. It holds that, for $h>0$,
\begin{equation}
    \log \cN(h,\cF^1_\sigma, L^\infty(\dotp{\ \cdot \ }^{-\eta})) \lesssim_{\log_+(1/h)} h^{-\frac{2m}{2s+m}}.
\end{equation}
\end{lem}

Let $B$ be a Banach space. We say that $B$ is of type $2$ if there exists a constant $T_B>0$ such that for all $x_1,\dots,x_n\in B$,
\begin{equation}
    \E \left[ \left\|\sum_{i=1}^n \eps_i x_i\right\|_B^2\right] \leq T_B^2 \sum_{i=1}^n \|x_i\|^2_B,
\end{equation}
where $\eps_1,\dots,\eps_n$ are i.i.d. Rademacher variables. Siegel and Xu show in \cite{siegel2021sharp} that, if $\sigma:\cM\to B$ is of regularity $s$ for some Banach space $B$ of type $2$ (where regularity is defined through dual maps, see \cite[Definition 2]{siegel2021sharp} for details), then the covering numbers of $\cF_\sigma^1$ satisfy for $h>0$
\begin{equation}\label{eq:control_siegel}
    \log \cN(h,\cF_\sigma^1,B) \leq C\pran{\frac h{T_B}}^{-\frac{2m}{2s+m}}.
\end{equation}
Such a result is almost enough to conclude. Indeed, the Banach space $L^\infty(\dotp{\ \cdot \ }^{-\eta})$ is not of type $2$. We use an approximation scheme, that shows that the covering numbers for the $\infty$-norm cannot be too different from the covering numbers for the $p$-norm, where $p$ is large. As spaces $L_p$ are of type $2$, this gives us the conclusion. 

\begin{lem}
Let $\cF$ be a set of $(\beta,a)$-smooth functions, with $\sup_{f\in\cF}\|\nabla f(0)\|\leq M$ and $f(0)=0$ for all $f\in\cF$. Assume that $\cF$ satisfies that, for all finite measures $\rho$ and $2\leq p<\infty$,
\begin{equation}\label{eq:type2_control}
    \log \cN(h,\cF,L_p(\rho))\leq C \pran{\frac{h}{T_{L_p(\rho)}}}^{-\gamma}.
\end{equation}
Then, for all $\eta>a+2$ and for all $h>0$,
\begin{equation}
    \log \cN(h,\cF,L^\infty(\dotp{\ \cdot \ }^{-\eta})) \lesssim_{\log_+(1/h)} h^{-\gamma},
\end{equation}
up to a constant depending on all the parameters involved, including $\eta$ and $a$.
\end{lem}

Before writing a proof, let us remark that one can use this lemma with \eqref{eq:control_siegel} to obtain that \ref{cond:covering_convex} holds with $\gamma=\frac{2m}{2s+m}$, which gives \Cref{thm:barron} through \Cref{thm:strongly_convex}. 

\begin{proof}
Let us first recall Khintchine's inequality. It states that for any finite measure $\rho$ and $2\leq p <\infty$, the space $L_p(\rho)$ is of type $2$ with $T_{L_p(\rho)} \leq C \rho(\Rd)^{1/p} \sqrt{p}$, for some absolute constant $C$ (see e.g.,  \cite[Section 2.6]{vershynin2018high}). 
Let $\rho_R$ be the Lebesgue measure on $B(0;R)$ and let $\{f_1,\dots,f_K\}$ be an $h'$-covering of $\cF$ for the norm $L_p(\rho_R)$  for some parameter $h'$ to fix. By assumption, we can pick $K$ such that 
\[\log K\leq C \pran{\frac{h'}{\rho_R(\Rd)^{1/p} \sqrt{p}}}^{-\gamma}.\] 
Let $f\in \cF$, and let $f_k$ be such that $\|f-f_k\|_{L_p(\rho_R)}\leq h'$. By \Cref{lem:smooth_basic}, we have
\begin{equation}
    \|f-f_k\|_{L^\infty(\dotp{\ \cdot \ }^{-\eta})} \leq \|f-f_k\|_{L^\infty(\rho_R)} + 4(2\beta+M)\dotp{R}^{a+2-\eta}.
\end{equation}
We pick $R=ch^{-1/(a+2-\eta)}$ for $c$ large enough, so that the second term is smaller than $h/2$. Let us bound the first term. According to \Cref{lem:smooth_basic}, the function $f-f_k$ is $L$-Lipschitz continuous on $B(0;R)$, with $L=4\beta \dotp{R}^{a+1}+2M$. Let $x_0\in B(0;R)$ be such that $m=\sup_{\|x\|\leq R} |f(x)-f_k(x)|$ is attained at $x_0$. Therefore, $|f(x)-f_k(x)|\geq m-Lr$ for $x\in B(x_0;r)$. Then, for $r\leq R$
\begin{align*}
    (h')^p \geq \int_{B(0;R)} |f(x)-f_k(x)|^p \dd x \geq \int_{B(x_0;r)\cap B(0;R)}(m-Lr)^p\dd x \geq c_d r^d(m-Lr)^p.
\end{align*}
We pick $r=m/(2L)$ (that is smaller than $R$ for $h$ small enough) to obtain that
$h' \geq \pran{\frac{c_{d}}{(2L)^d}}^{\frac 1 p} \frac{m^{1+\frac{d}{p}}} 2$. Pick 
\[ h' = \pran{\frac{c_{d}}{(2L)^d}}^{\frac 1 p} \frac{(h/2)^{1+\frac{d}{p}}} 2\]

to obtain that $m= \|f-f_k\|_{L^\infty(\rho_R)} \leq h/2$. Therefore, the set $\{f_1,\dots,f_K\}$ is a $h$-covering of $\cF$ for the norm $L^\infty(\dotp{\ \cdot \ }^{-\eta})$ and we have shown that 
\begin{align*}
    \log \cN(h,\cF,L^\infty(\dotp{\ \cdot \ }^{-\eta})) &\leq \log K\leq C \pran{\frac{h'}{\rho_R(\Rd)^{1/p} \sqrt{p}}}^{-\gamma} \\
    &\leq  C\pran{\frac{ \pran{\frac{c_{d}}{(2L)^d}}^{\frac 1 p} (h/2)^{1+\frac{d}{p}}}{2(c_d' R^d)^{1/p} \sqrt{p}}}^{-\gamma}.
\end{align*} 
We pick $p=\log_+(1/h)$ to conclude, using that $h^{-1/p}\lesssim 1$ and that $R$ and $L$ grow polynomially in $h^{-1}$.
\end{proof}

At last, we prove \Cref{prop:sphere}.

\begin{proof}[Proof of \Cref{prop:sphere}]
The only difference with \Cref{thm:barron} consists in showing that \ref{cond:covering_gradient} holds with $m=d-1$.

As in Condition \ref{cond:covering_gradient}, let $\cG\subseteq  \{t\phi+(1-t)\phi':\ \phi,\phi'\in \cF,\ 0\leq t \leq 1\}$ be a set of $(\alpha,a)$-strongly convex potentials and let $\cG^*=\{g^*:\ g\in \cG\}$. Let $B_\tau$ (resp.~$B^*_\tau$) be a ball of radius $\tau$ in $\cG$ for the pseudo-norm $g\mapsto \|\nabla g\|_{L^2(P)}$ (resp.~in $\cG^*$ for the pseudo-norm $g^*\mapsto \|\nabla g^*\|_{L^2(Q)}$).

Let us prove that Condition \ref{cond:covering_gradient} is satisfied with $m=d-1$. Let $\phi_1,\phi_2\in \cG$. Note that for $x=r\theta$ with $r\geq 0$, we have $\phi_1(x)=r^s\phi_1(\theta)$. Hence,
   \begin{equation}\label{eq:lip_barron}
       \begin{split}
           \|\phi_1-\phi_2\|_{L^2(P)}^2 \leq &\int r^{2s} |\phi_1(u)-\phi_2(u)|^2 \dd P(ru)\\
       &\leq M_{\max} \|\phi_1-\phi_2\|_{L^2(P_U)}^2.
       \end{split}
   \end{equation}
 As we have
\begin{align*}
    \tau^2 \geq \|\nabla\phi_1-\nabla\phi_2\|_{L^2(P)}^2 = \int r^{2s-2}\|\nabla\phi_1(u)-\nabla\phi_2(u)\|^2 \dd P(ru) \geq M_{\min} \|\nabla\phi_1-\nabla\phi_2\|_{L^2(P_U)}^2,
\end{align*}
the ball $B_\tau$ is included in a ball $B'_\tau$ of radius $\tau/\sqrt{M_{\min}}$ for the pseudo-norm $g\mapsto \|\nabla g\|_{L^2(P_U)}$. By \eqref{eq:lip_barron}, the function 
\begin{align*}
   (B_\tau, \|\cdot\|_{L^2(P)}) &\to  (B'_\tau \|\cdot\|_{L^2(P_U)})\\
   \phi &\mapsto \phi_{|\S^{d-1}}
\end{align*}

is Lipschitz continuous. As $P_U$ satisfies \ref{cond:loc_poincare} and \ref{cond:loc_doubling}, it holds by \Cref{lem:C2_satisfied} that for all $h>0$,
\[ \log \cN(h,B'_\tau,L^2(P_U)) \lesssim \pran{\frac{h}{\tau}}^{-(d-1)}.\]
Therefore, the covering numbers of $B_\tau$ satisfy a similar inequality.

Let us now bound the covering numbers of $B^*_\tau$. Write $T_1=\nabla\phi_1$ and $T_2=\nabla \phi_2$. As in the proof of \Cref{lem:strict_brackets}, we remark that for $y\in\R^d$,
\[|\phi_1^*(y)-\phi_2^*(y)|\leq |\phi_1(T_1^{-1}(y))-\phi_2(T_1^{-1}(y))|+|\phi_1(T_2^{-1}(y))-\phi_2(T_2^{-1}(y))|.\] 
Remark that if $\phi\in \cG$, then $\nabla\phi(ru) = r^{s-1}\nabla\phi(u)$ and $(\nabla\phi)^{-1}(ru) = r^{1/(s-1)} (\nabla\phi)^{-1}(u)$. For $u\in \S^{d-1}$, we let $r(u)=\| T_1^{-1}\circ T_0(u)\|^{-1}$, so that  $\| T_1^{-1}\circ T_0(r(u)u)\|=1$.
\begin{equation}
    \begin{split}
     \int  & |\phi_1(T_1^{-1}(T_0(x)))-\phi_2(T_1^{-1}(T_0(x)))|^2 \dd P(x) \\
     &\qquad =\int   |\phi_1(T_1^{-1}(T_0(ru)))-\phi_2(T_1^{-1}(T_0(ru)))|^2 \dd P(x)\\
     &\qquad \leq \int \pran{\frac{r}{r(u)}}^{2s} |\phi_1(T_1^{-1}\circ T_0(r(u)u))- \phi_2(T_1^{-1}\circ T_0(r(u)u))|^2 \dd P(ru)\\
    & \qquad \leq \int \pran{r\| T_1^{-1}\circ T_0(u)\|}^{2s} |\phi_1(T_1^{-1}\circ T_0(r(u)u))- \phi_2(T_1^{-1}\circ T_0(r(u)u))|^2 \dd P(ru). 
    \end{split}
\end{equation}
By smoothness and strong convexity, $\| T_1^{-1}\circ T_0(u)\|$ is bounded over $u\in \S^{d-1}$. Hence,
\begin{equation}
    \begin{split}
     \int  & |\phi_1(T_1^{-1}(T_0(x)))-\phi_2(T_1^{-1}(T_0(x)))|^2 \dd P(x)\\
     &\qquad \lesssim CM_{\max}\int |\phi_1(T_1^{-1}\circ T_0(r(u)u))- \phi_2(T_1^{-1}\circ T_0(r(u)u))|^2 \dd P_U(u). 
    \end{split}
\end{equation}
We use the change of variables $\Psi(u)= r(u)T_1^{-1}\circ T_0(u)$. The function $\Psi$ is one-to-one. Indeed, if $u$ and $u'$ are such that $T_1^{-1}\circ T_0(u)$ is positively colinear to $T_1^{-1}\circ T_0(u')$, one can see by composing by $T_1$ and using the homogeneity properties of $T_1$ and $T_0$ that $u$ and $u'$ are also positively colinear (and therefore equal as they both belong to $\S^{d-1}$). As $\phi_1$ and $\phi_0$ are smooth and strongly convex on a ball, the function $\Psi$ is therefore a diffeomorphism onto its image,  whose Jacobian is bounded away from zero and infinity. As $P_U$ is assumed to have a density bounded away from zero on infinity on $\S^{d-1}$, this shows that 
\[  \int |\phi_1(\Psi(u))- \phi_2(\Psi(u))|^2 \dd P_U(u) \lesssim \|\phi_1-\phi_2\|^2_{L^2(P_U)}.\]
Applying the same argument to $  \int   |\phi_1(T_2^{-1}(T_0(x)))-\phi_2(T_2^{-1}(T_0(x)))|^2 \dd P(x)$, we obtain that
\[ \|\phi_1^*-\phi_2^*\|_{L^2(Q)} \lesssim \|\phi_1-\phi_2\|_{L^2(P_U)}.\]
A similar argument shows that $\|\nabla\phi_1^*-\nabla\phi_2^*\|\gtrsim \|\nabla\phi_1-\nabla\phi_2 \|_{L^2(P_U)}$. Hence, we conclude as for $B_\tau$ that the the covering numbers of $B_\tau^*$ satisfy the condition in \ref{cond:covering_gradient}.

Finally, we prove the minimax lower bound. Consider  $\cC^{s+\frac{d+1}2}(B(0;1))$ to be the space of functions of regularity $s+\frac{d+1}{2}$ on the unit ball $B(0;1)$ in $\R^d$. According to \cite[Theorem 2 and Remark 3]{hutter2021minimax}, the minimax rate of estimation for potentials $\phi\in \cC^{s+\frac{d+1}2}(B(0;1))$ is of order at least $n^{-\frac{2s+d-1}{2s+2d-3}}$. According to Proposition 5 and the homogeneous reformulation in \cite{bach2017breaking}, the space $\cC^{s+\frac{d+1}2}(B(0;1))$ is a subset of $\cF^1_\sigma$. Therefore, the minimax lower bound on $\cF^1_\sigma$ is larger than the one on $\cC^{s+\frac{d+1}2}(B(0;1))$, giving the minimax lower bound.
\end{proof}

\section{Proofs: Computational aspects}\label{sec: comp_proofs}
\begin{lem}\label{lem:conjugate_convex}
	For any $\phi_0, \phi_1$ and $\lambda \in [0, 1]$, the inequality
	\begin{equation*}
		((1-\lambda) \phi_0 + \lambda \phi_1)^* \leq (1-\lambda)\phi_0^* + \lambda \phi_1^*
	\end{equation*}
	holds.
\end{lem}
\begin{proof}
The proof follows from subadditivity of the supremum:
\begin{align*}
    ((1-\lambda) \phi_0 + \lambda \phi_1)^*(y) &= \sup_{x} \ \langle x , y \rangle - (1-\lambda)\phi_0(x) - \lambda \phi_1(x) \\
    &= \sup_{x} \ (1-\lambda)\left\{\langle x , y \rangle - \phi_0(x)\} + \lambda\{ \langle x , y \rangle - \phi_1(x)\right\} \\
    &\leq (1-\lambda)\phi_0^*(y) + \lambda \phi_1^*(y)\,.\qedhere
\end{align*}
\end{proof}
\begin{proof}[Proof of \cref{thm: Sn_convex_smooth}]
It suffices to prove convexity and smoothness of the functional $S_n(\cdot)$; as the constraint set $\Delta_J$ is clearly convex. The runtime result is standard (see \cite[Theorem 3.7]{Bub15} for example).

The convexity of $S_n(\lambda)$ in $\lambda$ is inherited from the convexity of $S_n(\phi)$ in $\phi$; see \cref{lem:conjugate_convex}.

To prove smoothness, we assume all the functions $\phi_j$ are strongly convex on $\Omega$.
Since $\Omega$ is compact and convex functions are locally Lipschitz, we may also assume that each $\phi_j$ is Lipschitz on $\Omega$.
Recall that for a given $\lambda \in \R^J_+$, the gradient of the $j^{\text{th}}$ component reads
\begin{align*}
        [\nabla_\lambda S_n(\lambda)]_j = \frac{1}{n}\sum_{i=1}^n \phi_{j}(X_i) - \frac{1}{n}\sum_{k=1}^n \phi_{j}(x^*_k)\,,
\end{align*}
where $x^*_k$ is the maximizer in \cref{eq: conj_maximizer}.
This expression depends on $\lambda$ only through $x^*_k$.
Thus, for fixed $y = Y_k$, it suffices to show that there exists $L_J$ such that
\begin{align}
    |\phi_j(x^*_\lambda) - \phi_j(x^*_\eta)| \leq L_J \| x^*_\lambda - x^*_\eta\| \,, 
\end{align}
where $x^*_\lambda$ and $x^*_\lambda$ are the maximizers for two different parameters $\lambda$ and $\eta$. Note that
\begin{align*}
    y = \nabla \phi_\lambda(x^*_\lambda) = \nabla \phi_\eta(x^*_\eta)\,,
\end{align*}
by virtue of $x^*_\lambda$ (resp. $x^*_\eta$) being maximizers to the conjugate objective. Since $\phi_\lambda$ and $\phi_\eta$ are convex combinations of strongly convex functions, they are both strongly convex, and the following two inqualities hold for some $\alpha > 0$:
\begin{align*}
    & \phi_\lambda(x^*_\eta) \geq \phi_\lambda(x^*_\lambda) + \langle y , x^*_\eta - x^*_\lambda \rangle + \frac{1}{2\alpha}\|x^*_\lambda - x^*_\eta\|^2 \\
    & \phi_\eta(x^*_\lambda) \geq \phi_\eta(x^*_\eta) + \langle y , x^*_\lambda - x^*_\eta \rangle + \frac{1}{2\alpha}\|x^*_\lambda - x^*_\eta\|^2\,.\\
\end{align*}
Adding these two inequalities and using the fact that each $\phi_j$ is $L$-Lipschitz for some $L$ gives the following chain of inequalities:
\begin{align*}
    \frac{1}{\alpha}\|x^*_\lambda - x^*_\eta\|^2 &\leq [\phi_\lambda - \phi_\eta](x^*_\eta) + [\phi_\eta - \phi_\lambda](x^*_\lambda) \\
    &= \sum_{j=1}^J (\lambda_j - \eta_j) (\phi_j(x^*_\eta) - \phi_j(x^*_\lambda)) \\
    &\leq L\|x^*_\lambda - x^*_\eta\|\sum_{j=1}^J|\lambda_j - \eta_j| \\
    &\leq \sqrt{J} L\|x^*_\lambda - x^*_\eta\|\|\lambda-\eta\|_2\,.
\end{align*}
Altogether, this reads
\begin{align*}
    \|x^*_\lambda - x^*_\eta\| \leq \sqrt{J}\alpha L \|\lambda-\eta\|_2\,.
\end{align*}
Thus, the Lipschitz constant of each coordinate $[\nabla_\lambda S_n(\lambda)]_j$ is $\sqrt{J}\alpha L$, and $\nabla_\lambda S_n(\lambda)$ is $L_J$ Lipschitz for $L_J = J \alpha L$.
\end{proof}
\subsection{Proof of Theorem~\ref{thm:hard}}
Following the notation of~\cite{BruRegSon21}, we first define the ``homogeneous CLWE (hCLWE)'' distribution.
Given a unit vector $u \in \RR^d$ and parameters $\beta, \gamma > 0$, denote by $Q_{u, \beta, \gamma}$ the distribution on $\RR^d$ with density proportional to
\begin{equation}\label{eq:q_density}
	e^{-\|x\|^2/2} \cdot \sum_{k \in \ZZ} e^{- (k - \gamma \langle u, x \rangle)^2/2\beta^2}\,.
\end{equation}
We denote by $Q_0$ the standard Gaussian distribution on $\RR^d$.

We require two lemmas about the hCLWE distribution.
\begin{lem}\label{lem:w2_far}
	Assume $\gamma \geq 1$.
	There exist universal constants $c, C > 0$ such that if $\beta < c$, then $W_2^2(Q_0, Q_{u, \beta, \gamma}) \geq C \gamma^{-2}$.
\end{lem}
\begin{lem}\label{lem:hclwe_smooth_map}
	For any $\beta > 0$, there exist $c = c(\beta)$ and $C = C(\beta)$ such that the optimal transport map between $P = \cN(0, I_d)$ and $Q_{u, \beta, \gamma}$ is given by $\nabla \phi_0$ for some $\phi_0 \in \cF_{\mathrm{spiked}, 1}(c, C)$.
\end{lem}

With these lemmas in hand, we can now complete the proof.
\cite[Lemma 4.1 and Corollary 4.2]{BruRegSon21} implies that there exists a polynomial-time reduction from CLWE and a polynomial-time quantum reduction from GapSVP to the following Bayesian testing problem: for any constant $\beta \in (0, 1)$ and $\gamma \geq 2\sqrt{d}$ and $n = \mathrm{poly}(d)$, distinguish between $n$ samples from $Q_0$ and $n$ samples from $Q_{u, \beta, \gamma}$, where $u$ is uniformly distributed on the unit sphere in $\RR^d$.

Fix $\gamma = 2 \sqrt{d}$, and let $\beta$ be small enough that Lemma~\ref{lem:w2_far} guarantees that $W_2(Q_0, Q_{u, \beta, \gamma}) \geq \frac{8c'}{\sqrt{d}}$ for some constant $c' > 0$.
Let $P$ be the standard Gaussian measure.
\Cref{lem:hclwe_smooth_map} implies that there exist constants $c, C$ such that the optimal transport map between $P$ and $Q_{u, \beta, \gamma}$ for any unit vector $u$ lies in $\cF_{\mathrm{spiked}, 1}(c, C)$.  We may assume without loss of generality that $c \leq 1 \leq C$, so that the function $\|x\|^2/2$ also lies in $\cF_{\mathrm{spiked}, 1}(c, C)$.

Suppose there exists an estimator $\hat \phi$ as in the statement of the theorem.
The optimal transport map between $P$ and $Q_0$ is the identity function, which is the gradient of $\|x\|^2/2 \in \cF_{\mathrm{spiked}, 1}(c, C)$, and by construction the optimal transport map between $P$ and $Q_{u, \beta, \gamma}$ is the gradient of a function that also lives in this class.
Finally, if $\phi_0$ is the potential whose gradient gives a map between $P$ and $Q_{u, \beta, \gamma}$, then
\begin{equation*}
	\| x - \nabla \phi_0\|_{L^2(P)} \geq W_2(Q_0, Q_{u, \beta, \gamma}) \geq \frac{8c'}{\sqrt{d}}\,.
\end{equation*}
Hence, a standard reduction based on Le Cam's argument~\cite[Section 2.2]{Tsy09} implies that any estimator of $\nabla \phi_0$ as in the statement of the theorem would give rise to a test distinguishing between samples from $Q_0$ and samples from $Q_{u, \beta, \gamma}$, for any unit vector $u$.
The aforementioned reduction then implies the existence of the claimed algorithms for CLWE and GapSVP, completing the proof.

\begin{proof}[Proof of \cref{lem:w2_far}]
	The Wasserstein distance between $Q_0$ and $Q_{u, \beta, \gamma}$ is bounded below by the Wasserstein distance of any of its one-dimensional marginals (since any coupling between $Q_0$ and $Q_{u, \beta, \gamma}$ give rise to a coupling between the one-dimensional marginals of each measure).
	Considering the projectsions of $Q_0$ and $Q_{u, \beta, \gamma}$ onto the subspace parallel to $u$, it therefore suffices to compare a standard Normal distribution on $\RR$ with the one-dimensional measure whose density is proportional to
	\begin{equation}\label{eq:one_dim_q_density}
		e^{-y^2/2}\cdot \sum_{k \in \ZZ} e^{- (k - \gamma y)^2/2\beta^2}
	\end{equation}
	This function can be rewritten
	\begin{equation}
		\sum_{k \in \ZZ} \alpha_k e^{- \frac{(\beta^2 + \gamma^2) \left(y - \frac{\gamma}{\beta^2 + \gamma^2} k \right)^2}{2 \beta^2}}\,, \quad \quad \alpha_k = e^{-\frac{k^2}{2(\beta^2 + \gamma^2)}}
	\end{equation}
	Therefore, if we write $\rho$ for the probability measure proportional to
	\begin{equation}
		\sum_{k \in \ZZ} \alpha_k \delta_{\frac{\gamma}{\beta^2 + \gamma^2} k}\,,
	\end{equation}
	then it suffices to compute $W_2^2(\cN(0, 1), \rho * \cN(0, \beta^2/(\beta^2 + \gamma^2)))$.
	
	The points in the support of $\rho$ are at distance $\frac{\gamma}{\beta^2 + \gamma^2}$ from each other, so the union of intervals of radius $\frac{\gamma}{4(\beta^2 + \gamma^2)}$ centered at the support of $\rho$ covers at most half the interval $[-1, 1]$.
	Since the density of $\cN(0, 1)$ is bounded below on $[-1, 1]$, this fact implies that a constant fraction of the mass of $\cN(0, 1)$ lies at distance at least $\frac{\gamma}{4(\beta^2 + \gamma^2)}$ from the support of $\rho$.
	Therefore there exists a positive constant $C$ such that
	\begin{equation*}
		W_2(\cN(0,1), \rho) \geq C \frac{\gamma}{(\beta^2 + \gamma^2)}\,.
	\end{equation*}
	Since $W_2(\rho * \cN(0, \beta^2/(\beta^2 + \gamma^2), \rho) \leq \beta/\sqrt{\beta^2 + \gamma^2}$, we obtain by the triangle inequality
	\begin{align*}
		W_2(\cN(0,1), \rho * \cN(0, \beta^2/(\beta^2 + \gamma^2)) & \geq C \frac{\gamma}{(\beta^2 + \gamma^2)} - \frac{\beta}{\sqrt{\beta^2 + \gamma^2}} \\
		& = \gamma^{-1} \left(\frac{C}{1+(\beta/\gamma)^2} - \frac{\beta}{\sqrt{1+(\beta/\gamma)^2}}\right)
	\end{align*}
	This quantity is at least a constant multiple of $\gamma^{-1}$ as long as $\beta$ is smaller than a universal constant, as claimed.
\end{proof}

\begin{proof}[Proof of \cref{lem:hclwe_smooth_map}]
	We first show that the density of $Q_{u, \beta, \gamma}$ is within a constant multiple of the density of $P$, with constants depending only on $\beta$.
	Let $f(y) =  \sum_{k \in \ZZ} e^{- (k - y)^2/2\beta^2}$.
	
	For any $y \in \RR$, we have
	\begin{equation}
		f(y) \geq e^{-(1/2)^2/2\beta^2} = e^{-1/8\beta^2}\,,
	\end{equation}
	which shows that $f$ is bounded below by a function of $\beta$.
	
	On the other hand, the Jacobi formula~\cite[Chapter 21.51]{WhiWat21} shows
	\begin{equation*}
		 f(y) = \sqrt{2 \pi} \beta \left(1 + 2 \sum_{k=1}^\infty e^{- 2 \beta^2 \pi^2 k^2} \cos 2 \pi k  y \right)
	\end{equation*}
	Since
	\begin{align*}
		\sum_{k=1}^\infty e^{- 2 \beta^2 \pi^2 k^2} \cos 2 \pi k y & \leq \sum_{k=1}^\infty e^{- 2 \beta^2 \pi^2 k^2} \\
		& \leq e^{-2 \beta^2\pi^2}/(1-e^{-2 \beta^2 \pi^2})\,,
	\end{align*}
	we obtain that $f$ is also bounded above by a function of $\beta$.
	
	If we denote the density of $Q_{u, \beta, \gamma}$ by $q_{u, \beta, \gamma}$ and the density of $P$ by $p$, then
	\begin{equation*}
		q_{u, \beta, \gamma}(x) = p(x) \cdot \frac{1}{Z_{\beta, \gamma}} f(\gamma \langle x, u \rangle)\,,
	\end{equation*}
	where $Z_{\beta, \gamma} = \int_{\RR^d} p(x) f(\gamma \langle x, u \rangle) \dd x$.
	Since we have established that $f$ is bounded above and below by constants depending only on $\beta$, the same holds for $Z_{\beta, \gamma}$.
	Therefore $q_{u, \beta, \gamma}(x) \asymp p(x)$ for all $x \in \RR^d$, where the implicit constant depends only on $\beta$.
	
	We now employ the following lemma, whose proof is deferred.
	\begin{lem}\label{lem:bounded_smooth_map}
		For any $C > c > 0$, there exists $C' > c' > 0$ such that the following holds:
		if $\rho$ is the standard Gaussian density on $\RR$ and $\mu$ is any probability density satisfying the  inequality $c \rho(x) \leq \mu(x) \leq C \rho(x)$ for all $x \in \RR$, then the optimal transport map $T$ between $\rho$ and $\mu$ satisfies $c' \leq T'(x) \leq C'$ for all $x \in \RR$.
	\end{lem}
	
	To conclude, it suffices to note that $Q_{u, \beta, \gamma}$ is the law of $Y u + (I-uu^\top) Z$, where $Y$ and $Z$ are independent random variables, $Z \sim P$, and $Y$ is a one-dimensional random variable with density $\frac{1}{\sqrt{2 \pi} Z_{\beta, \gamma}}e^{-y^2/2} f(y)$.
	Since $f$ and $Z_{\beta, \gamma}$ are bounded above and below, \cref{lem:bounded_smooth_map} implies that there exists an increasing function $\phi: \RR \to \RR$ such that $c' \leq \phi'' \leq C'$ and $\phi'_\sharp (\cN(0, 1)) = \mathrm{Law}(Y)$.
	
	Then $\phi_0(x) = \phi(\langle x, u \rangle) + \frac 12 \|x - \langle u, x \rangle x\|^2 \in \cF_{\mathrm{spiked}, 1}(c', C')$ and satisfies $(\nabla \phi_0)_\sharp P = Q_{u, \beta, \gamma}$, as desired.
\end{proof}
\begin{proof}[Proof of Lemma~\ref{lem:bounded_smooth_map}]
	Let $F_\rho(x) = \int_{-\infty}^x \rho(y) \dd y$ and $F_\mu(x) = \int_{-\infty}^x \mu(y) \dd y$.
	The optimal transport map between $\rho$ and $\mu$ is given by~\cite[Theorem 2.5]{San15}
	\begin{equation*}
		T(x) = F_\mu^{-1}(F_\rho(x))\,,
	\end{equation*}
	which implies that
	\begin{equation*}
		T'(x) = \frac{\rho(x)}{\mu(F_\mu^{-1}(F_\rho(x)))}\,.
	\end{equation*}
	
	For two functions $f$ and $g$ on $\RR$, we use the notation $f \asymp g$ to indicate that there exist positive constants (depending on $c$ and $C$) such that $c f(x) \leq g(x) \leq C f(x)$ for all $x \in \RR$.
	It therefore suffices to show that $\rho \asymp \mu(F_\mu^{-1}(F_\rho))$, or, equivalently, that
	\begin{equation*}
		\rho \circ F_\rho^{-1}(t) \asymp \mu \circ F_\mu^{-1}(t) \quad \forall t \in (0, 1)\,.
	\end{equation*}
	In the case of $\rho$, well known bounds on the Mills' ratio~\cite{Due10} imply
	\begin{equation*}
		\rho(x) \asymp (1+|x|) \min\{F_\rho(x), 1- F_\rho(x)\}
	\end{equation*}
	The assumption that $\rho \asymp \mu$ implies that $F_\rho \asymp F_\mu$ and $1 - F_\rho \asymp 1 - F_\mu$, which implies that
	\begin{equation*}
		\mu(x) \asymp (1+|x|) \min\{F_\mu(x), 1- F_\mu(x)\}
	\end{equation*}
	also holds.
	
	We therefore obtain that $\rho$ satisfies
	\begin{equation*}
		\rho  \circ F_\rho^{-1}(t) \asymp (1+ |F_\rho^{-1}(t)|) \min\{t, 1-t\} \quad \forall t\in (0, 1)
	\end{equation*}
	and the analogous bound for $\mu$.
	Hence, the desired claim is equivalent to
	\begin{equation*}
		1 + |F_\rho^{-1}(t)| \asymp 1 + |F_\mu^{-1}(t)|\,.
	\end{equation*}
	It suffices to verify this claim for $t \to 0$ and $t \to 1$, because when $t$ is bounded away from $0$ and $1$, both sides are of constant order.
	By symmetry, we focus without loss of generality on the case $t \to 0$.
	
	To conclude, we observe that since $F_\rho$ and $F_\mu$ are increasing, a bound of the form $c \rho \leq \mu \leq C \rho$ implies
	\begin{equation*}
		F_\rho^{-1}(t/C) \leq F_\mu^{-1}(t) \leq F_\rho^{-1}(t/c)\,.
	\end{equation*}
	Since $F_\rho(x) \asymp \frac{1}{(1+|x|)} e^{-x^2/2}$ for $x \leq 0$, we have
	\begin{equation*}
		\lim_{t \to 0} \frac{F_\rho^{-1}(t/c)}{F_\rho^{-1}(t/C)} = 1\,. 
	\end{equation*}
	Therefore, as $t \to 0$, $F_{\mu}^{-1}(t) \asymp  F_{\rho}^{-1}(t)$, proving the claim.
\end{proof}

\bibliography{ref}

\begin{thebibliography}{75}
\providecommand{\natexlab}[1]{#1}
\providecommand{\url}[1]{\texttt{#1}}
\expandafter\ifx\csname urlstyle\endcsname\relax
  \providecommand{\doi}[1]{doi: #1}\else
  \providecommand{\doi}{doi: \begingroup \urlstyle{rm}\Url}\fi

\bibitem[Adams and Fournier(2003)]{adams2003sobolev}
R.~A. Adams and J.~J. Fournier.
\newblock \emph{Sobolev spaces}.
\newblock Elsevier, 2003.

\bibitem[Arjovsky et~al.(2017)Arjovsky, Chintala, and Bottou]{WassersteinGAN}
M.~Arjovsky, S.~Chintala, and L.~Bottou.
\newblock {W}asserstein generative adversarial networks.
\newblock \emph{Proceedings of the 34th International Conference on Machine
  Learning}, 70:\penalty0 214--223, 2017.

\bibitem[Bach(2017)]{bach2017breaking}
F.~Bach.
\newblock Breaking the curse of dimensionality with convex neural networks.
\newblock \emph{The Journal of Machine Learning Research}, 18\penalty0
  (1):\penalty0 629--681, 2017.

\bibitem[Bakry et~al.(2008)Bakry, Barthe, Cattiaux, and
  Guillin]{bakry2008simple}
D.~Bakry, F.~Barthe, P.~Cattiaux, and A.~Guillin.
\newblock A simple proof of the poincar{\'e} inequality for a large class of
  probability measures.
\newblock \emph{Electronic Communications in Probability}, 13:\penalty0 60--66,
  2008.

\bibitem[Balabdaoui et~al.(2021)Balabdaoui, Doss, and
  Durot]{balabdaoui2021unlinked}
F.~Balabdaoui, C.~R. Doss, and C.~Durot.
\newblock Unlinked monotone regression.
\newblock \emph{The Journal of Machine Learning Research}, 22\penalty0
  (1):\penalty0 7766--7825, 2021.

\bibitem[Bandeira et~al.(2018)Bandeira, Perry, and Wein]{BanPerWei18}
A.~S. Bandeira, A.~Perry, and A.~S. Wein.
\newblock Notes on computational-to-statistical gaps: predictions using
  statistical physics.
\newblock \emph{Portugaliae mathematica}, 75\penalty0 (2):\penalty0 159--186,
  2018.

\bibitem[Barron(1993)]{barron}
A.~R. Barron.
\newblock Universal approximation bounds for superpositions of a sigmoidal
  function.
\newblock \emph{IEEE Trans. Inform. Theory}, 39\penalty0 (3):\penalty0
  930--945, 1993.
\newblock ISSN 0018-9448.

\bibitem[Beck(2017)]{beck2017first}
A.~Beck.
\newblock \emph{First-order methods in optimization}.
\newblock SIAM, 2017.

\bibitem[Belomestny et~al.(2021)Belomestny, Moulines, Naumov, Puchkin, and
  Samsonov]{belomestny2021rates}
D.~Belomestny, E.~Moulines, A.~Naumov, N.~Puchkin, and S.~Samsonov.
\newblock Rates of convergence for density estimation with generative
  adversarial networks.
\newblock \emph{arXiv e-prints}, pages arXiv--2102, 2021.

\bibitem[Belomestny et~al.(2023)Belomestny, Naumov, Puchkin, and
  Samsonov]{belomestny2023simultaneous}
D.~Belomestny, A.~Naumov, N.~Puchkin, and S.~Samsonov.
\newblock Simultaneous approximation of a smooth function and its derivatives
  by deep neural networks with piecewise-polynomial activations.
\newblock \emph{Neural Networks}, 161:\penalty0 242--253, 2023.

\bibitem[Bernstein et~al.(2009)Bernstein, Buchmann, and Dahmen]{PQC09}
D.~J. Bernstein, J.~Buchmann, and E.~Dahmen, editors.
\newblock \emph{Post-quantum cryptography}.
\newblock Springer-Verlag, Berlin, 2009.
\newblock ISBN 978-3-540-88701-0.
\newblock \doi{10.1007/978-3-540-88702-7}.
\newblock URL \url{https://doi.org/10.1007/978-3-540-88702-7}.

\bibitem[Berthet and Rigollet(2013)]{BerRig13}
Q.~Berthet and P.~Rigollet.
\newblock Optimal detection of sparse principal components in high dimension.
\newblock \emph{The Annals of Statistics}, 41\penalty0 (4):\penalty0 1780,
  2013.

\bibitem[Bobkov and Ledoux(1997)]{bobkov1997poincare}
S.~Bobkov and M.~Ledoux.
\newblock Poincar{\'e}'s inequalities and talagrand's concentration phenomenon
  for the exponential distribution.
\newblock \emph{Probability Theory and Related Fields}, 107\penalty0
  (3):\penalty0 383--400, 1997.

\bibitem[Bobkov(1999)]{bobkov1999isoperimetric}
S.~G. Bobkov.
\newblock Isoperimetric and analytic inequalities for log-concave probability
  measures.
\newblock \emph{The Annals of Probability}, 27\penalty0 (4):\penalty0
  1903--1921, 1999.

\bibitem[Brennan and Bresler(2020)]{BreBre20}
M.~Brennan and G.~Bresler.
\newblock Reducibility and statistical-computational gaps from secret leakage.
\newblock In \emph{Conference on Learning Theory}, pages 648--847. PMLR, 2020.

\bibitem[Brennan et~al.(2018)Brennan, Bresler, and Huleihel]{BreBreHul18}
M.~Brennan, G.~Bresler, and W.~Huleihel.
\newblock Reducibility and computational lower bounds for problems with planted
  sparse structure.
\newblock In \emph{Conference On Learning Theory}, pages 48--166. PMLR, 2018.

\bibitem[Bruna et~al.(2021)Bruna, Regev, Song, and Tang]{BruRegSon21}
J.~Bruna, O.~Regev, M.~J. Song, and Y.~Tang.
\newblock Continuous {LWE}.
\newblock In \emph{S{TOC} '21---{P}roceedings of the 53rd {A}nnual {ACM}
  {SIGACT} {S}ymposium on {T}heory of {C}omputing}, pages 694--707. ACM, New
  York, 2021.

\bibitem[Bubeck(2015)]{Bub15}
S.~Bubeck.
\newblock Convex optimization: Algorithms and complexity.
\newblock \emph{Foundations and Trends{\textregistered} in Machine Learning},
  8\penalty0 (3-4):\penalty0 231--357, 2015.
\newblock \doi{10.1561/2200000050}.
\newblock URL \url{https://doi.org/10.1561%2F2200000050}.

\bibitem[Bunne et~al.(2022{\natexlab{a}})Bunne, Krause, and
  Cuturi]{bunne2022supervised}
C.~Bunne, A.~Krause, and M.~Cuturi.
\newblock Supervised training of conditional monge maps.
\newblock \emph{arXiv preprint arXiv:2206.14262}, 2022{\natexlab{a}}.

\bibitem[Bunne et~al.(2022{\natexlab{b}})Bunne, Papaxanthos, Krause, and
  Cuturi]{bunne2022proximal}
C.~Bunne, L.~Papaxanthos, A.~Krause, and M.~Cuturi.
\newblock Proximal optimal transport modeling of population dynamics.
\newblock In \emph{International Conference on Artificial Intelligence and
  Statistics}, pages 6511--6528. PMLR, 2022{\natexlab{b}}.

\bibitem[Caffarelli(2000)]{caffarelli2000monotonicity}
L.~A. Caffarelli.
\newblock Monotonicity properties of optimal transportation and the fkg and
  related inequalities.
\newblock \emph{Communications in Mathematical Physics}, 214\penalty0
  (3):\penalty0 547--563, 2000.

\bibitem[Carlier et~al.(2016)Carlier, Chernozhukov, and
  Galichon]{carlier2016vector}
G.~Carlier, V.~Chernozhukov, and A.~Galichon.
\newblock Vector quantile regression: an optimal transport approach.
\newblock \emph{The Annals of Statistics}, 44\penalty0 (3):\penalty0
  1165--1192, 2016.

\bibitem[Chernozhukov et~al.(2017)Chernozhukov, Galichon, Hallin, and
  Henry]{chernozhukov2017monge}
V.~Chernozhukov, A.~Galichon, M.~Hallin, and M.~Henry.
\newblock Monge--kantorovich depth, quantiles, ranks and signs.
\newblock \emph{The Annals of Statistics}, 45\penalty0 (1):\penalty0 223--256,
  2017.

\bibitem[Chewi and Pooladian(2022)]{chewi2022entropic}
S.~Chewi and A.-A. Pooladian.
\newblock An entropic generalization of caffarelli's contraction theorem via
  covariance inequalities.
\newblock \emph{arXiv preprint arXiv:2203.04954}, 2022.

\bibitem[Deb et~al.(2021)Deb, Ghosal, and Sen]{deb2021rates}
N.~Deb, P.~Ghosal, and B.~Sen.
\newblock Rates of estimation of optimal transport maps using plug-in
  estimators via barycentric projections.
\newblock \emph{arXiv preprint arXiv:2107.01718}, 2021.

\bibitem[DeGroot and Goel(1980)]{degroot1980estimation}
M.~H. DeGroot and P.~K. Goel.
\newblock Estimation of the correlation coefficient from a broken random
  sample.
\newblock \emph{The Annals of Statistics}, pages 264--278, 1980.

\bibitem[Demetci et~al.(2021)Demetci, Santorella, Sandstede, and
  Singh]{Demetci2021.SCOTv2}
P.~Demetci, R.~Santorella, B.~Sandstede, and R.~Singh.
\newblock Unsupervised integration of single-cell multi-omics datasets with
  disparities in cell-type representation.
\newblock \emph{bioRxiv}, 2021.

\bibitem[Duembgen(2010)]{Due10}
L.~Duembgen.
\newblock Bounding standard gaussian tail probabilities.
\newblock 12 2010.
\newblock URL \url{https://arxiv.org/pdf/1012.2063.pdf}.

\bibitem[E et~al.(2022)E, Ma, and Wu]{e_barron}
W.~E, C.~Ma, and L.~Wu.
\newblock The {B}arron space and the flow-induced function spaces for neural
  network models.
\newblock \emph{Constr. Approx.}, 55\penalty0 (1):\penalty0 369--406, 2022.
\newblock ISSN 0176-4276.

\bibitem[Feydy et~al.(2017)Feydy, Charlier, Vialard, and
  Peyr{\'e}]{feydy2017optimal}
J.~Feydy, B.~Charlier, F.-X. Vialard, and G.~Peyr{\'e}.
\newblock Optimal transport for diffeomorphic registration.
\newblock In \emph{International Conference on Medical Image Computing and
  Computer-Assisted Intervention}, pages 291--299. Springer, 2017.

\bibitem[Finlay et~al.(2020)Finlay, Gerolin, Oberman, and
  Pooladian]{finlay2020learning}
C.~Finlay, A.~Gerolin, A.~M. Oberman, and A.-A. Pooladian.
\newblock Learning normalizing flows from entropy-kantorovich potentials.
\newblock \emph{arXiv preprint arXiv:2006.06033}, 2020.

\bibitem[Flamary et~al.(2019)Flamary, Lounici, and
  Ferrari]{flamary2019concentration}
R.~Flamary, K.~Lounici, and A.~Ferrari.
\newblock Concentration bounds for linear monge mapping estimation and optimal
  transport domain adaptation.
\newblock \emph{arXiv preprint arXiv:1905.10155}, 2019.

\bibitem[Genevay et~al.(2018)Genevay, Peyr{\'e}, and
  Cuturi]{2017-Genevay-AutoDiff}
A.~Genevay, G.~Peyr{\'e}, and M.~Cuturi.
\newblock Learning generative models with {Sinkhorn} divergences.
\newblock In \emph{Proceedings of the 21st International Conference on
  Artificial Intelligence and Statistics}, pages 1608--1617, 2018.

\bibitem[Gin{\'e} and Nickl(2021)]{gine2021mathematical}
E.~Gin{\'e} and R.~Nickl.
\newblock \emph{Mathematical foundations of infinite-dimensional statistical
  models}.
\newblock Cambridge University Press, 2021.

\bibitem[Gozlan(2010)]{gozlan2010poincare}
N.~Gozlan.
\newblock Poincar{\'e} inequalities and dimension free concentration of
  measure.
\newblock In \emph{Annales de l'IHP Probabilit{\'e}s et statistiques},
  volume~46, pages 708--739, 2010.

\bibitem[Grathwohl et~al.(2018)Grathwohl, Chen, Bettencourt, Sutskever, and
  Duvenaud]{grathwohl2018ffjord}
W.~Grathwohl, R.~T. Chen, J.~Bettencourt, I.~Sutskever, and D.~Duvenaud.
\newblock Ffjord: Free-form continuous dynamics for scalable reversible
  generative models.
\newblock \emph{arXiv preprint arXiv:1810.01367}, 2018.

\bibitem[Gunsilius and Xu(2021)]{gunsilius2021matching}
F.~Gunsilius and Y.~Xu.
\newblock Matching for causal effects via multimarginal optimal transport.
\newblock \emph{arXiv preprint arXiv:2112.04398}, 2021.

\bibitem[Haroske and Triebel(1994)]{triebelcovering}
D.~Haroske and H.~Triebel.
\newblock Entropy numbers in weighted function spaces and eigenvalue
  distributions of some degenerate pseudodifferential operators i.
\newblock \emph{Mathematische Nachrichten}, 167\penalty0 (1):\penalty0
  131--156, 1994.
\newblock \doi{https://doi.org/10.1002/mana.19941670107}.
\newblock URL
  \url{https://onlinelibrary.wiley.com/doi/abs/10.1002/mana.19941670107}.

\bibitem[Heinonen et~al.(2015)Heinonen, Koskela, Shanmugalingam, and
  Tyson]{heinonen2015sobolev}
J.~Heinonen, P.~Koskela, N.~Shanmugalingam, and J.~T. Tyson.
\newblock \emph{Sobolev spaces on metric measure spaces}.
\newblock Number~27. cambridge university press, 2015.

\bibitem[Huang et~al.(2021)Huang, Chen, Tsirigotis, and
  Courville]{huang2021convex}
C.-W. Huang, R.~T.~Q. Chen, C.~Tsirigotis, and A.~Courville.
\newblock Convex potential flows: Universal probability distributions with
  optimal transport and convex optimization.
\newblock In \emph{International Conference on Learning Representations}, 2021.

\bibitem[H{\"u}tter and Rigollet(2021)]{hutter2021minimax}
J.-C. H{\"u}tter and P.~Rigollet.
\newblock Minimax estimation of smooth optimal transport maps.
\newblock \emph{The Annals of Statistics}, 49\penalty0 (2):\penalty0
  1166--1194, 2021.

\bibitem[Lucet(1997)]{lucet1997faster}
Y.~Lucet.
\newblock Faster than the fast legendre transform, the linear-time legendre
  transform.
\newblock \emph{Numerical Algorithms}, 16:\penalty0 171--185, 1997.

\bibitem[Makkuva et~al.(2020)Makkuva, Taghvaei, Oh, and
  Lee]{makkuva2020optimal}
A.~Makkuva, A.~Taghvaei, S.~Oh, and J.~Lee.
\newblock Optimal transport mapping via input convex neural networks.
\newblock In \emph{International Conference on Machine Learning}, pages
  6672--6681. PMLR, 2020.

\bibitem[Manole et~al.(2021)Manole, Balakrishnan, Niles-Weed, and
  Wasserman]{manole2021plugin}
T.~Manole, S.~Balakrishnan, J.~Niles-Weed, and L.~Wasserman.
\newblock Plugin estimation of smooth optimal transport maps.
\newblock \emph{arXiv preprint arXiv:2107.12364}, 2021.

\bibitem[Massart(2007)]{massart2007concentration}
P.~Massart.
\newblock \emph{Concentration inequalities and model selection: Ecole d'Et{\'e}
  de Probabilit{\'e}s de Saint-Flour XXXIII-2003}.
\newblock Springer, 2007.

\bibitem[Milgrom and Segal(2002)]{MilSeg02}
P.~Milgrom and I.~Segal.
\newblock Envelope theorems for arbitrary choice sets.
\newblock \emph{Econometrica}, 70\penalty0 (2):\penalty0 583--601, 2002.
\newblock ISSN 0012-9682.
\newblock \doi{10.1111/1468-0262.00296}.
\newblock URL \url{https://doi.org/10.1111/1468-0262.00296}.

\bibitem[Moriel et~al.(2021)Moriel, Senel, Friedman, Rajewsky, Karaiskos, and
  Nitzan]{moriel2021novosparc}
N.~Moriel, E.~Senel, N.~Friedman, N.~Rajewsky, N.~Karaiskos, and M.~Nitzan.
\newblock Novosparc: flexible spatial reconstruction of single-cell gene
  expression with optimal transport.
\newblock \emph{Nature Protocols}, 16\penalty0 (9):\penalty0 4177--4200, 2021.

\bibitem[Muzellec et~al.(2021)Muzellec, Vacher, Bach, Vialard, and
  Rudi]{muzellec2021near}
B.~Muzellec, A.~Vacher, F.~Bach, F.-X. Vialard, and A.~Rudi.
\newblock Near-optimal estimation of smooth transport maps with kernel
  sums-of-squares.
\newblock \emph{arXiv preprint arXiv:2112.01907}, 2021.

\bibitem[Niles-Weed and Rigollet(2022)]{niles2022estimation}
J.~Niles-Weed and P.~Rigollet.
\newblock Estimation of wasserstein distances in the spiked transport model.
\newblock \emph{Bernoulli}, 28\penalty0 (4):\penalty0 2663--2688, 2022.

\bibitem[Pananjady et~al.(2016)Pananjady, Wainwright, and
  Courtade]{pananjady2016linear}
A.~Pananjady, M.~J. Wainwright, and T.~A. Courtade.
\newblock Linear regression with an unknown permutation: Statistical and
  computational limits.
\newblock In \emph{2016 54th Annual Allerton Conference on Communication,
  Control, and Computing (Allerton)}, pages 417--424. IEEE, 2016.

\bibitem[Pooladian and Niles-Weed(2021)]{pooladian2021entropic}
A.-A. Pooladian and J.~Niles-Weed.
\newblock Entropic estimation of optimal transport maps.
\newblock \emph{arXiv preprint arXiv:2109.12004}, 2021.

\bibitem[Regev(2005)]{Reg05}
O.~Regev.
\newblock On lattices, learning with errors, random linear codes, and
  cryptography.
\newblock In \emph{S{TOC}'05: {P}roceedings of the 37th {A}nnual {ACM}
  {S}ymposium on {T}heory of {C}omputing}, pages 84--93. ACM, New York, 2005.
\newblock \doi{10.1145/1060590.1060603}.
\newblock URL \url{https://doi.org/10.1145/1060590.1060603}.

\bibitem[Regev(2010)]{Reg10}
O.~Regev.
\newblock The learning with errors problem.
\newblock In \emph{25th {A}nnual {IEEE} {C}onference on {C}omputational
  {C}omplexity---{CCC} 2010}, pages 191--204. IEEE Computer Soc., Los Alamitos,
  CA, 2010.

\bibitem[Rigollet and Weed(2019)]{rigollet2019uncoupled}
P.~Rigollet and J.~Weed.
\newblock Uncoupled isotonic regression via minimum wasserstein deconvolution.
\newblock \emph{Information and Inference: A Journal of the IMA}, 8\penalty0
  (4):\penalty0 691--717, 2019.

\bibitem[Salimans et~al.(2018)Salimans, Zhang, Radford, and
  Metaxas]{salimans2018improving}
T.~Salimans, H.~Zhang, A.~Radford, and D.~Metaxas.
\newblock Improving {GAN}s using optimal transport.
\newblock In \emph{International Conference on Learning Representations}, 2018.

\bibitem[Santambrogio(2015)]{San15}
F.~Santambrogio.
\newblock \emph{Optimal transport for applied mathematicians}, volume~87 of
  \emph{Progress in Nonlinear Differential Equations and their Applications}.
\newblock Birkh\"{a}user/Springer, Cham, 2015.
\newblock ISBN 978-3-319-20827-5; 978-3-319-20828-2.
\newblock \doi{10.1007/978-3-319-20828-2}.
\newblock URL \url{https://doi.org/10.1007/978-3-319-20828-2}.
\newblock Calculus of variations, PDEs, and modeling.

\bibitem[Schiebinger et~al.(2019)Schiebinger, Shu, Tabaka, Cleary, Subramanian,
  Solomon, Gould, Liu, Lin, Berube, et~al.]{schiebinger2019optimal}
G.~Schiebinger, J.~Shu, M.~Tabaka, B.~Cleary, V.~Subramanian, A.~Solomon,
  J.~Gould, S.~Liu, S.~Lin, P.~Berube, et~al.
\newblock Optimal-transport analysis of single-cell gene expression identifies
  developmental trajectories in reprogramming.
\newblock \emph{Cell}, 176\penalty0 (4):\penalty0 928--943, 2019.

\bibitem[Sch{\"o}lkopf et~al.(2002)Sch{\"o}lkopf, Smola, Bach,
  et~al.]{scholkopf2002learning}
B.~Sch{\"o}lkopf, A.~J. Smola, F.~Bach, et~al.
\newblock \emph{Learning with kernels: support vector machines, regularization,
  optimization, and beyond}.
\newblock MIT press, 2002.

\bibitem[Sch{\"u}tt(1984)]{schutt1984entropy}
C.~Sch{\"u}tt.
\newblock Entropy numbers of diagonal operators between symmetric banach
  spaces.
\newblock \emph{Journal of approximation theory}, 40\penalty0 (2):\penalty0
  121--128, 1984.

\bibitem[Siegel and Xu(2021)]{siegel2021sharp}
J.~W. Siegel and J.~Xu.
\newblock Sharp bounds on the approximation rates, metric entropy, and $ n
  $-widths of shallow neural networks.
\newblock \emph{arXiv preprint arXiv:2101.12365}, 2021.

\bibitem[Slawski and Sen(2022)]{slawski2022permuted}
M.~Slawski and B.~Sen.
\newblock Permuted and unlinked monotone regression in $\mathbb{R}^d$: an
  approach based on mixture modeling and optimal transport.
\newblock \emph{arXiv preprint arXiv:2201.03528}, 2022.

\bibitem[Solomon et~al.(2015)Solomon, De~Goes, Peyr{\'e}, Cuturi, Butscher,
  Nguyen, Du, and Guibas]{SolGoePey15}
J.~Solomon, F.~De~Goes, G.~Peyr{\'e}, M.~Cuturi, A.~Butscher, A.~Nguyen, T.~Du,
  and L.~Guibas.
\newblock Convolutional {W}asserstein distances: {E}fficient optimal
  transportation on geometric domains.
\newblock \emph{ACM Transactions on Graphics (TOG)}, 34\penalty0 (4):\penalty0
  66, 2015.

\bibitem[Solomon et~al.(2016)Solomon, Peyr{\'{e}}, Kim, and Sra]{SolPeyKim16}
J.~Solomon, G.~Peyr{\'{e}}, V.~G. Kim, and S.~Sra.
\newblock Entropic metric alignment for correspondence problems.
\newblock \emph{{ACM} Trans. Graph.}, 35\penalty0 (4):\penalty0 72:1--72:13,
  2016.

\bibitem[Torous et~al.(2021)Torous, Gunsilius, and Rigollet]{torous2021optimal}
W.~Torous, F.~Gunsilius, and P.~Rigollet.
\newblock An optimal transport approach to causal inference.
\newblock \emph{arXiv preprint arXiv:2108.05858}, 2021.

\bibitem[Tsybakov(2009)]{Tsy09}
A.~B. Tsybakov.
\newblock \emph{Introduction to nonparametric estimation}.
\newblock Springer Series in Statistics. Springer, New York, 2009.
\newblock ISBN 978-0-387-79051-0.
\newblock \doi{10.1007/b13794}.
\newblock URL \url{https://doi.org/10.1007/b13794}.
\newblock Revised and extended from the 2004 French original, Translated by
  Vladimir Zaiats.

\bibitem[Vaart and Wellner(2023)]{vaart2023empirical}
A.~v.~d. Vaart and J.~A. Wellner.
\newblock Empirical processes.
\newblock In \emph{Weak Convergence and Empirical Processes: With Applications
  to Statistics}, pages 127--384. Springer, 2023.

\bibitem[Vacher and Vialard(2021)]{vacher2021convex}
A.~Vacher and F.-X. Vialard.
\newblock Convex transport potential selection with semi-dual criterion.
\newblock \emph{arXiv preprint arXiv:2112.07275}, 2021.

\bibitem[van~de Geer(2002)]{van2002m}
S.~van~de Geer.
\newblock M-estimation using penalties or sieves.
\newblock \emph{Journal of Statistical Planning and Inference}, 108\penalty0
  (1-2):\penalty0 55--69, 2002.

\bibitem[Vershynin(2018)]{vershynin2018high}
R.~Vershynin.
\newblock \emph{High-dimensional probability: An introduction with applications
  in data science}, volume~47.
\newblock Cambridge university press, 2018.

\bibitem[Villani(2009)]{villani2009optimal}
C.~Villani.
\newblock \emph{Optimal transport: old and new}, volume 338.
\newblock Springer, 2009.

\bibitem[Wainwright(2019)]{wainwright2019high}
M.~J. Wainwright.
\newblock \emph{High-dimensional statistics: A non-asymptotic viewpoint},
  volume~48.
\newblock Cambridge University Press, 2019.

\bibitem[Whittaker and Watson(2021)]{WhiWat21}
E.~T. Whittaker and G.~N. Watson.
\newblock \emph{A course of modern analysis---an introduction to the general
  theory of infinite processes and of analytic functions with an account of the
  principal transcendental functions}.
\newblock Cambridge University Press, Cambridge, fifth edition, 2021.
\newblock ISBN 978-1-316-51893-9.

\bibitem[Yang et~al.(2020{\natexlab{a}})Yang, Damodaran, Venkatachalapathy,
  Soylemezoglu, Shivashankar, and Uhler]{dai2018autoencoder}
K.~D. Yang, K.~Damodaran, S.~Venkatachalapathy, A.~C. Soylemezoglu,
  G.~Shivashankar, and C.~Uhler.
\newblock Predicting cell lineages using autoencoders and optimal transport.
\newblock \emph{PLoS computational biology}, 16\penalty0 (4):\penalty0
  e1007828, 2020{\natexlab{a}}.

\bibitem[Yang et~al.(2020{\natexlab{b}})Yang, Jin, Wang, Wang, and
  Jordan]{yang2020function}
Z.~Yang, C.~Jin, Z.~Wang, M.~Wang, and M.~I. Jordan.
\newblock On function approximation in reinforcement learning: optimism in the
  face of large state spaces.
\newblock In \emph{Proceedings of the 34th International Conference on Neural
  Information Processing Systems}, pages 13903--13916, 2020{\natexlab{b}}.

\bibitem[Zhou(2008)]{zhou2008derivative}
D.-X. Zhou.
\newblock Derivative reproducing properties for kernel methods in learning
  theory.
\newblock \emph{Journal of computational and Applied Mathematics}, 220\penalty0
  (1-2):\penalty0 456--463, 2008.

\end{thebibliography}

\end{document}